\documentclass[twoside]{article}

\usepackage[accepted]{aistats2021}
\usepackage[round]{natbib}

\usepackage{fullpage}
\usepackage{amsmath, amsthm, amssymb}
\usepackage{graphicx} 
\usepackage{tikz} 
    \usetikzlibrary{quotes,angles} 
    \usetikzlibrary{decorations.pathreplacing} 
\usepackage{enumitem} 
\usepackage{nicefrac} 
\usepackage[
    algo2e, 
    ruled, 
    linesnumbered, 
    noline, 
    noend, 
    ]{algorithm2e} 
\usepackage{authblk} 
\usepackage{mleftright}
    \mleftright
\usepackage{hyperref} 
\usepackage{lipsum} 
\usepackage{ifthen} 
\usepackage{xcolor} 

\newcommand{\cA}{\mathcal{A}}

\newcommand{\cC}{\mathcal{C}}

\newcommand{\cF}{\mathcal{F}}

\newcommand{\cL}{\mathcal{L}}
\newcommand{\cM}{\mathcal{M}}
\newcommand{\cN}{\mathcal{N}}

\newcommand{\cP}{\mathcal{P}}

\newcommand{\cS}{\mathcal{S}}

\newcommand{\bbI}{\mathbb{I}}

\newcommand{\bbN}{\mathbb{N^*}}

\newcommand{\bbR}{\mathbb{R}}

\newcommand{\bbZ}{\mathbb{Z}}
\newcommand{\rmA}{A}
\newcommand{\rmB}{B}
\newcommand{\rmC}{C}
\newcommand{\rmD}{D}
\newcommand{\rmE}{E}
\newcommand{\rmF}{F}
\newcommand{\rmG}{G}
\newcommand{\rmH}{H}

\newcommand{\rmL}{L}

\newcommand{\rmN}{N}

\newcommand{\rmP}{P}

\newcommand{\rmS}{S}
\newcommand{\rmT}{T}

\newcommand{\rmV}{V}

\newcommand{\rmZ}{Z}
\newcommand{\rma}{a}
\newcommand{\rmb}{b}
\newcommand{\rmc}{c}
\newcommand{\rmd}{d}

\newcommand{\rmf}{f}
\newcommand{\rmg}{g}
\newcommand{\rmh}{h}
\newcommand{\rmi}{i}
\newcommand{\rmj}{j}
\newcommand{\rmk}{k}

\newcommand{\rmm}{m}
\newcommand{\rmn}{n}

\newcommand{\rmp}{p}

\newcommand{\rmr}{r}
\newcommand{\rms}{s}
\newcommand{\rmt}{t}
\newcommand{\rmu}{u}
\newcommand{\rmv}{v}
\newcommand{\rmw}{w}
\newcommand{\rmx}{x}
\newcommand{\rmy}{y}
\newcommand{\rmz}{z}
\newcommand{\bc}{\boldsymbol{c}}
\newcommand{\be}{\boldsymbol{e}}
\newcommand{\bo}{\boldsymbol{o}}
\newcommand{\bp}{\boldsymbol{p}}
\newcommand{\bu}{\boldsymbol{u}}
\newcommand{\bv}{\boldsymbol{v}}
\newcommand{\bw}{\boldsymbol{w}}
\newcommand{\bx}{\boldsymbol{x}}
\newcommand{\bxs}{\boldsymbol{x}^\star}

\newcommand{\by}{\boldsymbol{y}}
\newcommand{\bz}{\boldsymbol{z}}
\newcommand{\bhalf}{\boldsymbol{\frac{1}{2}}}
\newcommand{\bzero}{\boldsymbol{0}}
\newcommand{\e}{\varepsilon}
\newcommand{\fhi}{\varphi}

\newcommand{\lrb}[1]{\left(#1\right)}
\newcommand{\brb}[1]{\bigl(#1\bigr)}
\newcommand{\Brb}[1]{\Bigl(#1\Bigr)}
\newcommand{\bbrb}[1]{\biggl(#1\biggr)}
\newcommand{\lsb}[1]{\left[#1\right]}
\newcommand{\bsb}[1]{\bigl[#1\bigr]}

\newcommand{\lcb}[1]{\left\{#1\right\}}
\newcommand{\bcb}[1]{\bigl\{#1\bigr\}}
\newcommand{\Bcb}[1]{\Bigl\{#1\Bigr\}}
\newcommand{\lce}[1]{\left\lceil#1\right\rceil}
\newcommand{\bce}[1]{\bigl\lceil#1\bigr\rceil}

\newcommand{\lfl}[1]{\left\lfloor#1\right\rfloor}

\newcommand{\labs}[1]{\left\lvert#1\right\rvert}
\newcommand{\babs}[1]{\bigl\lvert#1\bigr\rvert}
\newcommand{\Babs}[1]{\Bigl\lvert#1\Bigr\rvert}
\newcommand{\lno}[1]{\left\lVert#1\right\rVert}
\newcommand{\bno}[1]{\bigl\lVert#1\bigr\rVert}
\newcommand{\Bno}[1]{\Bigl\lVert#1\Bigr\rVert}
\newcommand{\lan}[1]{\left\langle#1\right\rangle}
\newcommand{\ban}[1]{\bigl\langle#1\bigr\rangle}

\newcommand{\dcube}{[0,1]^d}
\newcommand{\tc}[1]{\textcolor{red}{TC: #1}}
\newcommand{\fb}[1]{\textcolor{purple}{FB: #1}}

\DeclareMathOperator*{\argmin}{arg\,min}

\newcommand{\bisect}[1]{\mathrm{bisect}\lrb{#1}}
\newcommand{\dif}{\,\mathrm{d}}

\newcommand{\level}[2]{\mathrm{level}\lrb{#1,#2}}
\newcommand{\lip}{Lipschitz}
\newcommand{\hold}{H\"older}
\newcommand{\gradlip}{gradient-\lip{}}
\newcommand{\gradho}{gra\-di\-ent-\hold{}}
\newcommand{\Gradho}{Gradient-\hold{}}

\newcommand{\mink}{Minkowski}

\newcommand{\findo}{Find-$\bo$}
\newcommand{\shooto}{Shoot-from-$\bo$}
\newcommand{\bai}{Bisect and Approximate}
\newcommand{\bi}{BA}

\newcommand{\bih}{BAH}

\newcommand{\biga}{BAG}
\newcommand{\nearo}{rate-optimal}
\newcommand{\Nearo}{Rate-optimal}

\newcommand{\interp}{approximator}

\newcommand{\tol}{tolerance}
\newcommand{\toll}{tol.}
\newcommand{\hyperc}{hypercube}

\newcommand{\proper}{proper}
\newcommand{\nls}{NLS}
\newcommand{\nlss}{Near-Level-Set}
\newcommand{\suppl}{Supplementary Material}

\newcommand{\wt}{\widetilde}
\newcommand{\s}{\subseteq}
\newcommand{\m}{\setminus}
\newcommand{\iop}{\infty}
\newcommand{\ld}{\ldots}
\newcommand{\norm}[1]{\left\lVert#1\right\rVert}

\newcommand{\ts}{\rmt^\star}

\newcommand{\fd}{\ell}
\newcommand{\fn}{\mathfrak{n}}
\newcommand{\Sn}{\rmS_\rmn}
\newcommand{\fa}{\{\rmf = \rma\}}
\newcommand{\fae}{\bcb{ \labs{ \rmf - \rma} \le \e } }

\newcommand{\propp}[1]{(\mathrm{P}_{#1})}
\newcommand{\ds}{\rmd^\star}
\newcommand{\Cs}{\rmC^\star}

\newcommand{\sdm}{\cS^{\rmd-1}}
\newcommand{\fz}{\rmf_{\bz}}
\newcommand{\oo}{\bo_1}
\newcommand{\eqdef}{\stackrel{\mbox{\scriptsize \rmfamily def}}{=}}

\newtheorem{lemma}{Lemma}
\newtheorem*{lemma*}{Lemma}
\newtheorem{proposition}{Proposition}
\newtheorem{claim}{Claim}
\newtheorem{corollary}{Corollary}
\newtheorem*{corollary*}{Corollary}
\newtheorem{theorem}{Theorem}
\newtheorem*{theorem*}{Theorem}
    \theoremstyle{definition}
\newtheorem{assumption}{Assumption}
\newtheorem{definition}{Definition}

\begin{document}


%

%

\twocolumn[

\aistatstitle{The Sample Complexity of Level Set Approximation}

\aistatsauthor{ Fran\c{c}ois Bachoc \And Tommaso R. Cesari \And  S\'ebastien Gerchinovitz }

\aistatsaddress{University Paul Sabatier \\ Institut de Mathématiques de Toulouse   \And  Toulouse School of Economics \And IRT Saint Exup\'ery \\Institut de Mathématiques de Toulouse } ]


\begin{abstract}
We study the problem of approximating the level set of an unknown function by sequentially querying its values.
We introduce a family of algorithms called \bai{} through which we reduce the level set approximation problem to a local function approximation problem.
We then show how this approach leads to rate-optimal sample complexity guarantees for \hold{} functions, and we investigate how such rates improve when additional smoothness or other structural assumptions hold true.
\end{abstract}


\section{INTRODUCTION}

Let $f\colon [0,1]^d \to \bbR$ be any function. For $a \in \bbR$, we consider the problem of finding the level set 
\[
\{f=a\} \eqdef \bcb{ \bx \in \dcube : \rmf(\bx) = \rma } \;.
\]

\paragraph{Setting: Sequential Black-Box Evaluation.}
We study the case in which $f$ is black-box, i.e., except for some \emph{a priori} knowledge on its smoothness, we can only access $f$ by sequentially querying its values at a sequence $\bx_1,\bx_2,\ldots \in [0,1]^d$ of points of our choice (Online Protocol~\ref{alg:protocol-det-algo}).
At every round $n \ge 1$, the query point $\bx_n$ can be chosen as a deterministic function of the values $f(\bx_1),\ldots,f(\bx_{n-1})$ observed so far. 
At the end of round $n$, the learner outputs a subset $S_n$ of $[0,1]^d$ with the goal of approximating the level set $\{f=a\}$.
{
\renewcommand*{\algorithmcfname}{Online Protocol}
\begin{algorithm2e}
\DontPrintSemicolon
\For
{%
    $\rmn = 1, 2, \ld$
}
{
    pick the next query point $\bx_\rmn \in \dcube $\nllabel{a:protocol-pick}\;
    observe the value $\rmf(\bx_\rmn)$\nllabel{a:protocol-observe}\;
    output an approximating set $\rmS_\rmn \s \dcube$\nllabel{a:protocol-output}\;
}
\caption{\label{alg:protocol-det-algo} Deterministic Scheme}
\end{algorithm2e}
}

The problem of identifying the level set $\{f = a \}$ of a black-box function arises often in practice. 
In particular, this problem is closely related to excursion set estimation (also called failure domain estimation), where the goal is to estimate $\{ f \geq a \}$.\footnote{As it will become apparent later, our techniques for estimating level sets can be adapted for sub/superlevel set approximation straightforwardly, whilst retaining the same sample complexity guarantees (see Footnote~\ref{ft:sublevel}).}
Level-set identification and failure domain estimation are relevant to the field of computer experiments and uncertainty quantification, where $f(x)$ provides the output of a complex computer model for some input parameter $x$ \citep{DACE,TDACE}. Typical fields of applications are nuclear engineering \citep{chevalier14}, coastal flooding \citep{azzimonti2020adaptive} and network systems \citep{Ranjan2008}. Level set identification is also relevant when $f(x)$ corresponds to natural data \citep{rahimi2004adaptive,galland2004synthetic}. In many such situations, $f$ is so complex that it is considered black-box.

A typical example of a real use-case of interest is the Bombardier research aircraft configuration \citep{priem2020efficient}. Here, geometry parameters of an aircraft wing can be selected. Any choice of these parameters yields a corresponding maximum take-off weight output, which is obtained by a costly computational fluid dynamics simulation. From the setting of \citet{priem2020efficient}, one could for instance tackle the problem of estimating the set of all $\bx \in [0,1]^4$, where $\bx$ corresponds to the variables {\it wing span}, {\it wing leading edge sweep}, {\it wing break location} and {\it wingtip chord} (see Table 3 in \citealt{priem2020efficient}), for which $f(\bx) = a$ for some prescribed value $a>0$ of the maximum take-off weight. Furthermore, by setting more or less input parameters as active or inactive in Table 3 in \citet{priem2020efficient}, a series of level set estimation problems can be obtained, from dimension~$1$ to dimension~$18$.

\paragraph{Learning Goal.} There exist several ways to compare the estimators $S_n$ and the level set $\{f=a\}$. 
A first possibility is to use metrics or pseudometrics $\rho(A,B)$ between sets $A,B \s [0,1]^d$, such as the Hausdorff distance or the volume of the symmetric difference (e.g., \citealt{tsybakov1997nonparametric}). However a small value of $\rho(S_n,\{f=a\})$ does not imply that $S_n$ contains the whole set $\{f=a\}$, nor---in the case of the volume of the symmetric difference---that $f(x) \approx a$ for all $\bx \in S_n$. In practice, we might fail to identify \emph{all} critical states of a given system, or raise unnecessary false alarms.

In this paper, we therefore consider an alternative (new) way of quantifying our performance. For any accuracy $\e>0$, denote by $$\fae \eqdef \bcb{ \bx \in \dcube : \babs{ \rmf(\bx) - \rma } \le \e }$$  the \emph{inflated} level set at scale~$\e$. We will focus on algorithms whose outputs $S_n$ are $\e$-approximations of $\{f=a\}$, as defined below.

\begin{definition}[$\e$-approximation of a level set]
\label{def:eps-approx}
We say that a set $\rmS \s \dcube$ is an \emph{$\e$-approximation} of the level set $\fa$ if and only if it contains $\fa$ while consisting only of points at which $\rmf$ is at most $\e$-away from $\rma$, i.e.,
\begin{equation}
    \label{e:e-approx}
    \fa
\s
    \rmS
\s
    \fae \;.
\end{equation}
\end{definition}

The main mathematical problem we address is that of determining the \emph{sample complexity} of level set approximation, that is, the minimum number of evaluations of $f$ after which $S_n$ is an $\e$-approximation of $\{f=a\}$ (see Section~\ref{sec:def} of the \suppl{} for a formal definition). We are interested in algorithms with rate-optimal worst-case sample complexity over classical function classes, as well as (slightly) improved sample complexity bounds in more favorable cases.

\paragraph{Main Contributions and Outline of the Paper.}

\begin{itemize}[nosep]
    \item We define a new learning goal for level set approximation (see above and Section~\ref{sec:def} of the \suppl{}) similar in spirit to that of \citet{gotovos2013active}.
    \item In Section~\ref{sec:hardness} we briefly discuss the inherent hardness of the level set approximation problem (Theorem~\ref{thm:inherenthardness}) and the role played by smoothness or structural assumptions on $f$.
    \item In Section~\ref{s:hold-bi} we design a family of algorithms called \emph{\bai{}} through which we reduce the level set approximation problem to a local function approximation problem.
    \item In Sections~\ref{s:holder} and~\ref{s:g-hold-biga} we instantiate \bai{} to the cases of \hold{} or \gradho{} functions.
    We derive upper and lower bounds showing that the sample complexity for level set approximation is of the order of $1/\e^{d/\beta}$ in the worst-case, 
    where $\beta \in (0,2]$ is a smoothness parameter.
    \item In Section~\ref{s:adapt-small-nls} we also show that \bai{} algorithms adapt to more favorable functions~$f$ by featuring a slightly improved sample complexity in such cases.
\end{itemize}
Some lemmas and proofs are deferred to the \suppl{}.

\paragraph{Related Works.}
Sequential learning (sometimes referred to as sequential design of experiments) for level set and sublevel set identification is an active field of research. Many algorithms are based on Gaussian process priors over the black box function $f$ \citep{Ranjan2008,
Vazquez.Bect2009, Picheny.etal2010, bect12,chevalier14, Ginsbourger.etal2014, Wang.etal2016, bect2017bss,gotovos2013active}. 
In contrast with this large number of algorithms, few theoretical guarantees exist on the consistency or rate of convergence. 
Moreover, the majority of these guarantees are probabilistic. 
This means that consistency results state that an error goes to zero almost surely with respect to the Gaussian process prior measure over the unknown function $f$, and that the rates of convergence hold in probability, with respect to the same prior measure.
In this probabilistic setting,
\cite{bect2019supermartingale} provide a consistency result for a class of methods called Stepwise Uncertainty Reduction. \cite{gotovos2013active} provide rates of convergence, with noisy observations and for a classification-based loss function. 

The loss function of \cite{gotovos2013active}, given for sublevel set estimation, is similar in spirit to the notion of $\e$-approximation studied here for level set approximation, since we both aim at making decisions that are approximately correct for all $x$ in the input space.
The main difference is that \cite{gotovos2013active} assume that $f$ is a realization of a Gaussian process and thus provide guarantees that are probabilistic, while we prove deterministic bounds (for a fixed function). On the other hand, they consider noisy observations, while we assume $f$ can be evaluated perfectly.


A related problem studied in statistics is density level set estimation, in which the superlevel set of a density $\rmf$ is estimated by looking at i.i.d.\ draws of random variables with density $\rmf$.
For this problem, several different performance measures are considered, such as the Hausdorff distance \citep{cadre2013estimation,singh2009adaptive,tsybakov1997nonparametric} or a measure of the symmetric difference \citep{cadre2006kernel,rigollet2009fast,tsybakov1997nonparametric}.

When the function $\rmf$ is convex, our problem is also related to that of approximating a convex compact body with a simpler set (e.g., a polytope) in Hausdorff distance.
This has been studied extensively in convex geometry and several sequential and non-sequential algorithms have been proposed (see, e.g., the two surveys \citealt{kamenev2019optimal,gruber1993aspects} and references therein).

The closest connections with our work are within the bandit optimization literature. More precisely, our \bai{} algorithm and its analysis are inspired from the branch-and-bound algorithm of \citet[Appendix A.2]{LC18-AdaptivitySmoothnessBandits} and from the earlier methods of \citet{Per-90-OptimizationComplexity}, \citet[HOO algorithm]{bubeck2011x}, and \citet[DOO algorithm]{munos2014bandits}.
All these algorithms address the problem of finding a global extremum of $f$, while we are interested in finding level sets. However the idea of using a $2^d$-ary tree to construct refined partitions of the input domain, and sequential methods to select which branch to explore next, are key in this paper.

There are also algorithmic connections with the nonparametric statistics literature. In particular, the idea of locally approximating a target function has been used many times for different purposes (e.g., \citealt{GyKoKrWa-02-DistributionFreeNonparametric,Tsy-09-NonParametric}).


\paragraph{Additional Notation.} We denote the set of all positive integers $\{1,2,\ld\}$ by $\bbN$.
For all $\rmx \in \bbR$, we denote by $\lce{\rmx}$ (resp., $\lfl{\rmx}$) the ceiling (resp., floor) function at $\rmx$, i.e., the smallest (resp., largest) integer larger (resp., smaller) or equal to $\rmx$. Finally, for two sets $A$ and $B$, we write $A \subseteq B$ to say that $A$ is included in $B$ (possibly with equality).

\section{INHERENT HARDNESS}
\label{sec:hardness}

In this section we show that level sets are typically $(d-1)$-dimensional, and discuss the consequences of this fact in terms of the inherent hardness of the level set approximation problem.

We evaluate the dimension through the growth rate of packing numbers, one of the classical ways to measure the size of a set. In the case of the unit \hyperc{} and the $\sup$-norm, recall that packing numbers are defined as follows.
\begin{definition}[Packing number]
\label{d:packing-number}
For all $\rmr>0$, the $\rmr$-\emph{packing number} $\cN(\rmE,\rmr)$ of a subset $\rmE$ of $\dcube$ (with respect to the $\sup$-norm $\lno{\cdot}_\iop$) is the largest number of $\rmr$-separated points contained in $\rmE$, i.e.,
\begin{multline}
\label{e:packing}
    \cN (\rmE,\rmr)
:=
    \sup \bigl\{
        \rmk \in \bbN  :
\\
        \exists \, \bx_1, \ld, \bx_k \in \rmE, \, \min_{\rmi\neq \rmj} \lno{ \bx_\rmi - \bx_\rmj }_\iop > \rmr
    \bigr\}
\end{multline}
if $\rmE$ is nonempty, zero otherwise.
\end{definition}

The next theorem indicates that, with the exceptions of sets of minimizers or maximizers, $\e$-packing numbers of level sets $\fa$ of continuous functions $\rmf$ are at least $(d-1)$-dimensional. This result is very natural since $\fa$ is the solution set of one equation with $d$ unknowns.
\begin{theorem}
\label{thm:inherenthardness}
Let $\rmf\colon \dcube \to \bbR$ be a non-constant continuous function, and $\rma \in \bbR$ be any level such that $\min_{\bx \in \dcube} \rmf(\bx) < \rma < \max_{\bx \in \dcube}\rmf(\bx)$.
Then, there exists $\kappa>0$ such that, for all $\e>0$,
\[ 
    \cN \brb{ \fa, \; \e }
\ge
    \kappa \frac{1}{\e^{\rmd-1}} \;.
\]
\end{theorem}
We restate and prove this result in the \suppl{} (Theorem~\ref{t:level-sets-are-big}, Section~\ref{s:nls-dim}).

We note an important difference with the global optimization problem. Indeed, the set of global maximizers (or minimizers) of a function $f$ is typically finite and thus $0$-dimensional. This implies that, depending on the shape of $f$ around a global optimum, global optimization algorithms feature a sample complexity ranging roughly between $\log(1/\e)$ and $(1/\e)^d$ (see, e.g., \citealt{Per-90-OptimizationComplexity,munos2014bandits}).

In our case, by Theorem~\ref{thm:inherenthardness}, level sets are large, so that we can expect the sample complexity to depend heavily on the input dimension $d$. This is however not the end of the story. Indeed, as in nonparametric statistics  (e.g., \citealt{GyKoKrWa-02-DistributionFreeNonparametric,Tsy-09-NonParametric}) or in convex optimization (e.g., \citealt{Nes-04-ConvexOptimization,BoVa-04-ConvexOptimization,Bub-15-ConvexOptimization}), additional smoothness or structural assumptions like convexity of $f$ play a role in the hardness of the level set approximation problem. Since this problem is important in practice, designing algorithms that best exploit such additional assumptions is an important question. This is what we address in this paper.

\section{\bi{} ALGORITHMS \& ANALYSIS}
\label{s:hold-bi}

In this section, we introduce and analyze a family of algorithms designed for the problem of approximating the level set of an unknown function.
They are based on an iterative refinement of the domain $\dcube$, as made precise in the following definition.\\[-0.1cm]

\begin{definition}[Bisection of a family of \hyperc{}s]
\label{d:bisection}
Let $\cC$ be a family of $\rmn$ \hyperc{}s included in $\dcube$.
We say that $\bisect{\cC}$ is the \emph{bisection} of $\cC$ if it contains exactly the $2^\rmd \, \rmn$ \hyperc{}s obtained by subdividing each $\rmC = [a_1,b_1] \times \cdots [a_d,b_d] \in \cC$ into the $2^\rmd$ equally-sized smaller \hyperc{}s
of the form $\rmC' = I_1 \times \cdots \times I_d$ with $I_j$ being either $\bigl[a_j,(a_j+b_j)/2\bigr]$ or $\bigl[(a_j+b_j)/2,b_j\bigr]$.
\end{definition}
Our algorithm is of the branch-and-bound type, similarly to other bandit algorithms for global optimization such as that of \citet[Appendix A.2]{LC18-AdaptivitySmoothnessBandits} and earlier methods \citep{Per-90-OptimizationComplexity, bubeck2011x, munos2014bandits}.

Our \bai{} algorithms\footnote{We refer to \bai{} algorithms in the plural form because different \bi{} algorithms can be defined with the same input, depending on which rules are used to pick points at line~\ref{a:pick} and \interp{}s at line~\ref{a:loc-interp}. 
E.g., \bih{} (Section~\ref{s:holder}) only looks at the center of each \hyperc{} and uses constant \interp{}s, while \biga{} (Section~\ref{s:g-hold-biga}) queries the value of $\rmf$ at all vertices of each \hyperc{} and builds higher-order polynomial \interp{}s.} (\bi{}, Algorithm~\ref{alg:bi}) maintain, at all iterations $\rmi$, a collection $\cC_\rmi$ of \hyperc{}s on which the target function $\rmf$ is determined to take values close to the target level~$\rma$.
A \bi{} algorithm takes as input the level $\rma$, a common number of queries $\rmk$ (to be performed in each \hyperc{} at all iterations), and a pair of \tol{} parameters $\rmb,\beta> 0$, related to the smoothness of $\rmf$ and the approximation power of the \interp{}s used by the algorithm. 
At the beginning of each iteration $\rmi$, the collection of \hyperc{}s  $\cC_{\rmi-1}$ determined at the end of the last iteration is bisected (line~\ref{a:bisect}), so that all new \hyperc{}s have diameter $2^{-\rmi}$ (in the $\sup$-norm).
Then, the values of the target function $\rmf$ at $\rmk$ points of each newly created \hyperc{} are queried (lines~\ref{a:pick}--\ref{a:observe}). 
The output set $\rmS_\rmn$ after $\rmn$ queries to $\rmf$ is a subset of the union of all \hyperc{}s in $\cC_{\rmi-1}$, i.e., the collection of all \hyperc{}s determined during to the latest \emph{completed} iteration.
The precise definition of $\rmS_\rmn$ depends on the \interp{}s used during the last completed iteration and the two \tol{} parameters $\rmb,\beta$ (lines~\ref{a:output}~and~\ref{a:output-def}).\footnote{Notably, the output set $\rmS(\rmi)$ at iteration $\rmi$ can be represented succinctly and testing if $\bx \in \rmS(\rmi)$ can be done efficiently in all our \bi{} instances in Sections~\ref{s:holder},~\ref{s:g-hold-biga}.}
After all $\rmk$ values of $\rmf$ are queried from a \hyperc{} $\rmC'$, this information is used to determine a local \interp{} $\rmg_{\rmC'}$ of $\rmf$ (line~\ref{a:loc-interp}).
Finally, the collection of \hyperc{}s $\cC_\rmi$ is updated using $\rmg_{\rmC'}$ as a proxy for $\rmf$ (line~\ref{a:update}) for all \hyperc{}s $\rmC'$. 
In this step, all \hyperc{}s $\rmC'$ in which the proxy $\rmg_{\rmC'}$ is too far from the target level $\rma$ are discarded, where the tightness of the rejection rule increases with the passing of the iterations $\rmi$ and it is further regulated by the two \tol{} parameters $\rmb, \beta$.
\begin{algorithm2e}
\DontPrintSemicolon
\SetKwInput{KwIn}{input}
\SetKwInput{kwInit}{init}
\KwIn{level $\rma \in \bbR$, queries $\rmk \in \bbN$, \toll{} $\rmb,\beta > 0$ }
\kwInit{$\rmD \gets \dcube$, $\cC_{0}\gets \{ \rmD \}$, $\rmg_\rmD \equiv \rma$, $\rmn \gets 0$ }
    \For
    {
        iteration $\rmi=1,2,\ld$
    }
    {
        $\rmS{(\rmi)}  \gets \! \! \! \! \! \underset{\rmC \in \cC_{i-1}}{\bigcup} \! \! \! \! \bcb{ \bx \in \rmC : \babs{ \rmg_{\rmC} (\bx) - \rma } \le \rmb \, 2^{- \beta (\rmi-1)} }$\nllabel{a:output-def}\;
        $\cC_{\rmi}' \gets \bisect{\cC_{\rmi-1}}$\nllabel{a:bisect}\;
        \For
        {%
           \textbf{\upshape each} \hyperc{} $\rmC'\in\cC_{i}'$
        }
        {
            \For
            {%
                $\rmj = 1, \ld, \rmk$
            }
            {
                update $\rmn \gets \rmn + 1$\;
                pick a query point $\bx_\rmn \in \rmC'$\nllabel{a:pick}\;
                observe $\rmf(\bx_\rmn)$\nllabel{a:observe}\;
                output $\rmS_\rmn \gets \rmS{(\rmi)}$\nllabel{a:output}\;
            }
        pick a local \interp{} $\rmg_{\rmC'} \colon \rmC' \to \bbR$\nllabel{a:loc-interp}\;
        }
        $\cC_{\rmi} \! \gets \! \! \bcb{ \rmC'\in\cC_{\rmi}' : \exists \bx \in \rmC', \babs{ \rmg_{\rmC'} (\bx) - \rma } \! \le \! \rmb \, 2^{- \beta \rmi} }$\nllabel{a:update}\;
    }
\caption{\label{alg:bi} \bai{} (\bi)}
\end{algorithm2e}

Our analysis of \bi{} algorithms (Theorem~\ref{t:general-packing}) hinges on the accuracy of the \interp{}s $\rmg_{\rmC'}$ selected at line~\ref{a:loc-interp}, as formalized in the following definition.
\begin{definition}[Accurate approximation]
\label{d:accurate-approx}
Let $\rmb,\beta > 0$ and $\rmC \s \dcube$. We say that a function $\rmg\colon \rmC \to \bbR$ is a $(\rmb,\beta)$-\emph{accurate approximation} of another function $\rmf\colon \dcube\to \bbR$ (on $\rmC$) if the distance (in the $\sup$-norm on $\rmC$) between $\rmf$ and $\rmg$ can be controlled with the diameter (in the $\sup$-norm) of $\rmC$ as
\[
    \sup_{\bx \in \rmC} \babs{ \rmg(\bx) - \rmf(\bx) }
\le
    \rmb \, \Brb{ \, \sup_{\bx,\by \in \rmC} \lno{ \bx - \by }_\iop \, }^\beta \;.
\]
\end{definition}
%
%
We now present one of our main results, which states that \bi{} algorithms run with accurate approximations of the target function return $\e$-approximations of the target level set after a number of queries that depends on the packing number (Definition~\ref{d:packing-number}) of the inflated level set at decreasing scales.\footnote{\label{ft:sublevel}Note that our \bi{} algorithm can be used for estimating a sublevel set $\{\rmf \le \rma \}$ by simply dropping the absolute values in lines \ref{a:output-def} and \ref{a:update} of Algorithm~\ref{alg:bi}. As the reader might realize, the same proof techniques would apply with the corresponding (straightforward) changes. 
An analogous argument applies to superlevel sets $\{\rmf \ge \rma \}$.}

\begin{theorem}
\label{t:general-packing}
Consider a \bai{} algorithm (Algorithm~\ref{alg:bi}) run with input $\rma, \rmk, \rmb,\beta$.
Let $\rmf \colon \dcube \to \bbR$ be an arbitrary function with level set $\fa \neq \varnothing$.
Assume that the \interp{}s $\rmg_{\rmC'}$ selected at line~\ref{a:loc-interp} are $(\rmb,\beta)$-accurate approximations of $\rmf$ (Definition~\ref{d:accurate-approx}).
%
%
%
Fix any accuracy $\e > 0$, let $\rmi(\e) := \bce{ (\nicefrac{1}{\beta}) \log_2 (\nicefrac{2\rmb}{\e}) }$, and define $\rmn(\e)$ by\footnote{Letting $\sum_{\rmi = 0}^{-\rmm} \rma_\rmi = 0$ for any $\rmm > 0$ and all $\rma_\rmj \in \bbR$.}
\begin{equation}
\label{e:thm-n-eps}
    4^\rmd \, \rmk \! \! \sum_{\rmi=0}^{\rmi(\e)-1} \! \! \lim_{\delta \to 1^-} \cN \Brb{ \bcb{ \labs{\rmf - \rma} \le 2 \, \rmb \, 2^{- \beta \rmi} }, \ \delta \, 2^{-\rmi} }\;.    
\end{equation}
Then, for all $\rmn > \rmn(\e)$, the output $\rmS_\rmn$ returned after the $\rmn$-th query is an $\e$-approximation of $\fa$.
\end{theorem}
The expression \eqref{e:thm-n-eps} can be simplified by taking $\delta = 1$ and increasing the leading multiplicative constant.
However, in the following sections we will see how to upper bound this quantity with simpler functions of $1/\e$, for which the limit can be computed exactly.
\begin{proof}
Fix any $\rmn > \rmn(\e)$.
We begin by proving that 
\begin{equation}
\label{e:first-inclusion-general}
    \fa \s \rmS_\rmn \;.
\end{equation}
Recall that $\rmS_\rmn = \bigcup_{\rmC \in \cC_{\iota-1}} \bcb{ \bx \in \rmC : \babs{ \rmg_{\rmC} (\bx) - \rma } \le \rmb \, 2^{- \beta (\iota-1)} }$ (line~\ref{a:output}), where $\iota = \iota(\rmn)$ is the iteration during which the $\rmn$-th value of $\rmf$ is queried.
To prove \eqref{e:first-inclusion-general}, we will show the stronger result: for all $i\ge0$,
\begin{equation}
    \label{e:first-incuision-stronger}
    \fa \s \bigcup_{\rmC \in \cC_\rmi} \Bcb{ \bx \in \rmC : \babs{ \rmg_{\rmC} (\bx) - \rma } \le \rmb \, 2^{- \beta \rmi} },
\end{equation}
i.e., that the level set $\fa$ is \emph{always} included in the output set, not only after iteration $\iota(\rmn)-1$ has been completed.
We do so by induction.
If $\rmi = 0$, then $\fa \s \dcube = \bcb{ \bx \in \dcube : \labs{ \rma - \rma } \le \rmb \, 2^{- \beta \cdot 0} }$, which is the union in \eqref{e:first-incuision-stronger} by definition of $\rmD = \dcube$, $\cC_0 = \bcb{ \rmD }$ and $\rmg_\rmD \equiv \rma$ in the initialization of Algorithm~\ref{alg:bi}.
Assume now that the inclusion holds for some $\rmi-1$: we will show that it keeps holding for the next iteration $\rmi \in \bbN$.
Indeed, fix any $\bz \in \fa$.
By induction, $\bz$ belongs to some \hyperc{} $\rmC \in \cC_{\rmi-1}$.
Since $\cC'_\rmi = \bisect{\cC_{\rmi-1}}$ (line~\ref{a:bisect}), by definition of bisection (Definition~\ref{d:bisection}) we have that
$
    \bigcup_{\rmC' \in \cC'_\rmi} \rmC' 
=
    \bigcup_{\rmC \in \cC_{\rmi-1}} \rmC
$, 
which in turns implies that there exists a \hyperc{} $\rmC'_{\bz} \in \cC'_\rmi$ such that $\bz \in \rmC'_{\bz}$.
We show now that this $\rmC'_{\bz}$ also belongs to $\cC_\rmi$, i.e., that it is not discarded during the update of the algorithm at line~\ref{a:update}.
Indeed, since $\rmg_{\rmC'_{\bz}}$ is a $(\rmb,\beta)$-accurate approximation of $\rmf$ on $\rmC'_{\bz}$ (by assumption) 
and the diameter (in the $\sup$-norm) of $\rmC'_{\bz}$ is $\sup_{\bx, \by \in \rmC'_{\bz}} \lno{\bx-\by}_{\iop} = 2^{-\rmi}$, we have that
$
    \babs{ \rmg_{\rmC'_{\bz}}(\bz) - \rma }
=
    \babs{ \rmg_{\rmC'_{\bz}}(\bz) - \rmf(\bz) }
\le
    \rmb \, 2^{-\beta \rmi}
$.
This gives both that $\bz \in \rmC'_{\bz} \in \cC_\rmi$ (by definition of $\cC_\rmi$ at line~\ref{a:update}) and, consequently, that $\bz \in \bigcup_{\rmC \in \cC_\rmi} \bcb{ \bx \in \rmC : \babs{ \rmg_{\rmC} (\bx) - \rma } \le \rmb \, 2^{- \beta \rmi} }$, which clinches the proof of \eqref{e:first-incuision-stronger} and in turn yields \eqref{e:first-inclusion-general}.

We now show the validity of the second inclusion
\begin{equation}
\label{e:second-inclusion-general}
    \rmS_\rmn
\s
    \fae \;.
\end{equation}
As above, let $\iota = \iota(\rmn)$ be the iteration during which the $\rmn$-th value of $\rmf$ is queried by the algorithm.
Fix any $\bz \in \rmS_\rmn$.
We will prove that $\bz \in \fae$ or, restated equivalently, that $\babs{\rmf(\bz) - \rma} \le \e$.
By definition of $\rmS_\rmn = \bigcup_{\rmC \in \cC_{\iota-1}} \bcb{ \bx \in \rmC : \babs{ \rmg_{\rmC} (\bx) - \rma } \le \rmb \, 2^{- \beta (\iota-1)} }$ (line~\ref{a:output}), since $\bz \in \rmS_\rmn$, then there exists $\rmC_{\bz} \in \cC_{\iota -1}$ such that $\bz \in \rmC_{\bz}$ and $\babs{ \rmg_{\rmC_{\bz}} (\bz) - \rma } \le \rmb \, 2^{- \beta (\iota-1)}$.
Moreover, since $\rmC_{\bz} \in \cC_{\iota -1 } \s \cC'_{\iota -1}$ has diameter $\sup_{\bx, \by \in \rmC_{\bz}} \lno{\bx-\by}_{\iop} = 2^{-(\iota -1)}$ (in the $\sup$-norm) and the \interp{} $\rmg_{\rmC_{\bz}}$ is a $(\rmb,\beta)$-accurate approximation of $\rmf$ on $\rmC_{\bz}$ (by assumption), we have that $\babs{\rmf(\bz) - \rmg_{\rmC_{\bz}}(\bz)} \le \rmb \, 2^{- \beta (\iota -1)}$.
Thus
\[
    \babs{ \rmf(\bz) - \rma } 
\! \le \!
    \babs{ \rmf(\bz) - \rmg_{\rmC_{\bz}}(\bz) } 
    + \babs{ \rmg_{\rmC_{\bz}}(\bz) - \rma } 
\! \le \!
    2 \rmb 2^{-\beta (\iota -1)}
\]
and the right-hand side would be smaller than $\e$ ---proving~\eqref{e:second-inclusion-general}--- if either $\e \ge 2 \rmb$ (trivially), or in case $\e \in (0,2\rmb)$, if we could guarantee that the iteration $\iota = \iota(\rmn)$ during which the $\rmn$-th value of $\rmf$ is queried satisfies $\iota-1 \ge \bce{ (\nicefrac{1}{\beta}) \log_2 (\nicefrac{2\rmb}{\e}) } = \rmi(\e)$.
In other words, assuming without loss of generality that $\e \in (0,2\rmb)$ (so that $\rmi(\e) \ge 1$) and recalling that $\rmn > \rmn(\e)$, in order to prove \eqref{e:second-inclusion-general} we only need to check that the $\rmi(\e)$-th iteration is guaranteed to be concluded after at most $\rmn(\e)$ queries, where $\rmn(\e)$ is defined in terms of packing numbers in~\eqref{e:thm-n-eps}.
To see this, note that the total number of values of $\rmf$ that the algorithm queries by the end of iteration $\rmi(\e)$ is
$
    \sum_{\rmi=1}^{\rmi(\e)} \rmk \, \labs{ \cC'_{\rmi} } 
= 
    2^d \, \rmk \, \sum_{\rmi=1}^{\rmi(\e)} \labs{ \cC_{\rmi-1} }
= 
    2^d \, \rmk \, \sum_{\rmi=0}^{\rmi(\e) -1} \labs{ \cC_{\rmi} }
$.
To conclude the proof, it is now sufficient to show that for all iterations $\rmi \ge 0$, the number of \hyperc{}s maintained by the algorithm can be upper bounded by
\begin{equation}
    \label{e:final-bound-main-thm}
    \labs{ \cC_{\rmi} } 
\le
    2^\rmd \! \lim_{\delta \to 1^-}  \cN \Brb{ \bcb{ \labs{\rmf - \rma} \le 2 \, \rmb \, 2^{-\beta \rmi} }, \ \delta \, 2^{-\rmi} }
\;.
\end{equation}
Fix an arbitrary $\delta \in (0,1)$. 
If $\rmi = 0$, then 
$
    \labs{\cC_0} 
= 
    1 
\le
    \cN \brb{ \bcb{ \labs{\rmf - \rma} \le 2 \, \rmb }, \, \delta }$
by the definitions of $\cC_0 = \lcb{ \dcube }$ (initialization of Algorithm~\ref{alg:bi}) and $\delta$-packing number (Definition~\ref{d:packing-number}) of $\bcb{ \labs{\rmf - \rma} \le 2 \, \rmb }$ (which is non-empty because it contains $\fa$).
Fix any iteration $\rmi\in \bbN$.
By definition of $\cC_\rmi$ (line~\ref{a:update}), for all \hyperc{}s $\rmC\in \cC_\rmi$ there exists a point $\bx_\rmC \in \rmC$ such that $\labs{ \rmg_{\rmC}(\bx_\rmC) - \rma } \le \rmb \, 2^{-\beta \rmi}$.
Hence, for each \hyperc{} $\rmC \in \cC_\rmi$ there exists one of its points $\bx_\rmC \in \rmC$ such that $\labs{ \rmf(\bx_\rmC) - \rma }$ can be upper bounded by
\begin{equation}
\label{e:can-be-strengthened}
	\labs{ \rmf(\bx_\rmC) - \rmg_\rmC(\bx_\rmC) } + \labs{ \rmg_\rmC(\bx_\rmC) - \rma }
\le
	2 \, \rmb \, 2^{- \beta \rmi} \;,
\end{equation}
where, recalling that all \hyperc{}s in $\cC_i$ have diameter $2^{-\rmi}$ (in the $\sup$-norm), the bound on the term $\labs{ \rmf(\bx_\rmC) - \rmg_\rmC(\bx_\rmC) }$ is a consequence of $\rmg_\rmC$ being a $(\rmb,\beta)$-accurate approximation of $\rmf$ on $\rmC$.

Now we claim that the family of \hyperc{}s $\cC_\rmi$ can be partitioned into $2^\rmd$ subfamilies $\cC_{\rmi} (1), \ld, \cC_{\rmi} (2^\rmd)$ with the property that all distinct \hyperc{}s $\rmC\neq \rmC'$ belonging to the same family $\rmC_{\rmi}(\rmk)$ are strictly $(\delta \, 2^{-\rmi})$-separated (in the $\sup$-norm), i.e., that for all $\rmk \in \{1,\ld,2^{\rmd}\}$ and all $\rmC,\rmC' \in \cC_\rmi(\rmk)$, $\rmC\neq\rmC'$, we have $\inf_{\bx \in \rmC, \by \in \rmC'} \lno{ \bx - \by }_\iop > \delta \, 2^{-\rmi}$.

We defer the proof of this claim to Section~\ref{s:proofs-ba} of the \suppl{} (for an insightful picture, see Figure~\ref{fig:damier} in the same section).
Assume for now that it is true and fix an arbitrary $\rmk \in \{1,\ld,2^{\rmd}\}$.
Then, for all $\rmC \in \cC_\rmi(\rmk)$, there exists $\bx_\rmC$ such that~\eqref{e:can-be-strengthened} holds.
Therefore, we determined the existence of $\babs{\cC_\rmi(\rmk)}$-many $\brb{ \delta \, 2^{-\rmi} }$-separated points that are all included in $\bcb{ \labs{ \rmf - \rma } \le 2 \, \rmb \, 2^{-\beta \rmi} }$.
By definition of $\brb{ \delta \, 2^{-\rmi} }$-packing number of $\bcb{ \labs{\rmf - \rma} \le 2 \, \rmb \, 2^{- \beta \rmi} }$ (i.e., the \emph{largest} cardinality of a set of $\brb{ \delta \, 2^{-\rmi} }$-separated points included in $\bcb{ \labs{\rmf - \rma} \le 2 \, \rmb \, 2^{- \beta \rmi} }$ ---Definition~\ref{d:packing-number}), this implies that 
$
    \babs{ \cC_\rmi(\rmk) }
\le
    \cN \brb{ \bcb{ \labs{\rmf - \rma} \le 2 \, \rmb \, 2^{- \beta \rmi} },\ \delta \, 2^{-\rmi} }
$.
Recalling that $\cC_{\rmi} (1), \ld, \cC_{\rmi} (2^\rmd)$ is a partition of $\rmC_\rmi$, we then obtain
\[
    \babs{ \cC_\rmi }
=
    \sum_{\rmk = 1}^{2^{\rmd}} \babs{ \cC_\rmi (\rmk) }
\le
    2^\rmd \cN \Brb{ \bcb{ \labs{\rmf - \rma} \le 2 \rmb 2^{- \beta \rmi} }, \delta 2^{-\rmi} }
\]
which, after taking the infimum over $\delta \in (0,1)$ and by the monotonicity of the packing number $\rmr \mapsto \cN ( \rmE, \rmr )$ (for any $\rmE \s \dcube$), yields
\begin{align*}
    \babs{ \cC_\rmi }
& 
\le
    \inf_{\delta \in (0,1)} \bbrb{ 2^\rmd \, \cN \Brb{ \bcb{ \labs{\rmf - \rma} \le 2 \rmb 2^{- \beta \rmi} }, \delta 2^{-\rmi} } }
\\
&
=
    2^\rmd \, \lim_{\delta \to 1^-} \cN \Brb{ \bcb{ \labs{\rmf - \rma} \le 2 \, \rmb \, 2^{- \beta \rmi} },\ \delta \, 2^{-\rmi} } \;.
\end{align*}
This gives~\eqref{e:final-bound-main-thm} and concludes the proof.
\end{proof}
By looking at the end of the proof of the previous result, one could see that the exponential term $4^\rmd$ in our bound~\eqref{e:thm-n-eps} could be lowered to $2^\rmd$ under the assumption that at any iteration $\rmi$, Algorithm~\ref{alg:bi} picks at least one $\bx_{\rmC'} \in \rmC'$ for each $\rmC' \in \cC_{\rmi}'$ such that $\lno{\bx_{\rmC_1'} - \bx_{\rmC_2'}}_\iop \ge 2^{-\rmi}$ for all distinct $\rmC_1',\rmC_2' \in \cC_{\rmi}'$.
Notably this property is enjoyed by all our \bi{} instances in Sections~\ref{s:holder},~\ref{s:g-hold-biga}.


\section{H\"OLDER FUNCTIONS}
\label{s:holder}

In this section, we focus on \hold{} functions, and we present a \bi{} instance that is \nearo{} for determining their level sets.
\begin{definition}[\hold{} function]
\label{d:hold}
Let $\rmc>0$, $\gamma \in (0,1]$, and $\rmE \s \dcube$. We say that a function $\rmf\colon \rmE \to \bbR$ is $(\rmc, \gamma)$-\emph{\hold{}} (with respect to the $\sup$-norm $\lno{\cdot}_\iop$) if 
$
    \babs{ \rmf(\bx) - \rmf(\by) } \le \rmc \, \lno{ \bx - \by }_\iop^\gamma
$, 
for all $\bx,\by \in \rmE$.
\end{definition}

Our \bi{} instance for \hold{} functions (\bih, Algorithm~\ref{alg:bi-h}) runs Algorithm~\ref{alg:bi} with $k=1$, $\rmb = \rmc$, $\beta = \gamma$. 
The local \interp{}s $\rmg_{\rmC'}$ are constant and equal to the value $\rmf(\bc_{\rmC'})$ at the center $\bc_{\rmC'}$ of $\rmC'$. 
In particular, the output set $\rmS_\rmn$ is now the entire union of all \hyperc{}s determined in the latest completed iteration.

\begin{algorithm2e}
\DontPrintSemicolon
\SetKwInput{KwIn}{input}
\SetKwInput{kwInit}{init}
\KwIn{level $\rma \in \bbR$, \toll{} $\rmc>0$, $\gamma \in (0,1]$ }
\kwInit{$\rmD \gets \dcube$, $\cC_{0}\gets \{ \rmD \}$, $\rmn \gets 0$ }
    \For
    {
        iteration $\rmi=1,2,\ld$
    }
    {
    	$\rmS(\rmi) \gets \bigcup_{\rmC \in \cC_{\rmi-1}} \rmC$\;
    	let $\cC_{\rmi}' \gets \bisect{\cC_{\rmi-1}}$\nllabel{a:bisect-h}\;
        \For
        {%
          \textbf{\upshape each} \hyperc{} $\rmC'\in\cC_{i}'$
        }
        {
            update $\rmn \gets \rmn + 1$\;
            pick the center $\bc_{\rmC'}$ of $\rmC'$ as the next $\bx_\rmn$\nllabel{a:pick-h}\;
            observe $\rmf(\bx_\rmn)$\nllabel{a:observe-h}\;
            output $\rmS_\rmn \gets \rmS(\rmi)$\nllabel{a:output-h}\;
        }
        $\cC_{\rmi}\gets \bcb{ \rmC'\in\cC_{\rmi}' : \babs{ \rmf (\bc_{\rmC'}) - \rma } \le \rmc \, 2^{- \gamma \rmi} }$\nllabel{a:update-h}\;
    }
\caption{\label{alg:bi-h} \bi{} for \hold{} Functions (\bih)}
\end{algorithm2e}


The next result shows that the optimal worst-case sample complexity of the level set approximation of \hold{} functions is of order $1/\e^{\rmd/\gamma}$, and it is attained by \bih{} (Algorithm~\ref{alg:bi-h}).

\begin{theorem}
\label{t:bah-rate}
Let $\rma \in \bbR, \rmc>0, \gamma \in (0,1]$, and $\rmf \colon \dcube \to \bbR$ be any $(\rmc,\gamma)$-\hold{} function with level set $\fa \neq \varnothing$.
Fix any accuracy $\e > 0$. 
Then, there exists $\kappa_1>0$ (independent of $\e$) such that, for all $\rmn > \kappa_1 / \e^{\rmd/\gamma}$, the output $\rmS_\rmn$ returned by \bih{} after the $\rmn$-th query is an $\e$-approximation of $\fa$.

Moreover, there exists $\kappa_2$ (independent of $\e$) such that no deterministic algorithm can guarantee to output an $\e$-approximation of the level set $\fa$ for all $(\rmc,\gamma)$-\hold{} functions $\rmf$, querying less than $\kappa_2 / \e^{\rmd/\gamma}$ of their values.
\end{theorem}
The proof is deferred to Section~\ref{s:proofs-holder} in the \suppl{}. 
The upper bound is an application of Theorem~\ref{t:general-packing}.
The lower bound is proven by showing that no algorithm can distinguish between the function $\rmf\equiv 0$ and a function that is non-zero only on a small ball, on which it attains the value $2\e$.
We use a classical construction with bump functions that appears, e.g., in Theorem~3.2 of \citet{GyKoKrWa-02-DistributionFreeNonparametric} for nonparametric regression lower bounds.

We remark that the rate in the previous result is the same as that of a naive uniform grid filling of the space with step-size of order $\e^{1/\gamma}$.
While this rate cannot be improved in the worst case, the leading constant of our sequential algorithm may be better if a large fraction of the input space can be rejected quickly.
More importantly, we will see in Section~\ref{s:adapt-small-nls} that (slightly) better rates can be attained by \bi{} algorithms if the inflated level sets of the target function $\rmf$ are smaller (as it happens, e.g., if $\rmf$ is convex with proper level set $\fa$).

In the following section, we will also investigate if and to what extent higher smoothness helps.
To this end, we will switch our focus to differentiable functions with \hold{} gradients.

\section{\texorpdfstring{$\nabla$-H\"OLDER FUNCTIONS: \biga{} ALGORITHM \& ANALYSIS}{GRADIENT-H\"OLDER FUNCTIONS: \biga{} ALGORITHM \& ANALYSIS}}

\label{s:g-hold-biga}

In this section, we focus on differentiable functions with \hold{} gradients, and we present a \bi{} instance that is \nearo{} for determining their level sets.
\begin{definition}[\Gradho{}/\lip{} function]
Let $\rmc_1 > 0$, $\gamma_1 \in (0,1]$, and $\rmE \s \dcube$.
We say that a function $\rmf \colon \rmE \to \bbR$ is $(\rmc_1, \gamma_1)$-\emph{\gradho{}} (with respect to $\lno{\cdot}_\iop$) if it is the restriction\footnote{They are defined as restrictions of continuously differentiable functions in order to have simply and well-defined gradients on the boundary of their domains.} (to $\rmE$) of a continuously differentiable function defined on $\bbR^\rmd$ such that $\bno{ \nabla \rmf(\bx) - \nabla \rmf(\by) }_\iop \le \rmc_1 \, \lno{ \bx - \by }_\iop^{\gamma_1}$ for all $\bx, \by \in \rmE$.
If $\gamma_1 = 1$, we say that $\rmf$ is $\rmc_1$-\emph{\gradlip{}}.
\end{definition}
The next lemma introduces the polynomial \interp{}s that will be used by our \bi{} instance and it shows that they are $(\rmc_1\rmd, 1+\gamma_1)$-accurate approximations of $\rmf$ on all \hyperc{}s.
\begin{lemma}[\biga{} \interp{}s]
\label{l:interp-biga}
Let $\rmf\colon \rmC' \to \bbR$ be a $(\rmc_1,\gamma_1)$-\gradho{} function, for some $\rmc_1>0$ and $\gamma_1\in (0,1]$.
Let $\rmC' \s \dcube$ be a \hyperc{} with diameter $\fd \in (0,1]$ and set of vertices $\rmV'$, i.e., $\rmC' = \prod_{\rmj=1}^\rmd \lsb{ \rmu_\rmj, \rmu_\rmj + \fd }$, for some $\bu := (\rmu_1,\ld,\rmu_\rmd) \in [0, 1 - \fd]^\rmd$, and $\rmV' = \prod_{\rmj=1}^\rmd \lcb{ \rmu_\rmj, \rmu_\rmj + \fd }$. 
The function 
\begin{align}
	\rmh_{\rmC'}\colon \rmC'
&
	\to \bbR
\nonumber
\\[-2ex]
	\label{e:interp-for-biga-new}
	\bx
&
	\mapsto 
	\sum_{\bv \in \rmV'}
	\rmf(\bv)
	\prod_{\rmj = 1}^{\rmd}
	\rmp_{\rmv_{\rmj}}(\rmx_{\rmj})\;,
\end{align}
where
\[
    \rmp_{\rmv_{\rmj}}(\rmx_{\rmj})
:=
	\lrb{ 1 - \frac{\rmx_{\rmj} - \rmu_\rmj}{\fd} } 
	\, \bbI_{\rmv_\rmj = \rmu_\rmj}
	+ \frac{\rmx_\rmj - \rmu_\rmj}{\fd} 
	\, \bbI_{\rmv_\rmj = \rmu_\rmj + \fd} ,
\]
interpolates the $2^\rmd$ pairs $\bcb{ \lrb{ \bv, \rmf (\bv) } }_{\bv \in \rmV'}$ and it satisfies
\[
	\sup_{\bx \in \rmC'} \babs{ \rmh_{\rmC'}(\bx) - \rmf(\bx) } 
\le 
	\rmc_1 \rmd \, \fd^{1+\gamma_1}\;.
\]
\end{lemma}

The technical proof of the previous lemma is deferred to Section~\ref{s:proofs-gradho} of the \suppl{}.

Our \bai{} instance for \gradho{} functions (\biga) runs Algorithm~\ref{alg:bi} with $k=2^\rmd$, $\rmb = \rmc_1\rmd$, and $\beta = 1+\gamma_1$. 
The local \interp{} $\rmh_{\rmC'}$ (defined in~\eqref{e:interp-for-biga-new}) are computed by querying the values of $\rmf$ at all vertices of $\rmC'$.
Note that line~\ref{a:update-g} of Algorithm~\ref{alg:bi-g} can be carried out efficiently since it is sufficient to check the condition on $\babs{ \rmh_{\rmC'} (\bx) - \rma }$ at the vertices $\bx$ of $\rmC'$.\footnote{Indeed, only three cases can occur. We set $\rho = \rmc_1 \rmd \, 2^{- (1+\gamma_1) \rmi}$. Case~1: if one of the vertices $\bx$ satisfies $\babs{ \rmh_{\rmC'} (\bx) - \rma } \le \rho$, then the condition is checked. Case~2: if the values $\rmh_{\rmC'} (\bx)$ at the vertices are all strictly below $a-\rho$ or all strictly above $a+\rho$, then it is also the case for all $\bx \in \rmC'$, since $\rmh_{\rmC'}(\bx)$ is a convex combination of all values at the vertices; so the condition is not checked. Case~3: if there are two vertices $\bx$ and $\by$ such that $\rmh_{\rmC'} (\bx)<a-\rho$ and $\rmh_{\rmC'} (\by)>a+\rho$, then there exists $\bz \in \rmC'$ such that $\babs{ \rmh_{\rmC'} (\bz) - \rma } \le \rho$ by continuity of $\rmh_{\rmC'}$ on $\rmC'$; so the condition is checked.}
Also, note that the output set $\rmS_\rmn$ is the union over \hyperc{}s of pre-images of segments from the polynomial functions in \eqref{e:interp-for-biga-new}.

\begin{algorithm2e}
\DontPrintSemicolon
\SetKwInput{KwIn}{input}
\SetKwInput{kwInit}{init}
\KwIn{level $\rma \in \bbR$, \toll{} $\rmc_1>0$, $\gamma_1 \in (0,1]$  }
\kwInit{$\rmD \gets \dcube$, $\cC_{0}\gets \{ \rmD \}$, $\rmh_\rmD \equiv \rma$, $\rmn \gets 0$ }
    \For
    {
        iteration $\rmi=1,2,\ld$
    }
    {
        $\rmS{(\rmi)}  \gets \bigcup_{\rmC \in \cC_{i-1}} \bcb{ \bx \in \rmC : \babs{ \rmh_{\rmC} (\bx) - \rma } \le \rmc_1 \rmd \, 2^{- (1+\gamma_1) (\rmi-1)} }$\nllabel{a:output-def-g}\;
        $\cC_{\rmi}' \gets \bisect{\cC_{\rmi-1}}$\nllabel{a:bisect-g}\;
        \For
        {%
          \textbf{\upshape each} \hyperc{} $\rmC'\in\cC_{i}'$
        }
        {
        	let $\rmV' \s \rmC'$ be the set of vertices of $\rmC'$\;
            \For
            {%
                \textbf{\upshape each} vertex $\bv \in \rmV'$
            }
            {
                update $\rmn \gets \rmn + 1$\;
                pick vertex $\bv \in \rmV'$ as the next $\bx_\rmn$\nllabel{a:pick-g}\;
                observe $\rmf(\bx_\rmn)$\nllabel{a:observe-g}\;
                output $\rmS_\rmn \gets \rmS{(\rmi)}$\nllabel{a:output-g}\;
            }
		interpolate the $2^\rmd$ pairs $\bcb{ \brb{ \bv, \rmf(\bv) } }_{\bv \in \rmV'}$ with $\rmh_{\rmC'} \colon \rmC' \to \bbR$ given by \eqref{e:interp-for-biga-new}\nllabel{a:loc-interp-g}\;
        }
        update $\cC_{\rmi}\gets \bcb{ \rmC'\in\cC_{\rmi}' : \text{there exists } \bx \in \rmC' \text{ such that }  \babs{ \rmh_{\rmC'} (\bx) - \rma } \le \rmc_1 \rmd \, 2^{- (1+\gamma_1) \rmi} }$\nllabel{a:update-g}\;
    }
\caption{\label{alg:bi-g} \bi{} for $\nabla$-\hold{} $\rmf$ (\biga)}
\end{algorithm2e}

\subsection{Worst-Case Sample Complexity}

The next result shows that the optimal worst-case sample complexity of the level set approximation of \gradho{} functions is of order $1/\e^{\rmd/(1+\gamma_1)}$, and it is attained by \biga{} (Algorithm~\ref{alg:bi-g}).

\begin{theorem}
\label{t:biga-rate}
Let $\rma \in \bbR, \rmc_1>0, \gamma_1 \in (0,1]$, and $\rmf \colon \dcube \to \bbR$ be any $(\rmc_1,\gamma_1)$-\gradho{} function with level set $\fa \neq \varnothing$.
Fix any accuracy $\e > 0$. 
Then, there exists $\kappa_1 >0$ (independent of $\e$) such that, for all  $\rmn > \kappa_1 / \e^{\rmd/(1+\gamma_1)}$, 
the output $\rmS_\rmn$ returned by \biga{} after the $\rmn$-th query is an $\e$-approximation of $\fa$.

Moreover, there exists $\kappa_2 >0$ (independent of $\e$) such that no deterministic algorithm can guarantee to output an $\e$-approximation of the level set $\fa$ for all $(\rmc_1,\gamma_1)$-\gradho{} functions $\rmf$, querying less than $\kappa_2 / \e^{\rmd/(1+\gamma_1)}$ of their values.
\end{theorem}

The proof proceeds similarly to that of Theorem~\ref{t:bah-rate}.
It is deferred to Section~\ref{s:proofs-gradho} of the \suppl{}.

Similarly to Section~\ref{s:holder}, the rate in the previous result could also be achieved by choosing query points on a regular grid with step-size of order $\e^{1/(1+\gamma_1)}$.
However, our sequential algorithm features an improved sample complexity outside of a worst-case scenario, as shown in the following section.

\subsection{\texorpdfstring{Adaptivity to Smaller $\ds$}{Adaptivity to Smaller d*}}
\label{s:adapt-small-nls}

Our general result (Theorem~\ref{t:general-packing}) suggests that the sample complexity can be controlled whenever there exists $\ds \ge 0$ such that\\[-0.4cm]
\[
    \forall \rmr \in (0,1) , \quad
    \cN \Brb{ \bcb{ \labs{ \rmf - \rma } \le \rmr }, \; \rmr }
\le
    \Cs \lrb{ \frac{1}{\rmr} }^{\ds}  \hspace{-9pt} \;.
\]
for some $\Cs>0$. 
We call such a $\ds$ a \emph{\nls{} dimension} of $\fa$.
Note that such a $\ds$ always exists and $\ds \le \rmd$ by $\bcb{ \labs{ \rmf-\rma } \le \rmr } \s \dcube$. 
However $\ds \ge \rmd -1$ by Theorem~\ref{thm:inherenthardness} for non-degenerate level sets of continuous functions (for more details, see Section~\ref{s:nls-dim} in the \suppl{}).
The definition of \nls{} dimension leads to the following result.

\begin{corollary}
\label{c:for-biga-main}
Let $\rma \in \bbR, \rmc_1>0, \gamma_1 \in (0,1]$, and $\rmf \colon \dcube \to \bbR$ be any $(\rmc_1,\gamma_1)$-\gradho{} function with level set $\fa \neq \varnothing$.
Let $\ds \in [\rmd-1,\rmd]$ be a \nls{} dimension of $\fa$.
Fix any accuracy $\e > 0$.
Then, for all $\rmn > \rmm(\e)$, the output $\rmS_\rmn$ returned by \biga{} after the $\rmn$-th query is an $\e$-approximation of $\fa$, where
\[
    \rmm(\e)
:=
    \begin{cases}
        \displaystyle{     \kappa_1 + \kappa_2 \log_2 \lrb{ \frac{1}{\e^{1/(1+\gamma_1)}} }^+ }
    &
        \text{if } \ds = 0 \;,\\[3ex]
        \displaystyle{ \kappa(\ds) \frac{1}{\e^{\ds/(1+\gamma_1)}} }
    &
        \text{if } \ds > 0 \;,\\
    \end{cases}
\]
for $\kappa_1, \kappa_2, \kappa(\ds) \ge 0$ independent of $\e$, that depend exponentially on $\rmd$, where $\rmx^+ = \max\{\rmx,0\}$.
\end{corollary}

We remark that $\ds = \rmd - 1$ can be achieved by well-behaved functions.
This is typically the case when $\rmf$ is convex or (as a corollary) if it consists of finitely many convex components.\footnote{More precisely, if for some $\rma' > \rma$, we have that the sublevel $\{ \rmf \le \rma' \}$ is a disjoint union of a finite number of convex sets on which $\rmf$ is convex.}
This non-trivial claim is proved in Section~\ref{s:convex-case} of the \suppl{} for convex functions with a proper level set\footnote{$\fa$ is $\Delta$-proper for some $\Delta>0$ if we have $\min_{\bx \in \dcube}\rmf(\bx) + \Delta \le \rma \le \min_{\bx \in \partial \dcube} \rmf(\bx)$.}.
The following result, combined with this fact and Corollary~\ref{c:for-biga-main}, shows that \biga{} is \nearo{} for determining proper level sets of convex \gradlip{} functions.

\begin{theorem}
\label{t:convex-summary}
Fix any level $\rma \in \bbR$ and an arbitrary accuracy $\e > 0$.
No deterministic algorithm $\rmA$ can guarantee to output an $\e$-approximation of any $\Delta$-proper level set $\fa$ of an arbitrary convex $\rmc_1$-\gradlip{} functions $\rmf$ with $\rmc_1\ge 3$ and $\Delta \in (0, \nicefrac{1}{4}]$, querying less than $\kappa / \e^{(\rmd-1)/2}$ of their values, where $\kappa > 0$ is a constant independent of $\e$.
\end{theorem}

We give a complete proof of this result in Section~\ref{s:convex-lower-bound} of the \suppl{}.

\section{CONCLUSION}
\label{sec:conclusion}

We studied the problem of determining $\e$-ap\-prox\-i\-ma\-tions of the level set of a target function $\rmf$ by only querying its values.
After discussing the inherent hardness of the problem (Theorem~\ref{thm:inherenthardness}), we designed the class of \bi{} algorithms for which we proved theoretical guarantees under the assumption that accurate local approximations of $\rmf$ can be computed by only looking at its values (Theorem~\ref{t:general-packing}).

This provides a general method to reduce our level set approximation problem to a local approximation problem, decoupled from the original one. 

Such an approach leads to \nearo{} worst-case sample complexity guarantees for the case of \hold{} and \gradho{} functions (Theorems~\ref{t:bah-rate},~\ref{t:biga-rate}).
At the same time, we show that in some cases our \bi{} algorithms adapt to a natural structural property of~$\rmf$, namely small \nls{} dimension (Corollary~\ref{c:for-biga-main}) including convexity (Theorem~\ref{t:convex-summary} and preceding discussion).



\paragraph{Future Work.}

Compared to the best achievable rate $1/\e^{\rmd/\gamma}$ for $(\rmc,\gamma)$-\hold{} functions, we show that \bi{} algorithms converge at a faster $1/\e^{\rmd/(1+\gamma_1)}$ rate if $\rmf$ is $(\rmc_1,\gamma_1)$-\gradho{}.
This points at an interesting line of research: the study of general \hold{} spaces in which the target function is $\rmk$ times continuously differentiable and the $\rmk$-th partial derivatives are $(\rmc_\rmk,\gamma_\rmk)$-\hold{}, for some $\rmk \in \bbN$, $\rmc_\rmk>0$, and $\gamma_\rmk\in(0,1]$.
We conjecture that a suitable choice of \interp{}s for our \bi{} algorithms would lead to a \nearo{} sample complexity of order $1/\e^{\rmd/(\rmk + \gamma_k)}$ for this class of functions, making optimal solutions for this problem sample-efficient. 
Another possible line of research is the design of algorithms that adapt to the smoothness of $\rmf$ when the latter is unknown, similarly to global bandit optimization (e.g., \citealt{grill2015black,bartlett2019simple}).
We leave these interesting directions open for future work.

 \section*{Acknowledgments}

 The work of Tommaso Cesari and S{\'e}bastien Gerchinovitz has benefited from the AI Interdisciplinary Institute ANITI, which is funded by the French ``Investing for the Future – PIA3'' program under the Grant agreement ANR-19-P3IA-0004. S\'{e}bastien Gerchinovitz gratefully acknowledges the support of the DEEL project (\url{https://www.deel.ai/}). This work benefited from the support of the project BOLD from the French national research agency (ANR).

\bibliography{biblio}

\begin{thebibliography}{36}
\providecommand{\natexlab}[1]{#1}
\providecommand{\url}[1]{\texttt{#1}}
\expandafter\ifx\csname urlstyle\endcsname\relax
  \providecommand{\doi}[1]{doi: #1}\else
  \providecommand{\doi}{doi: \begingroup \urlstyle{rm}\Url}\fi

\bibitem[Azzimonti et~al.(2020)Azzimonti, Ginsbourger, Chevalier, Bect, and
  Richet]{azzimonti2020adaptive}
Dario Azzimonti, David Ginsbourger, Cl{\'e}ment Chevalier, Julien Bect, and
  Yann Richet.
\newblock Adaptive design of experiments for conservative estimation of
  excursion sets.
\newblock \emph{Technometrics}, pages 1--14, 2020.

\bibitem[Bartlett et~al.(2019)Bartlett, Gabillon, and
  Valko]{bartlett2019simple}
Peter~L. Bartlett, Victor Gabillon, and Michal Valko.
\newblock A simple parameter-free and adaptive approach to optimization under a
  minimal local smoothness assumption.
\newblock In \emph{Algorithmic Learning Theory}, pages 184--206, 2019.

\bibitem[Bect et~al.(2012)Bect, Ginsbourger, Li, Picheny, and Vazquez]{bect12}
Julien Bect, David Ginsbourger, Ling Li, Victor Picheny, and Emmanuel Vazquez.
\newblock Sequential design of computer experiments for the estimation of a
  probability of failure.
\newblock \emph{Statistics and Computing}, 22\penalty0 (3):\penalty0 773--793,
  2012.

\bibitem[Bect et~al.(2017)Bect, Li, and Vazquez]{bect2017bss}
Julien Bect, Ling Li, and Emmanuel Vazquez.
\newblock {B}ayesian subset simulation.
\newblock \emph{SIAM/ASA Journal on Uncertainty Quantification}, 5\penalty0
  (1):\penalty0 762--786, 2017.

\bibitem[Bect et~al.(2019)Bect, Bachoc, Ginsbourger,
  et~al.]{bect2019supermartingale}
Julien Bect, Fran{\c{c}}ois Bachoc, David Ginsbourger, et~al.
\newblock A supermartingale approach to {G}aussian process based sequential
  design of experiments.
\newblock \emph{Bernoulli}, 25\penalty0 (4A):\penalty0 2883--2919, 2019.

\bibitem[Bouttier et~al.(2020)Bouttier, Cesari, and
  Gerchinovitz]{bouttier2020regret}
Cl{\'e}ment Bouttier, Tommaso Cesari, and S{\'e}bastien Gerchinovitz.
\newblock Regret analysis of the {P}iyavskii--{S}hubert algorithm for global
  {L}ipschitz optimization.
\newblock \emph{arXiv preprint arXiv:2002.02390}, 2020.

\bibitem[Boyd and Vandenberghe(2004)]{BoVa-04-ConvexOptimization}
Stephen Boyd and Lieven Vandenberghe.
\newblock \emph{Convex optimization}.
\newblock Cambridge University Press, 2004.

\bibitem[Bubeck(2015)]{Bub-15-ConvexOptimization}
S\'{e}bastien Bubeck.
\newblock Convex optimization: Algorithms and complexity.
\newblock \emph{Foundations and Trends in Machine Learning}, 8\penalty0
  (3-4):\penalty0 231--357, 2015.
\newblock ISSN 1935-8237.

\bibitem[Bubeck et~al.(2011)Bubeck, Munos, Stoltz, and
  Szepesv{\'a}ri]{bubeck2011x}
S{\'e}bastien Bubeck, R{\'e}mi Munos, Gilles Stoltz, and Csaba Szepesv{\'a}ri.
\newblock X-armed bandits.
\newblock \emph{Journal of Machine Learning Research}, 12\penalty0
  (May):\penalty0 1655--1695, 2011.

\bibitem[Cadre(2006)]{cadre2006kernel}
Beno{\^\i}t Cadre.
\newblock Kernel estimation of density level sets.
\newblock \emph{Journal of multivariate analysis}, 97\penalty0 (4):\penalty0
  999--1023, 2006.

\bibitem[Cadre et~al.(2013)Cadre, Pelletier, and Pudlo]{cadre2013estimation}
Beno{\^\i}t Cadre, Bruno Pelletier, and Pierre Pudlo.
\newblock Estimation of density level sets with a given probability content.
\newblock \emph{Journal of Nonparamentric Statistics}, 25\penalty0
  (1):\penalty0 261--272, 2013.

\bibitem[Chevalier et~al.(2014)Chevalier, Bect, Ginsbourger, Vazquez, Picheny,
  and Richet]{chevalier14}
Cl{\'e}ment Chevalier, Julien Bect, David Ginsbourger, Emmanuel Vazquez, Victor
  Picheny, and Yann Richet.
\newblock Fast parallel kriging-based stepwise uncertainty reduction with
  application to the identification of an excursion set.
\newblock \emph{Technometrics}, 56\penalty0 (4):\penalty0 455--465, 2014.

\bibitem[Galland et~al.(2004)Galland, R{\'e}fr{\'e}gier, and
  Germain]{galland2004synthetic}
Fr{\'e}d{\'e}ric Galland, Philippe R{\'e}fr{\'e}gier, and Olivier Germain.
\newblock Synthetic aperture radar oil spill segmentation by stochastic
  complexity minimization.
\newblock \emph{IEEE Geoscience and Remote Sensing Letters}, 1\penalty0
  (4):\penalty0 295--299, 2004.

\bibitem[Ginsbourger et~al.(2014)Ginsbourger, Baccou, Chevalier, Perales,
  Garland, and Monerie]{Ginsbourger.etal2014}
David Ginsbourger, Jean Baccou, Cl{\'e}ment Chevalier, Fr{\'e}d{\'e}ric
  Perales, Nicolas Garland, and Yann Monerie.
\newblock {B}ayesian adaptive reconstruction of profile optima and optimizers.
\newblock \emph{SIAM/ASA Journal on Uncertainty Quantification}, 2\penalty0
  (1):\penalty0 490--510, 2014.

\bibitem[Gotovos et~al.(2013)Gotovos, Casati, Hitz, and
  Krause]{gotovos2013active}
Alkis Gotovos, Nathalie Casati, Gregory Hitz, and Andreas Krause.
\newblock Active learning for level set estimation.
\newblock In \emph{Proceedings of the Twenty-Third international joint
  conference on Artificial Intelligence}, pages 1344--1350. AAAI Press, 2013.

\bibitem[Grill et~al.(2015)Grill, Valko, and Munos]{grill2015black}
Jean-Bastien Grill, Michal Valko, and R{\'e}mi Munos.
\newblock Black-box optimization of noisy functions with unknown smoothness.
\newblock In \emph{Advances in Neural Information Processing Systems}, pages
  667--675, 2015.

\bibitem[Gruber(1993)]{gruber1993aspects}
Peter Gruber.
\newblock Aspects of approximation of convex bodies.
\newblock In \emph{Handbook of convex geometry}, pages 319--345. Elsevier,
  1993.

\bibitem[Gy{\"o}rfi et~al.(2002)Gy{\"o}rfi, Krzy{\.z}ak, Kohler, and
  Walk]{GyKoKrWa-02-DistributionFreeNonparametric}
L{\'a}szl{\'o} Gy{\"o}rfi, Adam Krzy{\.z}ak, Michael Kohler, and Harro Walk.
\newblock \emph{{A Distribution-Free Theory of Nonparametric Regression}}.
\newblock Springer Series in Statistics. Springer-Verlag, New York, 2002.

\bibitem[Kamenev(2019)]{kamenev2019optimal}
George Kamenev.
\newblock Optimal non-adaptive approximation of convex bodies by polytopes.
\newblock In \emph{Numerical Geometry, Grid Generation and Scientific
  Computing}, pages 157--172. Springer, 2019.

\bibitem[Locatelli and Carpentier(2018)]{LC18-AdaptivitySmoothnessBandits}
Andrea Locatelli and Alexandra Carpentier.
\newblock Adaptivity to smoothness in {X}-armed bandits.
\newblock In Sébastien Bubeck, Vianney Perchet, and Philippe Rigollet,
  editors, \emph{Conference on Learning Theory}, volume~75 of \emph{Proceedings
  of Machine Learning Research}, pages 1463--1492. PMLR, 06--09 Jul 2018.

\bibitem[Munos et~al.(2014)]{munos2014bandits}
R{\'e}mi Munos et~al.
\newblock From bandits to {M}onte-{C}arlo {T}ree search: The optimistic
  principle applied to optimization and planning.
\newblock \emph{Foundations and Trends{\textregistered} in Machine Learning},
  7\penalty0 (1):\penalty0 1--129, 2014.

\bibitem[Nesterov(2004)]{Nes-04-ConvexOptimization}
Yurii Nesterov.
\newblock \emph{Introductory lectures on convex optimization: a basic course},
  volume~87 of \emph{Applied Optimization}.
\newblock Springer US, 2004.

\bibitem[Perevozchikov(1990)]{Per-90-OptimizationComplexity}
Alexander~G. Perevozchikov.
\newblock The complexity of the computation of the global extremum in a class
  of multi-extremum problems.
\newblock \emph{USSR Computational Mathematics and Mathematical Physics},
  30\penalty0 (2):\penalty0 28--33, 1990.

\bibitem[Picheny et~al.(2010)Picheny, Ginsbourger, Roustant, Haftka, and
  Kim]{Picheny.etal2010}
Victor Picheny, David Ginsbourger, Olivier Roustant, Raphael Haftka, and Nam-Ho
  Kim.
\newblock Adaptive designs of experiments for accurate approximation of a
  target region.
\newblock \emph{Journal of Mechanical Design}, 132\penalty0 (7), 2010.

\bibitem[Priem et~al.(2020)Priem, Gagnon, Chittick, Dufresne, Diouane, and
  Bartoli]{priem2020efficient}
Remy Priem, Hugo Gagnon, Ian Chittick, Stephane Dufresne, Youssef Diouane, and
  Nathalie Bartoli.
\newblock An efficient application of {Bayesian} optimization to an industrial
  {MDO} framework for aircraft design.
\newblock In \emph{AIAA AVIATION 2020 FORUM}, page 3152, 2020.

\bibitem[Rahimi et~al.(2004)Rahimi, Pon, Kaiser, Sukhatme, Estrin, and
  Srivastava]{rahimi2004adaptive}
Mohammad Rahimi, Richard Pon, William Kaiser, Gaurav Sukhatme, Deborah Estrin,
  and Mani Srivastava.
\newblock Adaptive sampling for environmental robotics.
\newblock In \emph{IEEE International Conference on Robotics and Automation,
  2004. Proceedings. ICRA'04. 2004}, volume~4, pages 3537--3544. IEEE, 2004.

\bibitem[Ranjan et~al.(2008)Ranjan, Bingham, and Michailidis]{Ranjan2008}
Pritam Ranjan, Derek Bingham, and George Michailidis.
\newblock Sequential experiment design for contour estimation from complex
  computer codes.
\newblock \emph{Technometrics}, 50\penalty0 (4):\penalty0 527--541, 2008.

\bibitem[Rigollet and Vert(2009)]{rigollet2009fast}
Philippe Rigollet and Regis Vert.
\newblock Optimal rates for plug-in estimators of density level sets.
\newblock \emph{Bernoulli}, 15\penalty0 (4):\penalty0 1154--1178, 2009.

\bibitem[Sacks et~al.(1989)Sacks, Welch, Mitchell, and Wynn]{DACE}
Jerome Sacks, William Welch, Toby Mitchell, and Henry Wynn.
\newblock Design and analysis of computer experiments.
\newblock \emph{Statistical science}, pages 409--423, 1989.

\bibitem[Santner et~al.(2003)Santner, Williams, Notz, and Williams]{TDACE}
Thomas Santner, Brian Williams, William Notz, and Brain Williams.
\newblock \emph{The design and analysis of computer experiments}, volume~1.
\newblock Springer, 2003.

\bibitem[Singh et~al.(2009)Singh, Scott, Nowak, et~al.]{singh2009adaptive}
Aarti Singh, Clayton Scott, Robert Nowak, et~al.
\newblock Adaptive {H}ausdorff estimation of density level sets.
\newblock \emph{The Annals of Statistics}, 37\penalty0 (5B):\penalty0
  2760--2782, 2009.

\bibitem[Tsybakov(1997)]{tsybakov1997nonparametric}
Alexandre~B. Tsybakov.
\newblock On nonparametric estimation of density level sets.
\newblock \emph{The Annals of Statistics}, 25\penalty0 (3):\penalty0 948--969,
  1997.

\bibitem[Tsybakov(2009)]{Tsy-09-NonParametric}
Alexandre~B. Tsybakov.
\newblock \emph{Introduction to Nonparametric Estimation}.
\newblock Springer, 2009.

\bibitem[Vazquez and Bect(2009)]{Vazquez.Bect2009}
Emmanuel Vazquez and Julien Bect.
\newblock A sequential {B}ayesian algorithm to estimate a probability of
  failure.
\newblock \emph{IFAC Proceedings Volumes}, 42\penalty0 (10):\penalty0 546--550,
  2009.

\bibitem[Wainwright(2019)]{wainwright2019high}
Martin Wainwright.
\newblock \emph{High-Dimensional Statistics: A Non-Asymptotic Viewpoint}.
\newblock Cambridge Series in Statistical and Probabilistic Mathematics.
  Cambridge University Press, 2019.

\bibitem[Wang et~al.(2016)Wang, Lin, and Li]{Wang.etal2016}
Hongqiao Wang, Guang Lin, and Jinglai Li.
\newblock {G}aussian process surrogates for failure detection: A {B}ayesian
  experimental design approach.
\newblock \emph{Journal of Computational Physics}, 313:\penalty0 247--259,
  2016.

\end{thebibliography}
\bibliographystyle{plainnat}

\clearpage

\onecolumn

\hsize\textwidth
  \linewidth\hsize \toptitlebar {\centering
  {\Large\bfseries The Sample Complexity of Level Set Approximation: \\
Supplementary Materials} \\[0.5cm] {\bfseries Fran\c{c}ois Bachoc \hspace{1cm} Tommaso R. Cesari \hspace{1cm}  S\'ebastien Gerchinovitz \par}}
 \bottomtitlebar \vskip 0.2in

\pagestyle{empty}

\appendix

\section{SAMPLE COMPLEXITY: FORMAL DEFINITIONS}
\label{sec:def}

We provide formal definitions for the notions of deterministic algorithm, sample complexity, and rate-optimal algorithm.

We first precisely define deterministic algorithms that query values of functions sequentially and rely only on this information to build approximations of their level sets (sketched in Online Protocol~\ref{alg:protocol-det-algo}).
The behavior of any such algorithm is completely determined by a pair $( \fhi, \psi )$, where 
$\fhi = (\fhi_\rmn)_{\rmn\in \bbN}$ is a sequence of functions $\fhi_\rmn\colon \bbR^{\rmn-1} \to \dcube$ mapping the $\rmn-1$ previously observed values $f(\bx_1),\ldots,f(\bx_{n-1})$ to the next query point $\bx_\rmn$, and
$\psi = (\psi_\rmn)_{\rmn\in \bbN}$ is a sequence of functions $\psi_\rmn\colon \bbR^\rmn \to \bcb{ \text{subsets of } \dcube }$ mapping the $\rmn$ currently known values $f(\bx_1),\ldots,f(\bx_n)$ to an approximation $\rmS_\rmn$ of the target level set.

We can now define the notion of sample complexity, which corresponds to the smallest number of queries after which the outputs $S_n$ of an algorithm are all $\e$-approximations of the level set $\fa$ (recall Definition~\ref{def:eps-approx} in the Introduction).
\begin{definition}[Sample complexity]
\label{n:smallest-n-queries}
For all functions $\rmf\colon \dcube \to \bbR$, all levels $\rma \in \bbR$, any deterministic algorithm $\rmA$, and any accuracy $\e > 0$, we denote by $\fn (\rmf,\rmA,\e,\rma)$ the smallest number of queries to $\rmf$ that $\rmA$ needs in order for its output sets $\Sn$ to be $\e$-approximations of the level set $\fa$ for all $\rmn\ge \fn (\rmf,\rmA,\e,\rma)$, i.e.,
\begin{equation}
    \label{e:min-n-queries}
    \fn (\rmf,\rmA,\e,\rma)
:=
    \inf\bigl\{ \rmn' \in \bbN : \forall \rmn \ge \rmn',
    \rmS_\rmn \text{ is an } \e \text{-approximation of } \{\rmf=\rma\} \bigr\}\;.
\end{equation}
We refer to $\fn (\rmf,\rmA,\e,\rma)$ as the \emph{sample complexity} of~$\rmA$ (for the $\e$-approximation of $\fa$).
\end{definition}
We can now define \nearo{} algorithms rigorously.
At a high-level, they output the tightest (up to constants) approximations of level sets that can possibly be achieved by deterministic algorithms.
\begin{definition}[\Nearo{} algorithm]
\label{d:near-opt}
For any level $\rma\in \bbR$ and some given family $\cF$ of real-valued functions defined on $\dcube$, we say that a deterministic algorithm $\rmA$ is \emph{\nearo} (for level $\rma$ and family $\cF$) if, in the worst-case, it needs the same number of queries (up to constants) of the best deterministic algorithm in order to output approximations of level sets within any given accuracy,
i.e., if there exists a constant $\kappa = \kappa(\rma, \cF) \ge 1$, depending only on $\rma$ and $\cF$, such that, for all $\e > 0$,
\begin{equation}
    \label{e:near-opt}
    \sup_{\rmf\in \cF} \fn (\rmf,\rmA,\e,\rma)
\le
    \kappa \inf_{\rmA' \in \cA} \sup_{\rmf \in \cF} \fn (\rmf,\rmA',\e,\rma) \;,
\end{equation}
where $\cA$ denotes the set of all deterministic algorithms.
\end{definition}

\section{USEFUL INEQUALITIES ABOUT PACKING AND COVERING NUMBERS}
\label{sec:lemmaPacking}

For all $\rmr>0$, the $\rmr$-\emph{covering number} $\cM(\rmE,\rmr)$ of a bounded subset $\rmE$ of $\bbR^\rmd$ (with respect to the $\sup$-norm $\lno{\cdot}_\iop$) is the smallest cardinality of an $\rmr$-covering of $\rmE$, i.e.,
\[
    	\cM(\rmE,\rmr)
:=
	\min \bigl(
		\rmk\in\bbN  :  \exists \bx_1,\ld,\bx_\rmk \in \bbR^\rmd, \,
        \forall \bx \in \rmE, \, \exists \rmi\in\{1,\ld, \rmk\}, \, \lno{ \bx-\bx_\rmi }_\iop \le \rmr
	\bigr)
\]
if $\rmE$ is nonempty, zero otherwise.

Covering numbers and packing numbers \eqref{e:packing} are closely related. In particular, the following well-known inequalities hold (see, e.g., \citealt[Lemmas~5.5~and~5.7, with permuted notation of $\cM$ and~$\cN$]{wainwright2019high}).\footnote{The definition of $r$-covering number of a subset $A$ of $\bbR^d$ implied by \citep[Definition~5.1]{wainwright2019high} is slightly stronger than the one used in our paper, because elements $\rmx_1, \ld, \rmx_\rmk$ of $r$-covers belong to $A$ rather than just $\bbR^\rmd$. Even if we do not need it for our analysis, Inequality~(\ref{eq:upper-convering}) holds also in this stronger sense.}
\begin{lemma}
\label{lem:packcover}
For any subset $\rmE$ of $\dcube$ and any real number $\rmr>0$,
\begin{equation}
\label{eq:wainwright}
	\cN(\rmE,2\rmr)
\le
	\cM(\rmE,\rmr)
\le
	\cN(\rmE,\rmr)\;.
\end{equation}
Furthermore, for all $\delta>0$ and all $\rmr>0$, if $\rmB(\delta) = \bcb{ \bx \in \bbR : \lno{ \bx }_\iop \le \delta }$,
\begin{equation}
\label{eq:upper-convering}
	\cM\brb{\rmB(\delta),\rmr}
\le
	\lrb{1 + 2 \frac{\delta}{\rmr} \bbI_{\rmr<\delta}}^\rmd\;.
\end{equation}
\end{lemma}

We now state a known lemma about packing numbers at different scales.

\begin{lemma}
\label{lm:pack-r1-r2}
For any subset $\rmE$ of $\dcube$ and any real numbers $\rmr_1,\rmr_2>0$,
\[
	\cN(\rmE,\rmr_1)
\le
	\lrb{1+ 4\frac{\rmr_2}{\rmr_1}\bbI_{\rmr_2>\rmr_1}}^\rmd \times \cN(\rmE,\rmr_2) \;.
\]
\end{lemma}

\begin{proof}
We can assume without loss of generality that $\rmE$ is nonempty and that $\rmr_1<\rmr_2$. Then,
\begin{align*}
	\cN(\rmE,\rmr_1)
& \le
	\cM(\rmE,\rmr_1/2)
	\tag{by (\ref{eq:wainwright})}\\
& \le
	\cM(\rmE,\rmr_2) \times \cM \brb{\rmB(\rmr_2), \rmr_1/2}
	\tag{see below}\\
& \le 
	\cN(\rmE,\rmr_2) \times \cM \brb{\rmB(\rmr_2), \rmr_1/2}
	\tag{by (\ref{eq:wainwright})}\\
& \le
	\cN(\rmE,\rmr_2) \times \lrb{1+\frac{4\rmr_2}{\rmr_1}}^\rmd \; .
	\tag{by (\ref{eq:upper-convering})}
\end{align*}
The second inequality is obtained by building the $\rmr_1/2$-covering of $\rmE$ in two steps. First, we cover $\rmE$ with balls of radius $\rmr_2$. Second, we cover each ball of the first cover with balls of radius $\rmr_1/2$. 
\end{proof}

The next lemma upper bounds the packing number of the unit \hyperc{} in the sup-norm, at all scales $\rmr$.

\begin{lemma}
\label{l:pack-unit-cube}
For any positive real number $\rmr > 0$, the $\rmr$-packing number of the unit cube in the sup-norm satisfies
\[
    \cN \brb{ \dcube, \, \rmr } 
\le 
    \lrb{ \lfl{ \frac{1}{\rmr} } + 1 }^\rmd \;.
\]
\end{lemma}

\begin{proof}
Since the diameter (in the $\sup$-norm $\lno{\cdot}_\iop$) of the unit \hyperc{} is $1$, if $\rmr \ge 1$, then the packing number is
$
    \cN \brb{ \dcube, \, \rmr } 
=
    1
\le
    \lrb{ \lfl{ \nicefrac{1}{\rmr} } + 1 }^\rmd
$.
Consider now the case $\rmr < 1$. 
Let $\rho := 1 - \lfl{\nicefrac{1}{\rmr}} \rmr \in [0, \rmr)$ and $\rmG$ be the set of $\rmr$-equispaced points $\bcb{ \nicefrac{\rho}{2}, \ \nicefrac{\rho}{2} + \rmr, \ \nicefrac{\rho}{2} + 2\rmr, \ \ld, \ \nicefrac{\rho}{2} + \lfl{ \nicefrac{1}{\rmr} }\rmr }^\rmd$.
Note that each point in $\dcube$ is at most $(\nicefrac{\rmr}{2})$-away from a point in $\rmG$ (in the sup-norm), i.e., $\rmG$ is an $(\nicefrac{\rmr}{2})$-covering of $\dcube$. We can thus use~\eqref{eq:wainwright} in Lemma~\ref{lem:packcover} at scale $\nicefrac{\rmr}{2}$ so see that $\cN \brb{ \dcube, \, \rmr } \le \cM \brb{ \dcube, \, \nicefrac{\rmr}{2} } \le \labs{\rmG}  = \brb{ \lfl{ \nicefrac{1}{\rmr} } + 1 }^\rmd$.
\end{proof}

The next lemma upper bounds the $\rmr$-packing number (in the sup-norm) of an inflated level set at scale $\rmr$. 

\begin{lemma}
\label{l:trivial-nls}
For any function $\rmf \colon \dcube \to \bbR$ and all scales $\rmr \in (0,1)$,
\[
    \cN \Brb{ \bcb{ \labs{ \rmf - \rma } \le \rmr }, \; \rmr }
\le  
    2^\rmd \, \lrb{\frac{1}{\rmr}}^\rmd \;.
\]
\end{lemma}

\begin{proof}
Let $\rmf \colon \dcube \to \bbR$ be an arbitrary function and $\rmr \in (0,1)$ any scale.
By the monotonicity of the packing number ($\rmE \s \rmF$ implies $\cN(\rmE,\rmr) \le \cN(\rmE,\rmr)$ by definition of packing number ---Definition~\ref{d:packing-number}) and the previous lemma (Lemma~\ref{l:pack-unit-cube}), we get
\[
    \cN \Brb{ \bcb{ \labs{ \rmf - \rma } \le \rmr }, \; \rmr }
\le
    \cN \brb{ \dcube, \, \rmr } 
\le
    \lrb{ \frac{1}{r} + 1 }^\rmd
\le
    \lrb{ \frac{1}{r} + \frac{1}{r} }^\rmd
\le
    2^\rmd \, \lrb{ \frac{1}{r} }^\rmd \;.
\]
\end{proof}

\section{MISSING PROOFS OF SECTION~\ref{s:hold-bi}}
\label{s:proofs-ba}

We now provide the missing proof of a claim we made in the proof of Theorem~\ref{t:general-packing}.

\begin{claim}
\label{cl:main-thm}
Under the assumptions of Theorem~\ref{t:general-packing}, let $\delta\in(0,1)$ and $\rmi \in \bbN$.
Then, the family of \hyperc{}s $\cC_\rmi$ maintained by Algorithm~\ref{alg:bi} can be partitioned into $2^\rmd$ subfamilies $\cC_{\rmi} (1), \ld, \cC_{\rmi} (2^\rmd)$ with the property that for all $\rmk \in \{1,\ld,2^{\rmd}\}$ and all $\rmC,\rmC' \in \cC_\rmi(\rmk)$, $\rmC\neq\rmC'$, we have $\inf_{\bx \in \rmC, \by \in \rmC'} \lno{ \bx - \by }_\iop > \delta \, 2^{-\rmi}$
\end{claim}

\begin{proof}
We build our partition by induction.
For a two-dimensional picture, see Figure~\ref{fig:damier}.
\begin{figure}
    \centering
    \begin{tikzpicture}
    \fill[orange] (0,0) rectangle (1,1);
    \begin{scope}
        \clip (0,0.5) rectangle (0.5,1);
        \foreach \x in {0,...,4}
        {
            \draw[yellow,line width=2pt] ({-0.1-0.2*\x}, {0}) -- ({1.9-0.2*\x}, {2});
        }
    \end{scope}
    \foreach \x in {0,...,3}
    {
        \foreach \y in {0,...,3}
        {
            \ifthenelse
            {
                \x = 0 \AND \y = 0
            }
            {
            }
            {
                \fill[yellow] (\x, {\y + 0.5}) rectangle ({\x + 0.5}, {\y + 1});
            }
            \draw[gray] (\y,{\x+0.5}) -- ({\y+1},{\x+0.5});
            \draw[gray] ({\x+0.5},\y) -- ({\x+0.5},{\y+1});
            \draw ({\x + 0.25}, {\y + 0.75}) node {$1$};
            \draw ({\x + 0.75}, {\y + 0.75}) node {$2$};
            \draw ({\x + 0.25}, {\y + 0.25}) node {$3$};
            \draw ({\x + 0.75}, {\y + 0.25}) node {$4$};
        }
    }
    \foreach \x in {1,...,3}
    {
        \draw (0,\x) -- (4,\x);
        \draw (\x,0) -- (\x,4);
    }
    \draw (0,0) rectangle (4,4);
    \end{tikzpicture}
    \caption{Constructing the partition when $\rmd=2$. In orange, the original enumeration $\rmA_0$. In yellow, the family $\cC_\rmi(1)$.}
    \label{fig:damier}
\end{figure}
Denote the elements of the standard basis of $\bbR^\rmd$ by $\be_1, \ld, \be_\rmd$.
For any $\bx \in \dcube$ and all $\rmE \s \dcube$, we denote by $\rmE + \bx$ the \mink{} sum $\bcb{ \by + \bx : \bx \in \rmE }$.
Let $\rmE$ be the collection of all the \hyperc{}s obtained by partitioning $\dcube$ with a standard uniform grid with step size $2^{-\rmi}$, i.e., $\rmE := \bcb{ [0,2^{-\rmi}]^d + \sum_{\rmk = 1}^\rmd \rmr_\rmk \, 2^{-\rmi} \, \be_\rmk: \rmr_1, \ld, \rmr_\rmk \in \{ 0,1, \ld,  2^{\rmi} - 1 \} }$.

Consider the family $\rmA_0$ containing the \hyperc{} $[0,2^{-\rmi}]$ and all other \hyperc{}s of $\rmE$ adjacent to it;
formally, $\rmA_0 := \bcb{ [0,2^{-\rmi}]^\rmd + \sum_{\rmk = 1}^\rmd \rmr_\rmk \, 2^{-\rmi} \, \be_\rmk : \rmr_1,\ld,\rmr_\rmd \in \{ 0,1 \} }$.
Assign to each of the $2^\rmd$ \hyperc{}s in $\rmA_0$ a distinct number between $1$ and $2^\rmd$.
Fix any $\rmk \in \{0, \ld, \rmd-1\}$.
For each \hyperc{} $\rmC \in \rmA_\rmk$, proceeding in the positive direction of the $\rmx_{\rmk +1}$ axis, assign the same number as $\rmC$ to every other \hyperc{} in $\rmE$; formally, assign the same number as $\rmC$ to all \hyperc{}s in $\bcb{ \rmC + 2 \, \rmr \, 2^{-\rmi} \, \be_{\rmk + 1} :  \rmr \in \{1, \ld 2^{\rmi - 1} - 1\} }$.
Denote by $\rmA_{\rmk +1}$ the collection of all \hyperc{}s that have been assigned a number so far.
By construction, $\rmA_{\rmd}$ coincides with the whole $\rmE$ and consists of $2^\rmd$ distinct subfamilies of \hyperc{}s, each containing only \hyperc{}s that have been assigned the same number.
For any number $\rmk \in \{1,\ld, 2^\rmd\}$, we denote by $\cC_\rmi(\rmk)$ the subfamily of all \hyperc{}s numbered with $\rmk$.
Fix any $\rmk \in \{1,\ld, 2^\rmd\}$.
By construction, each $\rmC \in \cC_\rmi(\rmk)$ contains no adjacent \hyperc{}s.
Thus, the smallest distance between two distinct \hyperc{}s $\rmC,\rmC' \in \cC_\rmi(\rmk)$ is $\inf_{\bx \in \rmC, \by \in \rmC'} \lno{\bx - \by}_\iop \ge 2^{-\rmi} > \delta \, 2^{-\rmi}$ for all $\delta \in (0,1)$.
\end{proof}

\section{MISSING PROOFS OF SECTION~\ref{s:holder}}
\label{s:proofs-holder}

In this section, we prove Theorem~\ref{t:bah-rate} of Section~\ref{s:holder}. The proof is divided into two parts: one for the upper bound, one for the lower bound. Each time, we restate the corresponding result to ease readability.

\subsection{Upper Bound}

\begin{proposition}[Theorem~\ref{t:bah-rate}, upper bound]
\label{p:upper-bound-bah}
Consider the \bih{} algorithm run with input $\rma,\rmc,\gamma$. 
Let $\rmf \colon \dcube \to \bbR$ be an arbitrary $(\rmc,\gamma)$-\hold{} function with level set $\fa \neq \varnothing$.
Fix any accuracy $\e > 0$. 
Then, for all 
\[
    \rmn 
>
    \kappa \,
    \frac{ 1 }{ \e^{\rmd/\gamma} }\;,
    \qquad\text{where }
    \kappa := \brb{ 2^{\gamma/\rmd} 8^\gamma \, 2\, \rmc }^{\rmd/\gamma}\;,
\]
the output $\rmS_\rmn$ returned after the $\rmn$-th query is an $\e$-approximation of $\fa$.
\end{proposition}
\begin{proof}
The proof is a simple application of Theorem~\ref{t:general-packing}, with $(\rmb,\beta) = (\rmc, \gamma)$.
Since we are assuming that the level set $\fa$ is nonempty, we only need to check that for all iterations $\rmi$ and all \hyperc{}s $\rmC' \in \cC'_\rmi$, the constant \interp{} $\rmg_{\rmC'} \equiv \rmf(\bc_{\rmC'})$ is a $(\rmc,\gamma)$-accurate approximation of $\rmf$ on $\rmC'$.
For any iteration $\rmi$ and all \hyperc{}s $\rmC' \in \cC'_\rmi$, 
we have that 
\[
    \sup_{\bx \in \rmC'} \babs{ \rmg_{\rmC'}(\bx) - \rmf(\bx)} 
= 
    \sup_{\bx \in \rmC'} \babs{ \rmf(\bc_{\rmC'}) - \rmf(\bx) }
\le 
    \rmc \, 2^{- \gamma \rmi} \;,
\]
by definition of $\rmg_{\rmC'}$, the $(\rmc,\gamma)$-\hold{}ness of $\rmf$, and the fact that the diameter of all \hyperc{}s $\rmC' \in \cC'_\rmi$ (in the $\sup$-norm) is $2^{-\rmi}$.
Thus, Theorem~\ref{t:general-packing} implies that for all $\rmn > \rmn(\e)$, the output $\rmS_\rmn$ returned after the $\rmn$-th query is an $\e$-approximation of $\fa$ where $\rmn(\e)$ is
\begin{equation}
    \label{e:bah-bouhd}
    4^\rmd \sum_{\rmi=0}^{\rmi(\e)-1} \lim_{\delta \to 1^-} \cN \Brb{ \bcb{ \labs{\rmf - \rma} \le 2 \, \rmc \, 2^{- \gamma \rmi} }, \ \delta \, 2^{-\rmi} }
\end{equation}
and $\rmi(\e) := \lce{ (\nicefrac{1}{\gamma}) \log_2 (\nicefrac{2\rmc}{\e}) }$.
If $\e \ge 2 \rmc$, than the sum in \eqref{e:bah-bouhd} ranges from $0$ to a \emph{negative} value, thus $\rmn(\e) = 0$ by definition of sum over an empty set and the result is true with $\kappa = 0$. 
Assume then that $\e < 2 \rmc$ so that the sum in \eqref{e:bah-bouhd} is not trivially zero.
Upper-bounding, for any $\delta\in(0,1)$ and all $\rmi\ge 0$,
\[
    \cN \Brb{ \bcb{ \labs{\rmf - \rma} \le 2 \, \rmc \, 2^{-\gamma \rmi} }, \ \delta \, 2^{-\rmi} }
\le
    \cN \brb{ \dcube, \ \delta \, 2^{-\rmi} }
\overset{(\dagger)}{\le}
    \brb{ 2^{\rmi}/\delta + 1 }^\rmd 
\le
    (\nicefrac{2}{\delta})^\rmd \, 2^{\rmd \rmi}    
\]
(for completeness, we include a proof of the known upper bound $(\dagger)$ in Section~\ref{sec:lemmaPacking}, Lemma~\ref{l:pack-unit-cube}) and recognizing the geometric sum below, we can conclude that
\begin{align*}
    \rmn(\e)
& 
\le
    8^\rmd 
    \, \sum_{\rmi=0}^{\lce{ (\nicefrac{1}{\gamma}) \log_2 (\nicefrac{2\rmc}{\e}) }-1} 
    \brb{ 2^{\rmd} }^\rmi
\\
&
=
    8^\rmd \,
    \frac{2^{\rmd \lce{ (\nicefrac{1}{\gamma}) \log_2 (\nicefrac{2\rmc}{\e}) }} -1 }{2^{\rmd} -1}
\\
&
\le
    8^\rmd \,
    \frac{2^{\rmd \lrb{ (\nicefrac{1}{\gamma}) \log_2 (\nicefrac{2\rmc}{\e}) +1 }} }{2^{\rmd} - \brb{ \nicefrac{2^\rmd}{2} }}
=
    2 \, 8^\rmd \, ( 2\, \rmc )^{\rmd/\gamma} \,
    \frac{ 1 }{ \e^{\rmd/\gamma} } \;.
\end{align*}
\end{proof}

\subsection{Lower Bound}

In this section, we prove our lower bound on the worst-case sample complexity of \hold{} functions.
We begin by stating a simple known lemma on \emph{bump} functions.
Bump functions are a standard tool to build lower bounds in nonparametric regression (see, e.g., \cite[Theorem~3.2]{GyKoKrWa-02-DistributionFreeNonparametric}, whose construction we also adapt for our following result and Proposition~\ref{p:lower-grad}).

\begin{lemma}
\label{l:bump}
Fix any amplitude $\alpha>0$, a step-size $\eta \in (0,\nicefrac{1}{4}]$, let $\rmZ := \{0,2\eta, \ldots, \lfl{\nicefrac{1}{2\eta}} 2\eta \}^\rmd \s \dcube$, and fix an arbitrary $\bz = (\rmz_1, \ld, \rmz_\rmd) \in \rmZ$. 
Consider the \emph{bump} functions
\begin{align*}
        \wt{\rmf} \colon \bbR & \to [0,1]
    &
        \rmf_{\alpha,\eta,\bz} \colon \bbR^\rmd & \to \bbR \nonumber
    \\
        \rmx & \mapsto \wt\rmf(\rmx) 
    := 
        \begin{cases}
            \displaystyle{ \exp \lrb{ \frac{-\rmx^2}{1-\rmx^2} } }
                & \text{if } \rmx \in (-1,1)
            \\
            0
                & \text{otherwise}\;,
        \end{cases}
        \qquad
    &
        \bx & \mapsto \rmf_{\alpha,\eta,\bz}(\bx) 
    := 
        \alpha \prod_{\rmj=1}^\rmd \wt\rmf \lrb{\frac{\rmx_\rmj - \rmz_\rmj}{\eta}} \;.
\end{align*}
Then $\wt\rmf$ is $3$-\lip{} 
and $\rmf_{\alpha,\eta,\bz}$ satisfies:
\begin{enumerate}[topsep = 0pt, parsep = 0pt, itemsep = 0pt]
    \item $\rmf_{\alpha,\eta,\bz}$ is infinitely differentiable;
    \item $\rmf_{\alpha,\eta,\bz}(\bx) \in [0,\alpha)$ for all $\bx \in \bbR^\rmd\m\{\bz\}$, and $\rmf_{\alpha,\eta,\bz}(\bz) = \alpha$;
    \item $\{ \rmf_{\alpha,\eta,\bu_1} > 0 \} \cap \{ \rmf_{\alpha,\eta,\bu_2} > 0 \} = \varnothing$ for any two distinct $\bu_1,\bu_2 \in \rmZ$;
    \item $\lno{\bx-\by}_\iop \le 2\eta$ for all $\bx,\by$ in the closure $\overline{\{ \rmf_{\alpha,\eta,\bz} > 0 \}}$ of $\{ \rmf_{\alpha,\eta,\bz} > 0 \}$ and all $\bz \in \rmZ$.
\end{enumerate}
\end{lemma}

The proof is a straightforward verification and it is therefore omitted.
We now prove our worst-case lower bound for \hold{} functions.

\begin{proposition}[Theorem~\ref{t:bah-rate}, lower bound]
\label{p:lower-hold}
Fix any level $\rma \in \bbR$, any two \hold{} constants $\rmc>0$, $\gamma \in (0,1]$, and an arbitrary accuracy $\e \in \brb{ 0, \, \rmc/(3\rmd 2^\gamma) }$.
Let $\rmn<\kappa / \e^{\rmd/\gamma}$ be a positive integer, where $\kappa := ( \nicefrac{\rmc}{12\rmd} )^{\rmd/\gamma}$.
For each deterministic algorithm $\rmA$ there is a $(\rmc,\gamma)$-\hold{} function $\rmf$ such that, if $\rmA$ queries $\rmn$ values of $\rmf$, then its output set $\rmS_\rmn$ is not an $\e$-approximation of $\fa$.
This implies in particular that (recall Definition~\ref{n:smallest-n-queries}),
\[
    \inf_\rmA \sup_\rmf \fn (\rmf,\rmA,\e,\rma)
\ge
    \kappa \, \frac{1}{\e ^{\rmd/\gamma}}\;,
\]
where the $\inf$ is over all deterministic algorithms $\rmA$ and the $\sup$ is over all $(\rmc,\gamma)$-\hold{} functions $\rmf$.
\end{proposition}

Note that the leading constant $\kappa = ( \nicefrac{\rmc}{12\rmd} )^{\rmd/\gamma}$ in our lower bound decreases quickly with the dimension $\rmd$.
Though we keep our focus on sample complexity rates, there are ways to improve the multiplicative constants appearing in our lower bounds. For instance, in the proof below, a larger constant $\kappa := ( \nicefrac{1}{4} \, ( \nicefrac{\rmc}{2} )^{1/\gamma} )^{\rmd}$ can be obtained by replacing bump functions with \emph{spike} functions 
$
    \bx 
    \mapsto 
    \bsb{ 2 \, \e -  \rmc \, \lno{ \bx - \bz }_\iop^\gamma }^+
$,
where $\rmx \mapsto [\rmx]^+ := \max\{\rmx, 0\}$ denotes the positive part of $\rmx$.
We choose to use bump functions instead because they are well-suited for any smoothness (e.g., in Section~\ref{s:lower-gradho}, we will apply the same argument to \gradho{} functions).


\begin{proof}
The following construction is a standard way to prove lower bounds on sample complexity (for a similar example, see \citealt[Theorem~3.2]{GyKoKrWa-02-DistributionFreeNonparametric}).
Consider the set of bump functions $\{\rmf_{\bz}\}_{\bz \in \rmZ}$, where $\rmZ$ and $\rmf_{\bz} := \rmf_{\alpha,\eta,\bz}$ are defined as in Lemma~\ref{l:bump},\footnote{More precisely, $\rmf_{\bz}$ is the restriction of $\rmf_{\alpha,\eta,\bz}$ to $\dcube$.} for $\alpha := 2 \e$ and some $\eta\in(0,\nicefrac{1}{4}]$ to be selected later.
Fix an arbitrary $\bu = (\rmu_1, \ld,\rmu_\rmd) \in \rmZ$. 
We show now that $\rmf_{\bu}$ is $(\rmc,\gamma)$-\hold{}, for a suitable choice of $\eta$.
For all $\bx,\by$ in the closure $\overline{\{\rmf_{\bu} > 0\}}$ of $\{\rmf_{\bu} > 0\}$, Lemma~\ref{l:bump} gives
\begin{align*}
    \babs{ \rmf_{\bu}(\bx) - \rmf_{\bu}(\by) }
&
\le
    2\e \sum_{j=1}^\rmd \labs{ \wt\rmf \lrb{\frac{\rmx_\rmj - \rmu_\rmj}{\eta}} - \wt\rmf \lrb{\frac{\rmy_\rmj - \rmu_\rmj}{\eta}}}
\le
    2\e \sum_{j=1}^\rmd 3 \labs{ \frac{\rmx_\rmj - \rmu_\rmj}{\eta} - \frac{\rmy_\rmj - \rmu_\rmj}{\eta}}
\le
    \frac{6 \e \rmd }{\eta} \lno{ \bx - \by }_\iop
\\
&
=
    \frac{6 \e \rmd }{\eta} \lno{ \bx - \by }_\iop^{1-\gamma} \lno{ \bx - \by }_\iop^\gamma
\le
    \frac{6 \e \rmd }{\eta} (2\eta)^{1-\gamma} \lno{ \bx - \by }_\iop^\gamma
=
    \frac{6 \e \rmd 2^{1-\gamma}}{\eta^\gamma
    } \lno{ \bx - \by }_\iop^\gamma
    \;,
\end{align*}
where the first inequality follows by applying $\rmd$ times the elementary consequence of the triangular inequality 
$
    \babs{ \rmg_1(\bx_1) \rmg_2(\bx_2) - \rmg_1(\by_1) \rmg_2(\by_2) } 
\le 
    \max\bcb{ \lno{\rmg_1}_\iop, \lno{\rmg_2}_\iop } \brb{ \labs{ \rmg_1(\bx_1) - \rmg_1(\by_1) } + \labs{ \rmg_2(\bx_2) - \rmg_2(\by_2) } }
$, 
which holds for any two bounded functions $\rmg_\rmi \colon \rmE_\rmi \s \bbR^{\rmd_\rmi} \to \bbR$ ($\rmd_\rmi \in \bbN, \rmi \in \{1,2\}$).
If $\bx', \by' \notin \{\rmf_{\bu} > 0\}$, then $\rmf_{\bu}(\bx') = 0 = \rmf_{\bu}(\by')$, hence $\babs{ \rmf_{\bu}(\bx') - \rmf_{\bu}(\by') } = 0$. 
Finally, if $\bx \in \{\rmf_{\bu} > 0\}$ but $\by' \notin \overline{\{\rmf_{\bu} > 0\}}$, let $\by$ be the unique\footnote{This follows from two simple observations. First, since $\rmf_{\bu}$ is continuous, the set $\{\rmf_{\bu} > 0\}$ is open, hence $\bx$ belongs to its interior. Second, $\{\rmf_{\bu} > 0\}$ is (the interior of) a \hyperc{}, therefore it is convex.} point in the intersection of the segment $[\bx,\by']$ and the boundary $\partial{\{\rmf_{\bu} > 0\}}$ of $\{\rmf_{\bu} > 0\}$; since $\rmf_{\bu}$ vanishes at the boundary of $\{\rmf_{\bu} > 0\}$, then $\rmf_{\bu}(\by)=\rmf_{\bu}(\by')$, therefore $\babs{ \rmf_{\bu}(\bx) - \rmf_{\bu}(\by') } = \babs{ \rmf_{\bu}(\bx) - \rmf_{\bu}(\by) }$ and we can reapply the argument above for $\bx,\by$ now both in $\overline{\{\rmf_{\bu} > 0\}}$, obtaining
\[
    \babs{ \rmf_{\bu}(\bx) - \rmf_{\bu}(\by') }
=
    \babs{ \rmf_{\bu}(\bx) - \rmf_{\bu}(\by) }
\le
    \frac{6 \e \rmd 2^{1-\gamma}}{\eta^\gamma
    } \lno{ \bx - \by }_\iop^\gamma
\le
    \frac{6 \e \rmd 2^{1-\gamma}}{\eta^\gamma
    } \lno{ \bx - \by' }_\iop^\gamma\;,
\]
where the last inequality follows by $\lno{ \bx - \by }_\iop \le \lno{ \bx - \by' }_\iop$ and the monotonocity of $\rmx \mapsto \rmx^\gamma$ on $[0,\iop)$.
Thus, selecting 
$
    \eta = (6 \e \rmd 2^{1-\gamma}/\rmc)^{1/\gamma}
$ 
so that 
$
    6 \e \rmd 2^{1-\gamma}/\eta^\gamma = \rmc
$,
we obtain that $\rmf_{\bz}$ is $(\rmc,\gamma)$-\hold{} for all $\bz \in \rmZ$.
Moreover, by definition of $\rmZ$ (Lemma~\ref{l:bump}) and $\kappa$, we have that 
\[
    \labs{\rmZ} 
= 
    \lfl{ \frac{1}{2\eta} + 1}^\rmd 
\ge 
    \lrb{ \frac{1}{2\eta} }^\rmd
=
    \lrb{ \frac{1}{2 (6 \e \rmd 2^{1-\gamma}/\rmc)^{1/\gamma} } }^\rmd
=
    \lrb{ \frac{\rmc}{12 \rmd} }^{\rmd/\gamma}
    \frac{1}{\e^{\rmd/\gamma}}
=
    \kappa \frac{1}{\e^{\rmd/\gamma}}\;.
\]
Recall that the sets $\{\rmf_{\bz_1}>0\}$ and $\{\rmf_{\bz_2}>0\}$ are disjoint for distinct $\bz_1,\bz_2 \in \rmZ$ (Lemma~\ref{l:bump}).
Thus, consider an arbitrary deterministic algorithm and assume that only $\rmn < \kappa / \e^{\rmd/\gamma}$ values are queried.
By construction, there exists at least a $\bz \in \cP$ such that, if the algorithm is run for the level set $\{\rmf = 0\}$ of the constant function $\rmf \equiv 0$, no points are queried inside $\{\rmf_{\bz} > 0 \}$ (and being $\rmf$ constant, the algorithm always observes $0$ as feedback for the $\rmn$ evaluations).
Being deterministic, if the algorithm is run for the level set $\{\rmf_{\bz} = 0\}$ of $\rmf_{\bz}$ it will also query no points inside $\{\rmf_{\bz} > 0 \}$, observing only zeros for all the $\rmn$ evaluations.
Since either way, only zeros are observed, using again the fact that the algorithm is deterministic, it returns the same output set $\rmS_\rmn$ in both cases.
This set cannot be simultaneously an $\e$-approximation of both $\{\rmf = 0\}$ and $\{\rmf_{\bz} = 0\}$.
Indeed, for the first set we have that $\{\rmf = 0\} = \dcube = \{\rmf \le \e \}$.
Thus, if $\rmS_\rmn$ is an $\e$-approximation of $\{\rmf = 0\}$ it has to satisfy $\{\rmf = 0\} \s \rmS_\rmn \s \{\rmf \le \e \}$, which in turn gives $\rmS_\rmn = \dcube$.
On the other hand, $\max_{\bx \in \dcube} \rmf_{\bz}(\bx) = 2 \, \e$, which implies that $\{\rmf_{\bz} \le \e \}$ is \emph{properly} included in $\dcube$.
Hence, if $\rmS_\rmn = \dcube$ were also an $\e$-approximation of $\{\rmf_{\bz} = 0\}$, we would have that $\dcube = \rmS_\rmn \s \{\rmf_{\bz} \le \e \} \neq \dcube$, which yields a contradiction.
This concludes the proof of the first claim.
The second claim follows directly from the first part and Definition~\ref{n:smallest-n-queries}.
\end{proof}

\section{MISSING PROOFS OF SECTION~\ref{s:g-hold-biga}}
\label{s:proofs-gradho}

In this section, we present all missing proofs of our results in Section~\ref{s:holder}.
We restate them to ease readability.

\subsection{Upper Bound}

\begin{lemma*}[Lemma~\ref{l:interp-biga}]
Let $\rmf\colon \rmC' \to \bbR$ be a $(\rmc_1,\gamma_1)$-\gradho{} function, for some $\rmc_1>0$ and $\gamma_1\in (0,1]$.
Let $\rmC' \s \dcube$ be a \hyperc{} with diameter $\fd \in (0,1]$ and set of vertices $\rmV'$, i.e., $\rmC' = \prod_{\rmj=1}^\rmd \lsb{ \rmu_\rmj, \rmu_\rmj + \fd }$, for some $\bu := (\rmu_1,\ld,\rmu_\rmd) \in [0, 1 - \fd]^\rmd$, and $\rmV' = \prod_{\rmj=1}^\rmd \lcb{ \rmu_\rmj, \rmu_\rmj + \fd }$. 
The function 
\begin{align*}
	\rmh_{\rmC'}\colon \rmC'
&
	\to \bbR
\nonumber
\\[-2ex]
	\bx
&
	\mapsto 
	\sum_{\bv \in \rmV'}
	\rmf(\bv)
	\prod_{\rmj = 1}^{\rmd}
	\rmp_{\rmv_{\rmj}}(\rmx_{\rmj})\;,
\end{align*}
where
\[
    \rmp_{\rmv_{\rmj}}(\rmx_{\rmj})
:=
	\lrb{ 1 - \frac{\rmx_{\rmj} - \rmu_\rmj}{\fd} } 
	\, \bbI_{\rmv_\rmj = \rmu_\rmj}
	+ \frac{\rmx_\rmj - \rmu_\rmj}{\fd} 
	\, \bbI_{\rmv_\rmj = \rmu_\rmj + \fd} ,
\]
interpolates the $2^\rmd$ pairs $\bcb{ \lrb{ \bv, \rmf (\bv) } }_{\bv \in \rmV'}$ and it satisfies
\[
	\sup_{\bx \in \rmC'} \babs{ \rmh_{\rmC'}(\bx) - \rmf(\bx) } 
\le 
	\rmc_1 \rmd \, \fd^{1+\gamma_1}\;.
\]
\end{lemma*}

\begin{proof}
Up to applying the translation $\bx \mapsto \bx + \bu$, we can (and do) assume without loss of generality that $\bu = \bzero$. 
The \hyperc{} and its set of vertices then become $\rmC' = [0,\fd]^\rmd$ and $\rmV' = \{0,\fd\}^\rmd$ respectively. 
To verify that $\rmh_{\rmC'}$ interpolates the $2^{\rmd}$ pairs $\bcb{ \lrb{ \bv, \rmf (\bv) } }_{\bv \in \rmV'}$, note that by definition of $\rmh_{\rmC'}$, for any vertex
$
	\bw \in \rmV' 
= 
	\lcb{ 0, \fd }^\rmd
$,
we have
\[
	\rmh_{\rmC'}(\bw) 
= 
	\sum_{\bv\in \rmV'} \rmf(\bv) \prod_{\rmj=1}^\rmd \rmp_{\rmv_\rmj}(\rmw_\rmj)
= 
	\sum_{\bv\in \rmV'} \rmf(\bv) \prod_{\rmj=1}^\rmd \bbI_{\rmw_j = \rmv_j}
=
	\rmf(\bw) \;.
\]
To prove the inequality, for all $\rmk \in \{0,\ld,\rmd\}$, let $\propp{\rmk}$ be the property:
if an $\bx \in \rmC'$ has at most $\rmk$ components which are not in $\{0,\fd\}$, then it holds that $\babs{ \rmh_{\rmC'}(\bx) - \rmf(\bx) } \le \rmc_1 \rmk \, \fd^{1+\gamma_1}$.
To show that 
$
	\babs{ \rmh_{\rmC'}(\bx) - \rmf(\bx) } 
\le 
	\rmc_1 \rmd \, \fd^{1+\gamma_1}
$ 
for all $\bx \in \rmC'$ (therefore concluding the proof) we then only need to check that the property $\propp{\rmd}$ is true.
We do so by induction.
If $\rmk = 0$, then $\propp{0}$ follows by $\rmh_{\rmC'}$ being an \interp{} for $\bcb{ \lrb{ \bv, \rmf (\bv) } }_{\bv \in \rmV'}$.
Assume now that $\propp{\rmk}$ holds for $\rmk \in \{0 , \ld, \rmd-1 \}$. 
To prove $\propp{\rmk+1}$, fix an arbitrary $\bx := (\rmx_1,\ld,\rmx_\rmd) \in \rmC'$, assume that $\rmk+1$ components of $\bx$ are not in $\{0,\fd\}$ and let $\rmi\in\{1,\ld,\rmd\}$ be any one of them (i.e., $\rmx_\rmi \in (0,\fd)$).
Consider the two univariate functions
\begin{align*}
	\rmh_\rmi \colon [0,\fd] 
& 
	\to \bbR
\\
	\rmt 
& 
	\mapsto \rmh_{\rmC'}(\rmx_1,\ld,\rmx_{i-1},t,\rmx_{i+1},\ld,\rmx_{\rmd})\;,
\\
	\rmf_\rmi \colon [0,\fd]
& 
	\to \bbR
\\
	\rmt 
& 
	\mapsto \rmf(\rmx_1,\ld,\rmx_{i-1},t,\rmx_{i+1},\ld,\rmx_{\rmd})
    \;.
\end{align*}
Being $\rmh_\rmi$ linear (by definition of $\rmh_{\rmC'}$), we get
\begin{equation}
	\label{e:h-i-linear}
	\rmh_\rmi (\rmx_\rmi)
=
	\frac{\fd - \rmx_\rmi}{\fd} \rmh_\rmi(0)
	+
	\frac{\rmx_\rmi}{\fd} \rmh_\rmi(\fd)\;.
\end{equation}
Being $\rmf_\rmi$ continuous on $[0,\fd]$ and derivable on $(0,\fd)$ (by our assumptions on $\rmf$), the mean value theorem applied to $\rmf_\rmi$ on $[0,\rmx_\rmi]$ and $[0,\fd]$ respectively yields the existence of $\xi_1\in (0,\rmx_\rmi)$ and $\xi_2 \in (0,\fd)$ such that
\begin{align}
	\label{e:f-i-1}
	\rmf_\rmi(\rmx_\rmi)
&
	= \rmf_\rmi(0) + \rmf'_\rmi(\xi_1) \, \rmx_i \;,
\\
	\label{e:f-i-2}
	\rmf_\rmi(\fd)
&
	= \rmf_\rmi(0) + \rmf'_\rmi(\xi_2) \, \fd \;.
\end{align}
Putting everything together, we get that
\[
    \babs{ \rmh_{\rmC'}(\bx) - \rmf(\bx) }
= 
    \babs{ \rmh_{\rmi}(\rmx_i) - \rmf_\rmi(\rmx_i) }
\]
(by definition of $\rmh_\rmi$ and $\rmf_\rmi$).
By \eqref{e:h-i-linear} and \eqref{e:f-i-1}, the right-hand side is equal to
\[
    \labs{  
	\frac{\fd - \rmx_\rmi}{\fd} \rmh_\rmi(0)
	+
	\frac{\rmx_\rmi}{\fd} \rmh_\rmi(\fd)
	-
	\rmf_\rmi(0)
	-
	\rmf'_\rmi(\xi_1) \, \rmx_i
	}\;,
\]
which by the triangular inequality is at most
\[
    \labs{  
	\frac{\fd - \rmx_\rmi}{\fd} \rmh_\rmi(0)
	+
	\frac{\rmx_\rmi}{\fd} \rmh_\rmi(\fd)
	-
	\rmf_\rmi(0)
	-
	\rmx_\rmi \frac{\rmf_\rmi(\fd) - \rmf_\rmi(0)}{\fd}
	}
	+
	\labs{
	\rmx_\rmi \frac{\rmf_\rmi(\fd) - \rmf_\rmi(0)}{\fd}
	-
	\rmf'_\rmi(\xi_1) \, \rmx_i
	}\;.
\]
By \eqref{e:f-i-2}, this is equal to
\[
    \labs{  
	\frac{\fd - \rmx_\rmi}{\fd} 
		\brb{ \rmh_\rmi(0) - \rmf_\rmi(0) }
	+
	\frac{\rmx_\rmi}{\fd} 
		\brb{ \rmh_\rmi(\fd) - \rmf_\rmi(\fd) }
	}
	+
	\babs{
	\rmx_\rmi \, \rmf'_\rmi(\xi_2)
	-
	\rmf'_\rmi(\xi_1) \, \rmx_i
	}\;.
\]
Finally, using again the triangular inequality, we can further upper bound with
\[
	\frac{\fd - \rmx_\rmi}{\fd} 
	\underbrace{
	\Babs{  
		\brb{ \rmh_\rmi(0) - \rmf_\rmi(0) }
	}
	}_{\le \rmc_1 \rmk \fd^{1+\gamma_1}}
	+
	\frac{\rmx_\rmi}{\fd} 
	\underbrace{
	\Babs{
		\brb{ \rmh_\rmi(\fd) - \rmf_\rmi(\fd) }
	}
	}_{\le \rmc_1 \rmk \fd^{1+\gamma_1}}
	+
	\underbrace{
	\rmx_\rmi
	}_{\le \fd}
	\underbrace{
	\babs{
	\rmf'_\rmi(\xi_2)
	-
	\rmf'_\rmi(\xi_1)
	}
	}_{\le \rmc_1 \fd^{\gamma_1}}
	\le \rmc_1 (\rmk +1) \fd^{1+\gamma_1}\;,
\]
where on the last line, we applied property $\propp{\rmk}$ to the first two terms and we upper bounded the last one leveraging the $(\rmc_1,\gamma_1)$-\hold{}ness of the gradients of $\rmf$. 
This proves $\propp{k+1}$ and concludes the proof.
\end{proof}

\begin{proposition}[Theorem~\ref{t:biga-rate}, upper bound]
Consider the \biga{} algorithm (Algorithm~\ref{alg:bi-g}) run with input $\rma,\rmc_1,\gamma_1$. 
Let $\rmf \colon \dcube \to \bbR$ be an arbitrary $(\rmc_1,\gamma_1)$-\gradho{} function with level set $\fa \neq \varnothing$.
Fix any accuracy $\e > 0$.
Then, for all 
\[
    \rmn 
>
    \kappa \,
    \frac{ 1 }{ \e^{\rmd/(1+\gamma_1)} }\;,
    \text{ if }
    \kappa := \brb{ 2^{ 5 + 4 \gamma_1 + (1+\gamma_1)/\rmd} \, \rmc_1  \rmd }^{\rmd/\gamma}\;,
\]
the output $\rmS_\rmn$ returned after the $\rmn$-th query is an $\e$-approximation of $\fa$.
\end{proposition}

\begin{proof}
We proceed as in the proof of Theorem~\ref{t:bah-rate}.
Theorem~\ref{t:general-packing} implies that for all $\rmn > \rmn(\e)$, where $\rmi(\e) := \lce{ (\nicefrac{1}{(1+\gamma_1)}) \log_2 (\nicefrac{2\rmc_1\rmd}{\e}) }$ and $\rmn(\e)$ is
\[
    8^\rmd \, \sum_{\rmi=0}^{\rmi(\e)-1} \lim_{\delta\to 1^-} \cN \Brb{ \bcb{ \labs{\rmf - \rma} \le 2 \rmc_1 \rmd 2^{- (1+\gamma_1) \rmi} }, \delta 2^{-\rmi} }
\;,    
\]
the output $\rmS_\rmn$ returned after the $\rmn$-th query is an $\e$-approximation of $\fa$.
If $\e \ge 2 \rmc_1 \rmd$, than the sum in the definition of $\rmn(\e)$ ranges from $0$ to a \emph{negative} value, thus $\rmn(\e) = 0$ by definition of sum over an empty set and the result is true with $\kappa = 0$. 
Assume then that $\e < 2 \rmc_1 \rmd$ so that such sum is not trivially zero.
Upper-bounding, for any $\delta\in(0,1)$ and all $\rmi\ge 0$,
\[
    \cN \Brb{ \bcb{ \labs{\rmf - \rma} \le 2 \rmc_1 \rmd \,  2^{- (1 + \gamma_1) \rmi} }, \ \delta \, 2^{-\rmi} }
\le
    \cN \brb{ \dcube, \ \delta \, 2^{-\rmi} }
\overset{(\dagger)}{\le}
    \brb{ 2^{\rmi}/\delta + 1 }^\rmd 
\le
    (\nicefrac{2}{\delta})^\rmd \, 2^{\rmd \rmi}
\]
(for completeness, we include a proof of the known upper bound $(\dagger)$ in Section~\ref{sec:lemmaPacking}, Lemma~\ref{l:pack-unit-cube}) and recognizing the geometric sum below, we can conclude that
\begin{align*}
    \rmn(\e)
&
\le
    16^\rmd 
    \, \sum_{\rmi=0}^{\lce{ (\nicefrac{1}{(1+\gamma_1)}) \log_2 (\nicefrac{2\rmc_1\rmd}{\e}) }-1} 
    \brb{ 2^{\rmd} }^\rmi
\\
&
\le
    16^\rmd \,
    \frac{2^{\rmd \lrb{ (\nicefrac{1}{(1+\gamma_1)}) \log_2 (\nicefrac{2\rmc_1\rmd}{\e}) +1 } } }{2^{\rmd} - \brb{ \nicefrac{2^\rmd}{2} }}
\\
&
=
    2 \cdot 16^\rmd \, ( 2 \rmc_1 \rmd )^{\rmd/(1+\gamma_1)} \,
    \frac{ 1 }{ \e^{\rmd/(1+\gamma_1)} } \;.
\end{align*}
\end{proof}

\subsection{Lower Bound}
\label{s:lower-gradho}

We conclude the section by proving a matching lower bound. 
Similarly to Proposition~\ref{p:lower-hold}, we adapt some already known techniques from nonparametric regression (see, e.g., \citealt[Theorem~3.2]{GyKoKrWa-02-DistributionFreeNonparametric}).

\begin{proposition}[Theorem~\ref{t:biga-rate}, lower bound]
\label{p:lower-grad}
Fix any level $\rma \in \bbR$, any two \hold{} constants $\rmc_1>0$, $\gamma_1 \in (0,1]$, and any accuracy $\e \in \brb{ 0, \, \rmc_1/(132\rmd 2^{3+\gamma_1}) }$.
%
%
Let $\rmn<\kappa / \e^{\rmd/(1+\gamma_1)}$ be a positive integer, where $\kappa := \brb{ \rmc_1/(528\rmd)}^{\rmd/(1+\gamma_1)}$.
For each deterministic algorithm $\rmA$ there is a $(\rmc_1,\gamma_1)$-\gradho{} function $\rmf$ such that, if $\rmA$ queries $\rmn$ values of $\rmf$, then its output set $\rmS_\rmn$ is not an $\e$-approximation of $\fa$.
This implies in particular that (recall Definition~\ref{n:smallest-n-queries}),
\[
    \inf_\rmA \sup_\rmf \fn (\rmf,\rmA,\e,\rma)
\ge
    \kappa \, \frac{1}{\e ^{\rmd/(1+\gamma_1)}}\;,
\]
where the $\inf$ is over all deterministic algorithms $\rmA$ and the $\sup$ is over all $(\rmc_1,\gamma_1)$-\gradho{} functions~$\rmf$.
\end{proposition}

As we pointed out after Proposition~\ref{p:lower-hold}, the leading constant $\kappa = \brb{ \rmc_1/(528\rmd)}^{\rmd/(1+\gamma_1)}$ in our lower bound is small, and could likely be improved using smoothness-specific perturbations of the zero function, instead of the more universal bump functions.



\begin{proof}
The following construction is a standard way to prove lower bounds on sample complexity (for a similar example, see \citealt[Theorem~3.2]{GyKoKrWa-02-DistributionFreeNonparametric}).
Consider the set of bump functions $\{\rmf_{\bz}\}_{\bz \in \rmZ}$, where $\rmZ$ and $\rmf_{\bz} := \rmf_{\alpha,\eta,\bz}$ are defined as in Lemma~\ref{l:bump},\footnote{More precisely, $\rmf_{\bz}$ is the restriction of $\rmf_{\alpha,\eta,\bz}$ to $\dcube$.} for $\alpha := 2 \e$ and some $\eta\in(0,\nicefrac{1}{4}]$ to be selected later.
Fix an arbitrary $\bu = (\rmu_1,\ld,\rmu_\rmd) \in \rmZ$. 
We show now that $\rmf_{\bu}$ is $(\rmc_1,\gamma_1)$-\gradho{}, for a suitable choice of $\eta$.
This is sufficient to prove the result, following the same argument as in the proof of Proposition~\ref{p:lower-hold}.
Note first that for all $\rmi \in \{1,\ld,\rmd\}$ and any $\bx \in \dcube$, denoting by $\partial_\rmi$ the partial derivative with respect to the $\rmi$-th variable,
\[
    \partial_\rmi \rmf_{\bu} (\bx)
=
    \frac{2\e}{\eta} \wt\rmf'\lrb{\frac{\rmx_\rmi - \rmu_\rmi}{\eta}} \prod_{\substack{\rmj = 1\\\rmj\neq\rmi}}^\rmd \wt\rmf\lrb{\frac{\rmx_\rmj - \rmu_\rmj}{\eta}}\;.
\]
Hence, using the fact that $\wt\rmf$ is $3$-\lip{} (Lemma~\ref{l:bump}) and $22$-\gradlip{} (the latter can be done by checking that $\bno{\wt\rmf''}_\iop \le 22$),
for all $\rmi \in \{1,\ld,\rmd\}$ and any $\bx, \by \in \dcube$, we get
\begin{equation}
\label{e:bound-gradient}
    \babs{ \partial_\rmi \rmf_{\bu} (\bx) - \partial_\rmi \rmf_{\bu} (\by) }
\le
    \frac{2\e}{\eta}  3 \lrb{ 22 \labs{ \frac{\rmx_\rmi - \rmu_\rmi}{\eta} - \frac{\rmy_\rmi - \rmu_\rmi}{\eta} } + 3 \sum_{\substack{\rmj=1\\\rmj\neq\rmi}}^\rmd \labs{ \frac{\rmx_\rmj - \rmu_\rmj}{\eta} - \frac{\rmy_\rmj - \rmu_\rmj}{\eta} } }
\le
    \frac{132 \e \rmd }{\eta^2} \lno{\bx - \by}_\iop\;,
\end{equation}
where the first inequality follows by applying $\rmd$ times the elementary consequence of the triangular inequality 
$
    \babs{ \rmg_1(\bx_1) \rmg_2(\bx_2) - \rmg_1(\by_1) \rmg_2(\by_2) } 
\le 
    \max\bcb{ \lno{\rmg_1}_\iop, \lno{\rmg_2}_\iop } \brb{ \labs{ \rmg_1(\bx_1) - \rmg_1(\by_1) } + \labs{ \rmg_2(\bx_2) - \rmg_2(\by_2) } }
$, 
which holds for any two bounded functions $\rmg_\rmi \colon \rmE_\rmi \s \bbR^{\rmd_\rmi} \to \bbR$ ($\rmd_\rmi \in \bbN, \rmi \in \{1,2\}$), and then using the \lip{}ness of $\wt\rmf$ and $\wt\rmf'$.
Similarly to Proposition~\ref{p:lower-hold}, to prove that $\rmf_{\bu}$ is $(\rmc_1,\gamma_1)$-\gradho{}, we only need to check that the gradient of $\rmf_{\bu}$ is $(\rmc_1,\gamma_1)$-\hold{} on the closure $\overline{\{\rmf_{\bu} > 0\}}$ of $\{\rmf_{\bu} > 0\}$.
For all $\bx,\by \in \overline{\{\rmf_{\bu} > 0\}}$, Equation~\eqref{e:bound-gradient} and Lemma~\ref{l:bump} yield
\begin{align*}
    \bno{ \nabla\rmf_{\bu}(\bx) - \nabla\rmf_{\bu}(\by) }_\iop
&
\le
    \frac{132 \e \rmd}{\eta^2} \lno{\bx - \by}_\iop
=
    \frac{132 \e \rmd}{\eta^2} \lno{ \bx - \by }_\iop^{1-\gamma_1} \lno{ \bx - \by }_\iop^{\gamma_1}
\\
&
\le
    \frac{132 \e \rmd}{\eta^2} (2\eta)^{1-\gamma_1} \lno{ \bx - \by }_\iop^{\gamma_1}
=
    \frac{132 \e \rmd 2^{1-\gamma_1}}{\eta^{1+\gamma_1}
    } \lno{ \bx - \by }_\iop^{\gamma_1}
    \;.
\end{align*}
Therefore, selecting 
$
    \eta = (132 \e \rmd 2^{1-\gamma_1} /\rmc_1)^{1/(1+\gamma_1)}
$ 
so that 
$
    132 \e \rmd 2^{1-\gamma_1} / \eta^{1+\gamma_1} = \rmc_1
$,
we obtain that $\rmf_{\bz}$ is $(\rmc_1,\gamma_1)$-\hold{} for all $\bz \in \rmZ$.
Moreover, by definition of $\rmZ$ and $\kappa$, we have that 
\[
    \labs{\rmZ} 
\ge 
    \lrb{ \frac{1}{2\eta} }^\rmd
=
    \lrb{ \frac{1}{2 (132 \e \rmd 2^{1-\gamma_1}/\rmc_1)^{1/(1+\gamma_1)} } }^\rmd
=
    \lrb{ \frac{\rmc_1}{528 \rmd} }^{\rmd/(1+\gamma_1)}
    \frac{1}{\e^{\rmd/(1+\gamma_1)}}
=
    \kappa \frac{1}{\e^{\rmd/(1+\gamma_1)}}\;.
\]
Thus, proceeding as in the proof of Proposition~\ref{p:lower-hold}, no deterministic algorithm can output a set that is an $\e$-approximation of the level set $\{\rmf=0\} = \dcube$ of the constant function $\rmf\equiv 0$ and simultaneously an $\e$-approximation of the level set $\{\rmf_{\bz}=0\}$ of all bump functions $\rmf_{\bz}$ ($\bz \in \rmZ)$, without querying at least one value in each one of the $\labs{\rmZ} \ge \kappa/\e^{\rmd/(1+\gamma_1)}$ disjoint sets $\{\rmf_{\bz}>0\}$ ($\bz \in \rmZ)$ when applied to $\rmf \equiv 0$.
\end{proof}

\section{THE BENEFITS OF ADDITIONAL STRUCTURAL ASSUMPTIONS}
\label{s:faster-rates-nls}

In this section, we present some examples showing how our general results can be applied to yield (slightly) improved sample complexity bounds when $\rmf$ satisfies additional structural assumptions, such as convexity.

\subsection{\nls{} Dimension}
\label{s:nls-dim}

In order to derive more readable bounds on the number of queries needed to return approximations of level sets, we now introduce a quantity that measures the difficulty of finding such approximations.

\begin{definition}[\nls{} dimension]
\label{d:nls-dim}
Fix any level $\rma\in \bbR$ and a function $\rmf\colon\dcube \to \bbR$. 
We say that $\ds \in [0,\rmd]$ is a \emph{\nls{} (or \nlss{}) dimension} of the level set $\fa$ if there exists $\Cs > 0$ such that (recalling Definition~\ref{d:packing-number} ---packing number)
\begin{equation}
    \label{e:nls-def}
    \forall \rmr \in (0,1) , \
    \cN \Brb{ \bcb{ \labs{ \rmf - \rma } \le \rmr }, \; \rmr }
\le
    \Cs \lrb{ \frac{1}{\rmr} }^{\ds}  \hspace{-9pt} \;.
\end{equation}
\end{definition}
\nls{} dimensions are a natural generalization of the well-known concept of near-optimality dimension, from the field of non-convex optimization (see, e.g., \cite[Section 2.3 and following discussion in Appendix B]{bouttier2020regret}).
The idea behind Inequality~(\ref{e:nls-def}) is that inflated level sets at, say, scale $\rmr\in(0,1)$, are hard to pinpoint if their complement $\bcb{ \labs{ \rmf - \rma } > \rmr }$ is large. 
Since for any increasing sequence $\rmr := \rmr_0 < \rmr_1 < \rmr_2 < \ld$, the set $\bcb{ \labs{ \rmf - \rma } > \rmr }$ of points at which $\rmf$ is more than $\rmr$-away from $\rma$ can be decomposed into a union of ``layers'' $\bcb{ \rmr_0 < \labs{ \rmf - \rma } \le \rmr_1 }, \bcb{ \rmr_1 < \labs{ \rmf - \rma } \le \rmr_2 }, \bcb{ \rmr_2 < \labs{ \rmf - \rma } \le \rmr_3 } , \ld$, and each of these layers $\bcb{ \rmr_{\rms-1} < \labs{ \rmf - \rma } \le \rmr_{\rms} }$ is by definition included in $\bcb{ \labs{ \rmf - \rma } \le \rmr_{\rms} }$, by controlling the size of each of these $\bcb{ \labs{ \rmf - \rma } \le \rmr_{\rms} }$ we can control the size of $\bcb{ \labs{ \rmf - \rma } > \rmr }$. 
Therefore, by controlling how large the inflated level sets $\bcb{ \labs{ \rmf - \rma } \le \rmr }$ can be at all scales $\rmr\in (0,1)$, the parameters $\Cs$ and $\ds$ quantify the difficulty of the level set approximation problem. 
In contrast, scales $\rmr \ge 1$ are not informative since in this case the packing number in \eqref{e:nls-def} is always $1$.
To see this, simply note that if $\rmr \ge 1$, no more than $1$ strictly $\rmr$-separated point can be packed in $\bcb{ \labs{\rmf - \rma} \le \rmr }$, which is included in $\dcube$, that has diameter $1$ (in the $\sup$-norm).

The dimension $\rmd$ of the domain is always a \nls{} dimension of any function $\rmf\colon \dcube \to \bbR$ (with $\Cs = 2^\rmd$;
we add a proof of this claim in Section~\ref{sec:lemmaPacking}, Lemma~\ref{l:trivial-nls}).
Hence, it is sufficient to consider \nls{} dimensions $\ds \le \rmd$, as we do in our Definition~\ref{d:nls-dim}.
While (as we will see in Section~\ref{s:convex-case}) $\ds$ is in general strictly smaller than $\rmd$, bounds expressed in terms of a \nls{} dimension should only be considered slight refinements of worst-case bounds expressed in terms of $\rmd$.
Indeed, the following result shows that, with the exceptions of sets of minimizers and maximizers, level sets $\fa$ of continuous functions $\rmf$ have \nls{} dimension at least $\rmd -1$.
\begin{theorem}[Theorem~\ref{thm:inherenthardness}]
\label{t:level-sets-are-big}
Let $\rmf\colon \dcube \to \bbR$ be a non-constant continuous function, and $\rma \in \bbR$ be any level such that $\min_{\bx \in \dcube} \rmf(\bx) < \rma < \max_{\bx \in \dcube}\rmf(\bx)$.
Then, there exists $\Cs>0$ such that, for all $\rmr>0$,
\[ 
    \cN \brb{ \fa, \; \rmr }
\ge
    \Cs \lrb{ \frac{1}{\rmr} }^{\rmd-1} \;.
\]
\end{theorem}
\begin{proof}
For all $\rmd_0 \in \bbN$, $\bz \in \bbR^{\rmd_0}$, and $\rho > 0$, we denote by $\rmB_{\rmd_0}(\bz,\rho)$ the closed $\rmd_0$-dimensional Euclidean ball $\{\bx \in \bbR^{\rmd_0} : \lno{ \bx - \bz }_2 \le \rho \}$ with center $\bz$ and radius $\rho$.
Since $\rma$ is neither the maximum nor the minimum of $\rmf$ and $\rmf$ is continuous, then the two sets $\{\rmf < \rma\}$ and $\{\rmf > \rma\}$ are non-empty and open.
Therefore, we claim that there exist two points $\bx \in \{\rmf < \rma\}$, $\by \in \{\rmf > \rma\}$, and a radius $\rho > 0$, such that $\rmB_\rmd(\bx,\rho) \s \{\rmf < \rma\} \cap (0,1)^\rmd$ and $\rmB_\rmd(\by,\rho) \s \{\rmf > \rma\} \cap (0,1)^\rmd$ (Figure~\ref{fig:sebs-pic}).
\begin{figure}
    \centering
    \begin{tikzpicture}
    \draw[blue] (0.7, 1) circle (0.5);
    \draw[blue] (0.7, 0.6) node[below] {$\rmB_{\rmd}(\bx,\rho)$};
    \draw[blue] (2.1, 3.3) circle (0.5);
    \draw[blue] (2.3, 3.5) node[above right] {$\rmB_{\rmd}(\by,\rho)$};
    \draw[red] plot[smooth] coordinates {(0,3) (1,2) (3,2.3) (4,1.1)};
    \draw[red] (3.9,1.1) node[right] {$\fa$};
    \foreach \x in {0,1/5,2/5,3/5,4/5,1}
    {
        \draw[gray] ({0.7 + (-(1-\x)*0.5 + \x*0.5)*0.854199}, 
                {1 + ((1-\x)*0.5 - \x*0.5)*0.519947})
            -- ({2.1 + (-(1-\x)*0.5 + \x*0.5)*0.854199}, 
                {3.3 + ((1-\x)*0.5 - \x*0.5)*0.519947});
    }
    \draw ({0.7 - 0.5*0.854199}, {1 + 0.5*0.519947}) 
        -- ({0.7 + 0.5*0.854199}, {1 - 0.5*0.519947});
    \draw ({2.1 - 0.5*0.854199}, {3.3 + 0.5*0.519947}) 
        -- ({2.1 + 0.5*0.854199}, {3.3 - 0.5*0.519947});
    \fill[purple] (0.7, 1) circle (0.5pt);
    \draw[purple] (0.8, 1.1) node[below left] {$\bx$};
    \fill[purple] (2.1, 3.3) circle (0.5pt);
    \draw[purple] (2, 3.2) node[above right] {$\by$};
    \draw[->] (0,-0.5) -- (0,4.5);
    \draw[->] (-0.5,0) -- (4.5,0);
    \draw (0,0) rectangle (4,4);
    \draw ({0.7 + 0.5*0.854199 - 0.1}, {1 - 0.5*0.519947 - 0.1}) node[right] {$\rmH_{\bx} \cap \rmB_{\rmd}(\bx,\rho)$};
    \draw ({2.1 + 0.5*0.854199 - 0.1}, {3.3 - 0.5*0.519947 - 0.1}) node[right] {$\rmH_{\bx} \cap $\colorbox{white}{$\rmB_{\rmd}(\by,\rho)$}};
    \draw (0,0) node[below left] {$0$};
    \draw (4,0) node[below] {$1$};
    \draw (0,4) node[left] {$1$};
    \end{tikzpicture}
    \caption{The ``dimension'' of the level set is at least the same as that of hyperplanes $\rmH_{\bx}$ and $\rmH_{\by}$.}
    \label{fig:sebs-pic}
\end{figure}
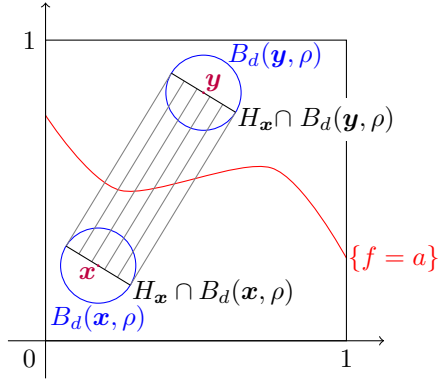
To see this, note that if $\rmf$ were identically equal to $\rma$ on $(0,1)^\rmd$, then, by continuity, $\rmf$ would be identically equal to $\rma$ on the whole $\dcube$, contradicting the assumption that it is non-constant.
Hence there exists an $\bx \in (0,1)^\rmd$ such that $\rmf(\bx)\neq \rma$.
Assume that $\rmf(\bx)<\rma$ (for the opposite case, proceed analogously). 
Then, being $\{\rmf<\rma\}\cap(0,1)^\rmd$ open, there exists a radius $\rho_1>0$ such that $\rmB_\rmd(\bx,\rho_1) \s \{\rmf<\rma\}\cap(0,1)^\rmd$.
Now, if $\rmf$ were lower than or equal to $\rma$ on $(0,1)^\rmd$, then, by continuity, $\rmf$ would be lower than or equal to $\rma$ on the whole $\dcube$ (and so would be its maximum), contradicting the assumption that $\rma < \max(\rmf)$.
Hence, there exists an $\by \in (0,1)^\rmd$ such that $\rmf(\by) > \rma$.
Then, being $\{\rmf > \rma\} \cap (0,1)^\rmd$ open, there exists a radius $\rho_2 > 0$ such that $\rmB_\rmd(\by,\rho_2) \s \{\rmf > \rma \} \cap (0,1)^\rmd$.
The claim is therefore proven by letting $\rho := \min(\rho_1,\rho_2)$.

Now, for all $\rmr \ge \rho/\sqrt{\rmd}$, we have
\[
    \cN\brb{ \fa, \; \rmr }
\ge
    1
\ge
    \lrb{ \frac{\rho}{\sqrt{\rmd}}}^{\rmd-1} \lrb{ \frac{1}{\rmr} }^{\rmd-1}
\]
and the result is proven with $\Cs = \brb{ \rho/\sqrt{\rmd} }^{\rmd-1}$.

Fix now an arbitrary $\rmr \in \brb{0, \rho/\sqrt{\rmd} }$. 
Consider the line $\cL := \bcb{ (1-\rmt) \bx + \rmt \by : \rmt \in \bbR }$ passing through $\bx$ and $\by$ and the two hyperplanes $\rmH_{\bx}$ and $\rmH_{\by}$ orthogonal to $\cL$ and passing through $\bx$ and $\by$ respectively.
We denote, for each $\bz \in \bbR^\rmd$ and $\rmE \s \bbR^\rmd$, the \mink{} sum $\{\bz + \bu : \bu \in \rmE\}$ of $\{\bz\}$ and $\rmE$ by $\bz + \rmE$.
Note that, by construction, $(\by-\bx)+\brb{ \rmH_{\bx} \cap \rmB_\rmd(\bx,\rho) } = \rmH_{\by} \cap \rmB_\rmd(\by,\rho)$, and there is a rigid transformation $\rmT\colon \bbR^\rmd \to \bbR^\rmd$ that maps $\rmH_{\bx} \cap \rmB_\rmd(\bx,\rho)$ into the $(\rmd-1)$-dimensional Euclidean ball $\rmB_{\rmd-1}(\bzero,\rho) =\bcb{\bz \in \bbR^{\rmd-1} : \lno{\bz}_2 \le \rho }$ of $\bbR^{\rmd-1}$ (where, with a slight abuse of notation, we identify from here on out $\bbR^{\rmd-1}$ with the subspace $\bcb{ (\rmz_1, \ld, \rmz_\rmd) \in \bbR^\rmd : \rmz_{\rmd}=0 }$ of $\bbR^\rmd$).
By the symmetry of the Euclidean balls, for all $\bz' \in \brb{ \rmH_{\bx} \cap \rmB_\rmd(\bx,\rho) }$ and $\rho'>0$, the transformed through the rigid transformation $\rmT$ of the intersection $\rmB_\rmd(\bz',\rho') \cap \brb{ \rmH_{\bx} \cap \rmB_\rmd(\bx,\rho) }$ of an arbitrary $\rmd$-dimensional Euclidean ball $\rmB_\rmd(\bz',\rho')$ centered at $\rmH_{\bx} \cap \rmB_\rmd(\bx,\rho)$ and $\rmH_{\bx} \cap \rmB_\rmd(\bx,\rho)$ itself is simply the intersection $\rmB_{\rmd-1}(\bz'',\rho') \cap \rmB_{\rmd-1}(\bzero,\rho)$ between the ball $\rmB_{\rmd-1}(\bzero,\rho)$ and a $(\rmd-1)$-dimensional ball $\rmB_{\rmd-1}(\bz'',\rho')$ with some center $\bz'' \in \rmB_{\rmd-1}(\bzero,\rho)$ and the same radius $\rho'$ of $\rmB_{\rmd}(\bz',\rho')$.
We recall that for any dimension $\rmd_0 \in \bbN$, norm $\lno{\cdot}$ on $\bbR^{\rmd_0}$, scale $\rmr_0 > 0$, and non-empty subset $\rmE_0$ of $\bbR^{\rmd_0}$, 
a set $\rmP \s \rmE_0$ is an $\rmr_0$-packing of $\rmE_0$ in $\bbR^{\rmd_0}$ with respect to $\lno{\cdot}$ if each two distinct points $\bz_1,\bz_2 \in \rmP$ satisfy $\lno{\bz_1 - \bz_2} > \rmr_0$, 
and a set $\rmC \s \rmE_0$ is an $\rmr_0$-covering of $\rmE_0$ in $\bbR^{\rmd_0}$ with respect to $\lno{\cdot}$ if for all $\bz \in \rmE_0$ there exists $\bc \in \rmC$ such that $\lno{\bz-\bc} \le \rmr_0$; we denote by $\cN_{\rmd_0, \lno{\cdot}}(\rmE_0, \rmr_0)$ the largest cardinality of an $\rmr_0$-packing of $\rmE_0$ in $\bbR^{\rmd_0}$ with respect to $\lno{\cdot}$, and by $\cM_{\rmd_0, \lno{\cdot}}(\rmE_0, \rmr_0)$ the smallest cardinality of an $\rmr_0$-covering of $\rmE_0$ in $\bbR^{\rmd_0}$ with respect to $\lno{\cdot}$.
By the previous observation, then, for all $\rmr_0>0$, a set is an $\rmr_0$-packing (resp., covering) of $\rmH_{\bx} \cap \rmB_\rmd(\bx,\rho)$ in $\bbR^{\rmd}$ with respect to the $\rmd$-dimensional Euclidean norm if and only if its transformed under $\rmT$ is an $\rmr_0$-packing (resp., covering) of $\rmB_{\rmd-1}(\bzero,\rho)$ in $\bbR^{\rmd-1}$ with respect to the $(\rmd-1)$-dimensional Euclidean norm.
Hence
\begin{align*}
& 
    \cN_{\rmd, \lno{\cdot}_2} \Brb{ \rmH_{\bx} \cap \rmB_\rmd(\bx,\rho) , \, \sqrt{\rmd} \, \rmr }
\\
& \hspace{16.9526pt} =
    \cN_{\rmd-1, \lno{\cdot}_2} \brb{ \rmB_{\rmd-1}(\bzero,\rho) , \, \sqrt{\rmd} \, \rmr }
\\
& \hspace{16.9526pt} \ge
    \cM_{\rmd-1, \lno{\cdot}_2} \brb{ \rmB_{\rmd-1}(\bzero,\rho) , \, \sqrt{\rmd} \, \rmr }
\ge
    \lrb{ \frac{\rho}{\sqrt{\rmd \, \rmr}} }^{\rmd-1} \;,
\end{align*}
where the first inequality follows from the fact that each packing that is maximal with respect to the inclusion is also a covering, and the second one is a known lower bound on the number of balls with the same radius that are needed to cover a ball with a bigger radius, expressed in terms of a ratio of volumes (see, e.g., \cite[Lemma~5.7]{wainwright2019high}).
Thus, we determined a $\brb{\sqrt{\rmd} \, \rmr}$-packing $\rmP$ of $\rmH_{\bx} \cap \rmB_\rmd(\bx,\rho)$ in $\bbR^\rmd$ with respect to the $\rmd$-dimensional Euclidean norm consisting of $\Cs (\nicefrac{1}{\rmr})^{\rmd-1}$ points, where again $\Cs := \brb{ \rho/\sqrt{\rmd} }^{\rmd-1}$.
For all $\bp \in \rmP$, consider the segment $[\bp, \bp + \by - \bx]$.
By construction, all these segments are parallel, with an endpoint in $\rmH_{\bx} \cap \rmB_\rmd(\bx,\rho) \s \{ \rmf < \rma \}$ and the other in $\rmH_{\by} \cap \rmB_\rmd(\by,\rho) \s \{ \rmf > \rma \}$.
Thus, the $\rmd$-dimensional Euclidean distance between any two points belonging to distinct segments is at least equal to the minimum distance between the corresponding lines, which is strictly greater than $\sqrt{\rmd} \, \rmr$ by construction.
By the continuity of $\rmf$, then, for each $\bp \in \rmP$ there exists a $\bp_\rma$ belonging to the segment $[\bp, \bp + \by - \bx]$ such that $\rmf(\bp_\rma) = \rma$ which, together with the previous remark, implies that the family $\rmP_\rma := \bigcup_{\bp \in \rmP} \bp_\rma$ obtained this way is a $\brb{\sqrt{\rmd} \rmr}$-packing of $\fa$ in $\bbR^\rmd$ with respect to the $\rmd$-dimensional Euclidean norm.
Since the two norms $\lno{\cdot}_\iop$ and $\lno{\cdot}_2$ on $\bbR^\rmd$ satisfy $\lno{\cdot}_\iop \ge \lno{\cdot}_2 / \sqrt{\rmd}$, then $\rmP_\rma$ is also an $\rmr$-packing of $\fa$ in $\bbR^\rmd$ with respect to the $\sup$-norm $\lno{\cdot}_\iop$, therefore its cardinality $\labs{\rmP_\rma} = \Cs (\nicefrac{1}{\rmr})^{\rmd-1}$ is smaller than or equal to the largest cardinality $\cN\brb{\fa, \rmr}$ of an $\rmr$-packing of $\fa$ with respect to the $\sup$-norm $\lno{\cdot}_\iop$.
This concludes the proof.
\end{proof}
We remark that our definition in Equation~\eqref{e:nls-def} could be refined by considering variable $\ds(\rmr)$ and $\Cs(\rmr)$ at different scales $\rmr$.
This would take into account that at different scales, the inflated level sets could have smaller size.
Notably, our general result (Theorem~\ref{t:general-packing}) would naturally adapt to this finer definition as they are stated in terms of packing numbers at decreasing scales.
For the sake of clarity, in this work we will stick to our worst-case definition of \nls{} dimension and we begin by showing how Theorem~\ref{t:general-packing} has an immediate corollary in terms of $\ds$.
Note that in the following results, our \bi{} instances are oblivious to the \nls{} dimension.
Also, recall from the comment before Theorem~\ref{t:level-sets-are-big} that typical level sets have \nls{} dimension $\ds \ge \rmd-1$.

\begin{corollary}[of Theorem~\ref{t:general-packing}]
\label{c:upp-bound-nls-dimension}
Consider a \bai{} algorithm (Algorithm~\ref{alg:bi}) run with input $\rma, \rmk, \rmb,\beta$.
Let $\rmf \colon \dcube \to \bbR$ be an arbitrary function with level set $\fa \neq \varnothing$ and let $\ds\in[0,\rmd]$ be a \nls{} dimension of $\fa$ (Definition~\ref{d:nls-dim}).
Assume that the \interp{}s $\rmg_{\rmC'}$ (defined at line~\ref{a:loc-interp}) are $(\rmb,\beta)$-accurate approximations of $\rmf$ (Definition~\ref{d:accurate-approx}), with $\beta \ge 1$.
Fix any accuracy $\e > 0$. 
Then, for all $\rmn > \rmm(\e)$, the output $\rmS_\rmn$ returned after the $\rmn$-th query is an $\e$-approximation of $\fa$, where
\[
    \rmm(\e)
:=
    \begin{cases}
        \displaystyle{ \kappa_1 + \kappa_2 \log_2 \lrb{ \frac{1}{\e^{1/\beta}} }^+ }
    &
        \text{if } \ds = 0 \;,\\
        \displaystyle{\kappa(\ds) \frac{1}{\e^{\ds/\beta}} }
    &
        \text{if } \ds > 0 \;,\\
    \end{cases}
\]
for $\kappa_1, \kappa_2, \kappa(\ds) \ge 0$ independent of $\e$, that depend exponentially on the dimension $\rmd$, where $\rmx^+ = \max\{\rmx,0\}$ for all $\rmx \in \bbR$.
\end{corollary}
\begin{proof}
Since all the conditions of Theorem~\ref{t:general-packing} are met by assumption, we have that for all $\rmn > \rmn(\e)$, the output $\rmS_\rmn$ returned after the $\rmn$-th query is an $\e$-approximation of $\fa$, where $\rmn(\e)$ is
\begin{equation}
    \label{e:good-stuff}
    4^\rmd \, \rmk \, \sum_{\rmi=0}^{\rmi(\e)-1} \lim_{\delta \to 1^-} \cN \Brb{ \bcb{ \labs{\rmf - \rma} \le 2 \, \rmb \, 2^{- \beta \rmi} }, \ \delta \, 2^{-\rmi} }
\end{equation}
and
$
    \rmi(\e) 
:= 
    \bce{ (\nicefrac{1}{\beta}) \log_2 (\nicefrac{2\rmb}{\e}) }
$.
If $\e \ge 2 \rmb$, than the sum in the definition of $\rmn(\e)$ ranges from $0$ to a \emph{negative} value, thus $\rmn(\e) = 0$ by definition of sum over an empty set and the result is true with $\kappa_1 = \kappa_2 = \kappa(\ds) = 0$. 
Assume then that $\e < 2 \rmb$ so that such sum is not trivially zero.
Being $\beta \ge 1$, we can further upper bound $\rmn(\e)$ by
\[
    4^\rmd \, \rmk \, \sum_{\rmi=0}^{\rmi(\e)-1} \lim_{\delta \to 1^-} \cN \Brb{ \bcb{ \labs{\rmf - \rma} \le 2 \rmb \, 2^{- \rmi} }, \ \delta \, 2^{-\rmi} }\;.
\]
By Lemma~\ref{lm:pack-r1-r2}, the packing number is at most
\[
    \lrb{ 1 + 4 \, \frac{2 \rmb}{\delta} \, \bbI_{2 \rmb > \delta} }^\rmd \cN \Brb{ \bcb{ \labs{\rmf - \rma} \le 2 \rmb \, 2^{- \rmi} }, \ 2 \rmb \, 2^{-\rmi} } \;.
\]
Taking the limit for $\delta\to 1^-$, the first term becomes 
$
    \brb{ 1 + 8\rmb \, \bbI_{2 \rmb \ge 1} }^\rmd
$, while our \nls{} assumption \eqref{e:nls-def} implies that the packing number is smaller than, or equal to
\begin{equation}
\label{e:necessary-evil}
    \bbI_{2\rmb 2^{-\rmi} \ge 1}
    +
    \Cs \lrb{ \frac{1}{2 \rmb \, 2^{-\rmi}} }^{\ds} \bbI_{2\rmb 2^{-\rmi} < 1}
\end{equation}
A direct computation shows that the sum over $\rmi$ of the first term in \eqref{e:necessary-evil} is
\begin{equation}
\label{e:annoying-term}
    \sum_{\rmi = 0}^{\rmi(\e) -1} \bbI_{2b2^{-\rmi} \ge 1}
\le
    \log_2(4b) \, \bbI_{2\rmb \ge 1} \;.
\end{equation}
For the sum over $\rmi$ of second term in \eqref{e:necessary-evil}, we upper bound the indicator function $\bbI_{2\rmb2^{-\rmi}<1}$ with $1$ for all $\rmi$ and study separately the two cases $\ds = 0$ and $\ds > 0$.
If $\ds = 0$, then, by definition of $\rmi(\e)$,
\[
    \sum_{\rmi = 0}^{\rmi(\e) -1} \brb{ 2^{\ds} }^i
= 
    \rmi(\e)
\le 
    \log_2 \lrb{ \frac{1}{\e^{1/\beta}} } + \log_2 \brb{ 2\, (2\rmb)^{1/\beta} } \;.
\]
Hence, the result follows by defining the additive and multiplicative terms $\kappa_1$ and $\kappa_2$, respectively, by
$
    \kappa'
    \, 
    \rmk \,
    \brb{
    \log_2(4\rmb) \, \bbI_{2\rmb \ge 1}
    +
    \Cs \log_2 \brb{ 2\, (2\rmb)^{1/\beta} }
    }
$
and 
$
    \kappa' \, \rmk \, \frac{ \Cs }{ (2 \rmb)^{\ds} }
$,
where $\kappa' := \brb{ 4 + 32 \, \rmb \, \bbI_{2 \rmb \ge 1} }^\rmd$.

If on the other hand, $\ds > 0$, recognizing the geometric sum below, we have, by definition of $\rmi(\e)$
\[
    \sum_{\rmi=0}^{\rmi(\e)-1}  \brb{ 2^{\ds} }^{\rmi}
= 
    \frac{\brb{2^{\ds}}^{\rmi(\e)}-1}{2^{\ds}-1}
\le 
    \frac{2^{\ds} \, (2\,\rmb)^{\ds/\beta} }{2^{\ds} - 1} \frac{1}{\e^{ \ds/\beta }}\;.
\]
Thus, if $2\rmb < 1$ or if simultaneously $2\rmb \ge 1$ and $\e \le 1/ \brb{ \log_2(4\rmb) }^{\beta/\ds}$ ---so that the term $\log_2(4\rmb) \, \bbI_{2\rmb \ge 1}$ in \eqref{e:annoying-term} can be upper bounded by $1/\e^{\ds/\beta} \, \bbI_{2\rmb \ge 1}$--- the result follows by defining $\kappa(\ds)$ as
\[
    \brb{ 4 + 32 \, \rmb \, \bbI_{2 \rmb \ge 1} }^\rmd \, \rmk \, \frac{ \Cs }{ (2 \rmb)^{\ds} }
    \lrb{
    \bbI_{2\rmb \ge 1}
    +
    \frac{2^{\ds} \, (2\,\rmb)^{\ds/\beta} }{2^{\ds} - 1}
    }\;.
\]
Finally, we consider the case in which $2 \rmb \ge 1$ and $\e > 1/ \brb{ \log_2(4\rmb) }^{\beta/\ds}$.
In this simpler instance, we upper bound $\rmn(\e)$ as in the proofs of Theorems~\ref{t:bah-rate}~and~\ref{t:biga-rate}.
Look back at Equation~\eqref{e:good-stuff}.
Upper-bounding, for any $\delta\in(0,1)$ and all $\rmi\ge 0$,
\[
    \cN \Brb{ \bcb{ \labs{\rmf - \rma} \le 2 \, \rmb \, 2^{- \beta \rmi} }, \ \delta \, 2^{-\rmi} }
\le
    \cN \brb{ \dcube, \ \delta \, 2^{-\rmi} }
\overset{(\dagger)}{\le}
    \brb{ 2^{\rmi}/\delta + 1 }^\rmd 
\le
    (\nicefrac{2}{\delta})^\rmd \, 2^{\rmd \rmi}
\]
(for completeness, we include a proof of the known upper bound $(\dagger)$ in Section~\ref{sec:lemmaPacking}, Lemma~\ref{l:pack-unit-cube}) and recognizing the geometric sum below, we have
\begin{align*}
&
    \rmn(\e)
\le
    8^\rmd 
    \, \rmk \sum_{\rmi=0}^{\lce{ (\nicefrac{1}{\beta}) \log_2 (\nicefrac{2\rmb}{\e}) }-1} 
    \brb{ 2^{\rmd} }^\rmi
\\
&
=
    8^\rmd \rmk
    \frac{2^{\rmd \lce{ (\nicefrac{1}{\beta}) \log_2 (\nicefrac{2\rmb}{\e}) }} -1 }{2^{\rmd} -1}
\le
    8^\rmd \rmk
    \frac{2^{\rmd \lrb{ (\nicefrac{1}{\beta}) \log_2 (\nicefrac{2\rmb}{\e}) +1 }} }{2^{\rmd} - \brb{ \nicefrac{2^\rmd}{2} }}
\\
&
=
    2 \, 8^\rmd \, \rmk \, ( 2\, \rmb )^{\rmd/\beta} \,
    \frac{ 1 }{ \e^{\rmd/\beta} } \;.
\end{align*}
Finally, using the assumption $\e > 1 / \brb{ \log_2(4\rmb) }^{\beta/\ds}$, we can upper bound the term $1/\e^{\rmd/\beta}$ with
\begin{align*}
    \frac{ 1 }{ \e^{\rmd/\beta} }
&
=
    \lrb{ \frac{ 1 }{ \e } }^{(\rmd - \ds)/\beta}
    \frac{ 1 }{ \e^{\ds/\beta} }
\\
&
<
    \lrb{ \brb{ \log_2(4\rmb) }^{\beta/\ds} }^{(\rmd - \ds)/\beta}
    \frac{ 1 }{ \e^{\ds/\beta} }
\\
&
=
    \brb{ \log_2(4\rmb) }^{(\rmd - \ds)/\ds}
    \frac{ 1 }{ \e^{\ds/\beta} } \;,
\end{align*}
and the results follows after defining the constant 
$
    \kappa(\ds)
:=
    2 \, 8^\rmd \, \rmk \, ( 2\, \rmb )^{\rmd/\beta} \, \brb{ \log_2(4\rmb) }^{(\rmd - \ds)/\ds} \;.
$
\end{proof}
The previous result has the following immediate consequence for \biga{} algorithms.
Recall from the comment before Theorem~\ref{t:level-sets-are-big} that typical level sets have \nls{} dimension $\ds \ge \rmd-1$.
\begin{corollary*}[Corollary~\ref{c:for-biga-main}]
Consider the \biga{} algorithm (Algorithm~\ref{alg:bi-g}) run with input $\rma,\rmc_1,\gamma_1$. 
Let $\rmf \colon \dcube \to \bbR$ be an arbitrary $(\rmc_1,\gamma_1)$-\gradho{} function with level set $\fa \neq \varnothing$ and let $\ds\in[0,\rmd]$ be a \nls{} dimension of $\fa$ (Definition~\ref{d:nls-dim}).
Fix any accuracy $\e > 0$.
Then, for all $\rmn > \rmm(\e)$, the output $\rmS_\rmn$ returned after the $\rmn$-th query is an $\e$-approximation of $\fa$, where
\[
    \rmm(\e)
:=
    \begin{cases}
        \displaystyle{     \kappa_1 + \kappa_2 \log_2 \lrb{ \frac{1}{\e^{1/(1+\gamma_1)}} } }
    &
        \text{if } \ds = 0 \;,\\[3ex]
        \displaystyle{ \kappa(\ds) \frac{1}{\e^{\ds/(1+\gamma_1)}} }
    &
        \text{if } \ds > 0 \;,\\
    \end{cases}
\]
for $\kappa_1, \kappa_2, \kappa(\ds) \ge 0$ independent of $\e$, that depend exponentially on the dimension $\rmd$.
\end{corollary*}
\begin{proof}
The result follows immediately from Corollary~\ref{c:upp-bound-nls-dimension} and Lemma~\ref{l:interp-biga}.
\end{proof}
The two previous corollaries suggest a general method for solving the level set approximation problem for a given class $\cF$, obtaining bounds that are slightly more refined than the worst-case ones that we saw in Sections~\ref{s:holder}~and~\ref{s:g-hold-biga}.
First, determine a family of \interp{}s that accurately approximate the functions in $\cF$.
Second, obtain for the resulting choice of \bi{} algorithm a sample complexity bound in terms of packing numbers of inflated level sets (as in Theorem~\ref{t:general-packing}).
Third, find a \nls{} dimension of an arbitrary $\rmf \in \cF$.
Importantly, both steps one and three of this process are decoupled from the task of determining approximations of level sets, and as such, they can be investigated independently.
For step two, we can simply plug in Theorem~\ref{t:general-packing}.

In the next section, we will discuss the notable convex case, in which the estimation of the \nls{} dimension is non-trivial. 
As it turns out, this also leads to a \nearo{} sample complexity for \bi{} algorithms.

\subsection{Upper Bound for Convex \gradho{} Functions}
\label{s:convex-case}

In this section we show a non trivial application of the theory presented so far. 
We will prove that our \biga{} algorithm is \nearo{} for approximating the level set of convex \gradlip{} functions.

For the sake of simplicity, we will focus on the approximation of what we call \emph{\proper} level sets (Definition~\ref{d:proper}, below).
Informally, a level set is \proper{} if it is non-empty, bounded away from the set of minimizers (where the problem collapses into a simpler, standard minimization problem) and it is not cropped by the boundary of $\dcube$.
\begin{definition}[Proper level sets]
\label{d:proper}
Fix any level $\rma\in \bbR$, a function $\rmf\colon\dcube\to\bbR$, and a margin $\Delta > 0$.
We say that $\fa$ is a $\Delta$-\emph{\proper} level set (for $\rmf$), if $\fa \neq \varnothing$ and
\[
    \min_{\bx \in \dcube} \rmf(\bx) + \Delta
\le
    \rma
\le
    \min_{\bx \in \partial \dcube} \rmf(\bx)\;,
\]
where we denoted by $\partial\dcube$ the boundary of $\dcube$.
When we need not explicitly refer to the margin $\Delta$, we simply say that $\fa$ is a \emph{proper} level set.
\end{definition}

In this section, we present an upper bound on the number of samples that our \biga{} algorithm needs in order to guarantee that its output is an approximation of the target level set of a convex \gradho{} function.
As we discussed in Section~\ref{s:faster-rates-nls}, now that we established a method on how to get these types of results, we only need to determine a \nls{} dimension $\ds$ (Definition~\ref{d:nls-dim}) of the level set of an arbitrary convex \gradho{} function.
The following results shows that $\ds = \rmd - 1$.

\begin{proposition} 
\label{prop:covering:level:set:convex}
Fix any level $\rma \in \bbR$, two \hold{} constants $\rmc>0,\gamma\in (0,1]$, and an arbitrary convex $(\rmc,\gamma)$-\hold{} function $\rmf \colon \dcube \to \bbR$ with proper level set $\fa$.
Then, there exists a constant $\Cs > 0$ such that
\[
    \forall \rmr \in (0,1) \;, \
    \cN \Brb{ \bcb{ \labs{ \rmf - \rma } \le \rmr }, \; \rmr }
\le
    \Cs \lrb{ \frac{1}{\rmr} }^{\rmd - 1}\;.
\]
\end{proposition}

\begin{proof}
Let $\Delta>0$ be a margin such that $\fa$ is $\Delta$-proper.
Fix any $\rmr \in (0,1)$.
If $\rmr > \Delta/2$, we can simply apply Lemma~\ref{l:trivial-nls} in Section~\ref{sec:lemmaPacking} and use the lower bound on $\rmr$ to obtain
\[
    \cN \Brb{ \bcb{ \labs{ \rmf - \rma } \le \rmr }, \; \rmr }
\le
    2^\rmd \, \lrb{ \frac{1}{r} }^\rmd
\le
    \frac{2^{\rmd+1}}{\Delta} \, \lrb{ \frac{1}{r} }^{\rmd-1}\;.
\]
Hence, without loss of generality, we can (and do) assume that $\rmr \in (0,\Delta/2)$. 
In the following, we denote by $\sdm$ the $(\rmd-1)$-dimensional unit sphere $\bcb{ \bx \in \bbR^\rmd  : \lno{ \bx }_2 \le 1 }$ with respect to the Euclidean norm $\lno{\cdot}_2$.
Let $\bxs$ be a minimizer of $\rmf$. 
Note that, being $\fa$ a proper level set (Definition~\ref{d:proper}), we have that
$
    \rmf(\bxs) < \rma \le \min_{\bx \in \partial \dcube} \rmf(\bx) \;,
$
therefore $\bxs$ belongs to the interior $(0,1)^\rmd$ of $\dcube$.
Now, for each $\bz \in \sdm$, let $\bp_{\bz}$ be the unique element of $\partial\dcube$ such that $(\bp_{\bz} - \bxs)/\lno{ \bp_{\bz} - \bx_m }_2 = \bz$ (Figure~\ref{fig:nls-convex-1}) 
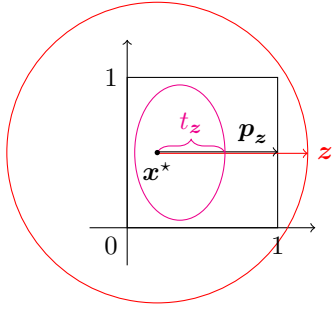
\begin{figure}
    \centering
    \begin{tikzpicture}
    \draw[->] (-0.5,1) -- (2.5,1);
    \draw[->] (0,0.5) -- (0,3.5);
    \draw (0,1) rectangle (2,3);
    \draw[red] (0.4,2) circle (2);
    \draw[red,->] (0.4,1.99) -- (2.4,1.99) node[right] {$\bz$};
    \draw[->] (0.4,2.01) -- (2,2.01) node[above left] {$\bp_{\bz}$};
    \draw[magenta,decorate,decoration={brace,amplitude=5pt}] (0.4,2) -- (1.3,2) node[midway,above,yshift=4pt] {$\rmt_{\bz}$};
    \fill (0.4,2) circle (1pt);
    \draw[magenta] (0.7,2) ellipse (0.6 and 0.9);
    \draw (0.4,2) node[below] {$\bxs$};
    \draw (0,1) node[below left] {$0$};
    \draw (2,1) node[below] {$1$};
    \draw (0,3) node[left] {$1$};
    \end{tikzpicture}
    \caption{In black, the unit \hyperc{} $\dcube$; in red, the unit sphere centered at the minimizer $\bxs$; in magenta, the level set $\fa$.}
    \label{fig:nls-convex-1}
\end{figure}
and define the convex univariate function
\begin{align*}
    \fz \colon \bsb{ 0, \lno{ \bxs - \bp_{\bz} }_2 } 
&
    \to \bbR
\\
    \rmt
&
    \mapsto \rmf( \bxs + \rmt \bz ) \;.
\end{align*}
Being $\fa$ a proper level set, for all $\bz \in \sdm$, the function $\fz$ satisfies
\[
    \min_{\bx \in \dcube} \rmf(\bx) 
= 
    \fz(0) 
< 
    a
\le
    \min_{ \bx \in \partial \dcube } \rmf(\bx)
\le
    \fz \brb{ \lno{ \bxs - \bp_{\bz} }_2  } \;.
\]
Thus, for each $\bz \in \sdm$, by the convexity and continuity of $\fz$, there exists a unique value $\rmt_{\bz} \in \bsb{ 0 , \lno{ \bxs - \bp_{\bz} }_2 }$ such that $\fz(\rmt_{\bz}) = \rma$ (Figure~\ref{fig:nls-convex-1}), which we use to define the following function on the unit sphere
\begin{align*}
    \rms \colon \sdm & \to \bbR \\
    \bz & \mapsto \rms(\bz) := \rmt_{\bz} \;.
\end{align*}
In words, $\rmt_{\bz}$ is the distance between the minimizer $\bxs$ and the level set $\fa$ in the direction of $\bz$.
We show now that $\rms$ is \lip{} with respect to the geodesic distance $\theta$ on $\sdm$, i.e., that there exists a constant $\ell > 0$ such that, for all $\bz_1,\bz_2 \in \sdm$,
\[
    \babs{ \rms(\bz_1) - \rms(\bz_2) }
\le
    \ell \, \theta(\bz_1, \bz_2) \;,
\]
where $\theta(\bz_1,\bz_2) = \arccos{ \brb{ \lan{ \bz_1, \bz_2 } }}$ is the angle between the two unit vectors $\bz_1,\bz_2$.
Fix two arbitrary $\bz_1 , \bz_2 \in \sdm$ with geodesic distance $\theta := \theta (\bz_1,\bz_2) \in (0,\pi]$. 
If $\theta \ge \pi/6$, we have
\[
    \frac{ \babs{ \rms(\bz_1) - \rms(\bz_2) } }{\theta}
\le
    \frac{6  \sqrt{\rmd}   }{ \pi }.
\]
Assume now that $\theta < \pi/6$. 
Consider the two-dimensional plane containing the triangle with vertices $\bxs$, $\bv_1 := \bxs + \rms(\bz_1) \bz_1$ and $\bv_2 := \bxs + \rms(\bz_2) \bz_2$ (note that the three points are not aligned). 
Let $\bv$ be the orthogonal projection of $\bxs$ on the line containing $\bv_1$ and $\bv_2$. 
Assume first that $\bv$ belongs to the segment $[\bv_1,\bv_2]$ (Figure~\ref{fig:nls-convex-2}, left). 
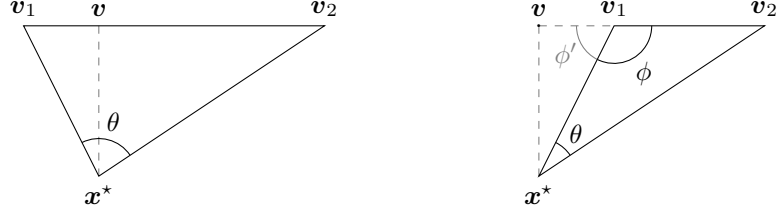
\begin{figure}
    \centering
    \begin{tikzpicture}
    \draw[gray, dashed] (0,2) -- (0,0);
    \draw (0,0) node[below] {$\bxs$} -- (-1,2) node[above] {$\bv_1$} -- (0,2) node[above] {$\bv$} -- (3,2) node[above] {$\bv_2$} -- cycle;
    \coordinate (xs) at (0,0);
    \coordinate (vone) at (-1,2);
    \coordinate (vtwo) at (3,2);
    \pic [draw, "$\theta$", angle eccentricity=1.5] {angle = vtwo--xs--vone};
    \end{tikzpicture}
    \hspace{2cm}
    \begin{tikzpicture}
    \draw[gray, dashed] (0,2) -- (0,0);
    \draw[gray, dashed] (0,2) -- (1,2);
    \draw (0,2) node {$.$} node[above] {$\bv$};
    \coordinate (xs) at (0,0);
    \coordinate (vone) at (1,2);
    \coordinate (vtwo) at (3,2);
    \coordinate (v) at (0,2);
    \pic[draw, "$\theta$", angle eccentricity=1.5] {angle = vtwo--xs--vone};
    \pic[draw, "$\phi$", angle eccentricity=1.5] {angle = xs--vone--vtwo};
    \pic[draw, gray, "$\phi'$", angle eccentricity=1.5] {angle = v--vone--xs};
    \draw (0,0) node[below] {$\bxs$} -- (1,2) node[above] {$\bv_1$} -- (3,2) node[above] {$\bv_2$} -- cycle;
    \end{tikzpicture}
    \caption{In the left (resp., right) picture, $\bv$ belongs (resp., does not belong) to the segment $[\bv_1,\bv_2]$.}
    \label{fig:nls-convex-2}
\end{figure}
Then, the function 
\begin{align*}
\rmg_{\bv_1,\bv_2} \colon \bbR & \to [0,+\iop) \\
\rmt & \mapsto \rmg_{\bv_1,\bv_2}(\rmt) := \Bno { \bxs - \brb{ \bv_1 + \rmt(\bv_2 - \bv_1) } }_2^2
\end{align*}
has its unique minimum at some $\ts \in [0,1]$. For all $\rmt \in \bbR$, we have
\begin{align*}
    \rmg_{\bv_1,\bv_2}(\rmt)
& =
    \bno{ (1-\rmt) (\bv_1 - \bxs) + \rmt (\bv_2 - \bxs)  }_2^2
\\
& 
= 
    (1-\rmt)^2 \rms(\bz_1)^2 + 
    \rmt^2 \rms(\bz_2)^2
    +2 \rmt (1-\rmt) \rms(\bz_1) \rms(\bz_2) \cos(\theta)
\\
& 
= 
    \rmt^2 \brb{ \rms(\bz_1)^2 + \rms(\bz_2)^2 - 2 \rms(\bz_1) \rms(\bz_2) \cos(\theta) }
    + \rmt \brb{ 2 \rms(\bz_1) \rms(\bz_2) \cos(\theta) - 2 \rms(\bz_1)^2 }
    + \rms(\bz_1)^2 \;.
\end{align*}
The derivative of this function is given, for all $\rmt \in \bbR$, by
\[
    \rmg_{\bv_1,\bv_2}'(\rmt)
=
    2\rmt \brb{ \rms(\bz_1)^2 + \rms(\bz_2)^2 - 2 \rms(\bz_1) \rms(\bz_2) \cos(\theta) }
    +
    \brb{ 2 \rms(\bz_1) \rms(\bz_2) \cos(\theta) - 2 \rms(\bz_1)^2 } \;.
\]
Hence we have
\[
    0 
\le
    \ts
=
    \frac{ 2 \rms(\bz_1)^2 - 2 \rms(\bz_1) \rms(\bz_2) \cos(\theta) }
    {
    2 \brb{ \rms(\bz_1)^2 + \rms(\bz_2)^2 - 2 \rms(\bz_1) \rms(\bz_2) \cos(\theta) }
    }.
\]
Since the above denominator is strictly positive, we obtain
\[
    2 \rms(\bz_1)^2 \ge 2 \rms(\bz_1) \rms(\bz_2) \cos(\theta) \;,
\]
thus, being $\rms(\bz_1)$ and $\rms(\bz_2)$ also strictly positive,
\[
    \cos(\theta) 
\leq
    1 - \frac{\rms(\bz_2) - \rms(\bz_1)}{\rms(\bz_2)}
\]
and in turn, since $\rms(\bz) \le \sqrt{\rmd}$ for all $\bz \in \sdm$,
\[
    \rms(\bz_2) - \rms(\bz_1)
\leq 
    \sqrt{\rmd}
    \brb{ 1 - \cos(\theta) } \;.
\]
Being $\theta > 0$, we have $1 -\cos(\theta) \le \theta$ and thus
\[
    \frac{\rms(\bz_2) - \rms(\bz_1)}{\theta}
\le 
    \sqrt{\rmd} \;.
\]
Swapping the roles of $\bv_1$ and $\bv_2$ (i.e., considering the function $\rmg_{\bv_2,\bv_1}$) 
we obtain similarly
\[
    \frac{\rms(\bz_1) - \rms(\bz_2)}{\theta}
\leq 
    \sqrt{\rmd} \;.
\]
Hence, when $\bv$ belongs to the segment $[\bv_1,\bv_2]$, we obtained
\[
    \frac{\babs{ \rms(\bz_1) - \rms(\bz_2) } }{\theta}
\le
    \sqrt{\rmd} \;.
\]
Consider now the last case where $\bv$ does not belong to the segment $[\bv_1,\bv_2]$ (Figure~\ref{fig:nls-convex-2}, right).
Without loss of generality, we can (and do) assume that $\rms(\bz_2) > \rms(\bz_1)$, and thus that $\bv$ is closer to $\bv_1$ than to $\bv_2$. 
By convexity of $\rmf$ on the line containing $\bv_1$ and $\bv_2$, we have $\rmf(\bv) \ge \rma$. 
Using the fact that the level set $\fa$ is $\Delta$-proper and the $(\rmc,\gamma)$-\hold{}ness of $\rmf$, we get
\[
    \Delta 
\le 
    \rma - \min_{\bx \in \dcube} \rmf(\bx) 
\le 
    \rmf(\bv) - \rmf(\bxs) 
\le
    \rmc \lno{ \bv - \bxs }_\iop^\gamma 
\le
    \rmc \lno{ \bv - \bxs }_2^\gamma \;,
\]
which in turn implies
\begin{equation} 
\label{eq:norm:v:minus:xm:smaller}
    \lno{ \bv - \bxs }_2 
\ge
    \lrb{ \frac{ \Delta }{ \rmc } }^{1/\gamma} \;.
\end{equation}
Let $\phi$ be the angle between $\bxs - \bv_1$ and $\bv_2 - \bv_1$. 
Applying the sine rule to the triangle $\bxs,\bv_1,\bv_2$, we obtain
\[
\frac{\sin(\theta)}{ \lno{ \bv_1 - \bv_2 }_2 }
=
\frac{\sin(\phi)}{ \rms(\bz_2)}
\] 
and thus
\begin{equation} 
\label{eq:sin:phi:eq}
    \sin(\phi) 
=  
    \frac{ \rms(\bz_2) \sin(\theta) }{ \lno{ \bv_1 - \bv_2 }_2 } \;.
\end{equation}
Let $\phi' = \pi - \phi$ be the angle between $\bxs- \bv_1$ and $\bv - \bv_1$. 
Note that the angle between $\bxs - \bv$ and $\bv_1 - \bv$ is $\pi/2$, being $\bv$ the orthogonal projection of $\bxs$ on the line containing $\bv_1$ and $\bv_2$. 
Hence, applying the sine rule to the triangle $\bxs , \bv_1 , \bv$ we obtain
\[
    \frac{\sin( \phi')}{ \lno{ \bv - \bxs }_2 }
= 
    \frac{\sin( \pi/2 )}{ \rms(\bz_1) }
\]
and thus
\begin{equation} 
\label{eq:norm:v:minus:xm:eq}
    \lno{ \bv - \bxs }_2
= 
    \rms(\bz_1) \sin (\phi) \;.
\end{equation}
From \eqref{eq:norm:v:minus:xm:smaller}, \eqref{eq:norm:v:minus:xm:eq}, and \eqref{eq:sin:phi:eq}, we obtain
\[
    \frac{\rms(\bz_1) \rms(\bz_2) \sin(\theta)}
    {\lno{ \bv_1 - \bv_2 }_2 } 
\geq 
    \lrb{ \frac{ \Delta }{ \rmc } }^{1/\gamma} \;.
\]
The triangle inequality yields
\[
    \babs{ \rms(\bz_1) - \rms(\bz_2) }
=
    \babs{ \bno{ \bv_1 - \bxs }_2  - \bno{ \bv_2  - \bxs }_2  } 
\le
    \bno{ (\bv_1 - \bxs) - (\bv_2  - \bxs) }_2 
= 
    \lno{ \bv_1 - \bv_2 }_2
\]
and thus
\[
    \babs{ \rms(\bz_1) - \rms(\bz_2) }
\le 
    \lrb{ \frac{ \rmc }{ \Delta } }^{1/\gamma}
    \rmd \sin(\theta)
\le 
    \lrb{ \frac{ \rmc }{ \Delta } }^{1/\gamma}
    \rmd \theta \;,
\]
where we used again $\rms(\bz) \le \rmd^{1/2}$ for any $\bz \in \sdm$. 
Putting everything together, we have shown that 
\begin{equation}
    \frac{ \babs{ \rms(\bz_1) - \rms(\bz_2) }}{\theta}
\le 
    \max
    \lrb{
    \frac{ 6 }{ \pi } \sqrt{\rmd}
    , \;
    \lrb{ \frac{ \rmc }{ \Delta } }^{1/\gamma} \rmd
    }
: = 
\ell\;, 
\end{equation}
for all $\theta \in (0,\pi]$, i.e., that $\rms$ is $\ell$-\lip{} on $\sdm$ with respect to the geodesic distance.

Consider now a covering of $\sdm$ with respect to the geodesic distance, with radius $\beta \rmr$, where $\beta \in (0, 1]$ will be selected later. 
This is a set of points $\bz_1,\ldots,\bz_{\rmn} \in \sdm$ such that the union of all the balls (with respect to the geodesic distance $\theta$) with radius $\beta \rmr$ centered at these points contains the whole $\sdm$.
We show now how such a covering can be taken using order of $1/\rmr^{\rmd-1}$ points.
Fix any two distinct $\bz_1,\bz_2 \in \sdm$ with geodesic distance $\theta(\bz_1,\bz_2) \in (0, \pi/2]$ and consider the isosceles triangle $\bz_1, \bzero, \bz_2$ with angles $\angle(\bz_1\,\bzero\,\bz_2)=\theta(\bz_1,\bz_2)$ and 
$
    \angle(\bzero \, \bz_2 \, \bz_1)
=
    \angle(\bz_2 \, \bz_1 \, \bzero)
=
    \brb{ \pi - \theta(\bz_1, \bz_2) }/2
=
    \pi/2 - \theta(\bz_1, \bz_2)/2
$ (Figure~\ref{fig:nls-convex-3}).
\begin{figure}
    \centering
    \begin{tikzpicture}
    \draw (0,0) node[below] {$\bzero$} -- (-1.5,2) node[above] {$\bz_1$} -- (1.5,2) node[above] {$\bz_2$} -- cycle;
    \coordinate (bzer) at (0,0);
    \coordinate (bzo) at (-1.5,2);
    \coordinate (bzt) at (1.5,2);
    \pic[draw, "$\theta(\bz_1{,} \bz_1)$", angle eccentricity=2] {angle = bzt--bzer--bzo};
    \end{tikzpicture}
    \caption{The isosceles triangle $\bz_1, \bzero, \bz_2$.}
    \label{fig:nls-convex-3}
\end{figure}
The sine rule yields
\[
    \frac{\lno{ \bz_1 - \bz_2 }_2}{\sin \brb{ \theta(\bz_1, \bz_2) }}
=
    \frac{1}{\cos \brb{ \theta(\bz_1, \bz_2)/2 } }
\]
or, equivalently stated,
\[
    \lno{ \bz_1 - \bz_2 }_2
=
    \frac{\sin \brb{ \theta(\bz_1, \bz_2) }}{\cos \brb{ \theta(\bz_1, \bz_2)/2 } } \;.
\]
Using the fact that $\sin(\rmx) \ge (2/\pi) \rmx$, for all $\rmx \in [0,\pi/2]$, the equality above gives
\[
    \lno{ \bz_1 - \bz_2 }_2
\ge
    \sin \brb{ \theta(\bz_1, \bz_2) }
\ge
    \frac{2}{\pi} \, \theta(\bz_1, \bz_2) \;.
\]
Therefore, if $\rmx \le \pi/2$, each ball with center $\bc$ radius $(2/\pi)\rho$ with respect to the Euclidean distance is included in the corresponding ball with center $\bc$ and radius $\rho$ with respect to the geodesic distance.
Thus, being $\rmr < 1 \le \pi/2$, in order to cover $\sdm$ with balls with radius $\beta \rmr$ with respect to the geodesic distance, it is enough to cover $\sdm$ with balls with radius $(2/\pi)\beta \rmr$ with respect to the Euclidean distance.
Moreover, since for any two points $\bx,\by \in \partial[-1,1]^\rmd$ on the boundary of the \hyperc{} $[-1,1]^\rmd$, their Euclidean distance $\lno{ \bx - \by }_2$ is larger than the Euclidean distance $\bno{ \bx/\lno{\bx}_2 - \by/\lno{\by}_2 }_2$ between their projections on the unit sphere, and since any point in the unit sphere can be reached this way, in order to cover the unit sphere with balls with radius $(2/\pi)\beta \rmr$ with respect to the Euclidean distance it is sufficient to cover the boundary $\partial[-1,1]^\rmd$ of $[-1,1]^\rmd$ with balls with radius $(2/\pi)\beta \rmr$ with respect to the Euclidean distance.
This is easy to do, as each one of the $2\rmd$ faces $\bcb{ [-1,1]\times \ldots \times [-1,1] \times  \{-1,1\} \times [-1,1] \times \ldots \times [-1,1] }$ of $\partial [-1,1]^\rmd$ can be covered with the same number of balls of radius $(2/\pi)\beta \rmr$ with respect to the $(\rmd-1)$-dimensional Euclidean distance that cover the \hyperc{} $[-1,1]^{\rmd-1}$.
This can be done, e.g., by taking a uniform grid of $(2/\pi)\beta \rmr$-spaced points.
Projecting these points onto $\sdm$ gives a covering $\bz_1, \ld, \bz_n$ of $\sdm$ with respect to the geodesic distance, with radius $\beta \rmr$, and with a number of points $\rmn$ that is at most
\begin{equation}
    \rmn
\le
    2 \rmd \lrb{ 1 + \lce{ \frac{\pi}{ 2 \beta \rmr} } }^{\rmd-1}
\le
    2 \rmd \lrb{ 2 + \frac{\pi}{ 2 \beta \rmr} }^{\rmd-1}
\le
    2 \rmd \lrb{ \frac{3}{ 2 }\pi }^{\rmd-1} \lrb{\frac{1}{\beta\rmr}}^{\rmd-1} \;.
\label{e:you-will-need-this-later}
\end{equation}
Fix this covering $\bz_1, \ld, \bz_\rmn$.
Fix also an arbitrary $\bx \in \bcb{  \labs{ \rmf - \rma| \le \rmr } }$.
Note that, being $\rmr \le \Delta/2$ and $\fa$ a $\Delta$-proper level set, then the minimizer $\bxs$ cannot belong to the set $\bcb{  \labs{ \rmf - \rma| \le \rmr } }$, hence $\bx \neq \bxs$.
Let $\bz =  (\bx-\bxs) / \lno{ \bx-\bxs}_2 $. 
Similarly as before, define for all $\rmt \in \bsb{ 0, \lno{ \bx - \bxs }_2 }$, the function $\rmf_{\bz}(\rmt) := \rmf(\bxs + \rmt \bz)$. 
Then $\rmf_{\bz}$ is convex, $\rmf_{\bz}(0) = \rmf(\bxs)$, $\rmf_{\bz} \brb{ \rms(\bz) } = \rma$ and $\babs{ \rmf_{\bz} \brb{ \lno{ \bx - \bxs }_2 } - a } \le \rmr$. 
If $\rmf_{\bz} \brb{ \lno{\bx - \bxs }_2 } < \rma$, by convexity, we have $\rms(\bz) > \lno{ \bx - \bxs }_2$, hence
\[
    \frac{\rma - \rmr - \rmf(\bxs)}{\lno{ \bx - \bxs }_2}
\le
    \frac{\rmf_{\bz}\brb{ \lno{ \bx - \bxs }_2 } - \rmf_{\bz}(0)}{\lno{ \bx - \bxs }_2 - 0}
\le
    \frac{ \rma - \rmf_{\bz} \brb{ \lno{\bx - \bxs }_2 } }{\rms(\bz) - \lno{ \bx - \bxs }_2 }
\le
    \frac{ \rmr }{\rms(\bz) - \lno{ \bx - \bxs }_2 }
\]
and recalling that $\rmr \le \Delta/2$ so that $ \rma - \rmr  - \rmf(\bxs) \ge \Delta/2 > 0 $, we have
\[
    \rms(\bz) - \lno{ \bx - \bxs }_2
\le
    \rmr
    \frac{\sqrt{\rmd}}
    { \rma - \rmr - \rmf(\bxs) }
\le
    \lrb{
    2
    \frac{\sqrt{\rmd}}
    { \Delta  }
    } \, \rmr \;,
\]
where we used $\lno{ x - x_m }_2 \le \sqrt{\rmd}$. 
If $\rmf_{\bz} \brb{  \lno{ \bx - \bxs }_2 } \ge \rma$, proceed similarly.
By convexity of $\rmf_{\bz}$ we have $\rms(\bz) \le \lno{ \bx - \bxs }_2$.
If $\rms(\bz) = \lno{ \bx - \bxs }_2$, then trivially $\lno{ \bx - \bxs }_2 - \rms(\bz) = 0 \le \brb{ 2 \sqrt{\rmd}/\Delta } \rmr$. 
If on the other hand, $\rms(\bz) < \lno{ \bx - \bxs }_2$, using the convexity of $\rmf_{\bz}$ once again, we get
\[
    \frac{\rma - \rmf_{\bz}(0)}{\rms(\bz) - 0}
\! \le \! 
    \frac{\rmf_{\bz}\brb{ \lno{ \bx - \bxs }_2 } - \rma}{\lno{ \bx - \bxs }_2 - \rms(\bz)}
\!  \le \! 
    \frac{ \rmr }{ \lno{ \bx - \bxs }_2 - \rms(\bz) }
\]
and using $\rma - \rmf(\bxs) \ge \Delta > 0 $ and $\rms(\bz) \le \sqrt{\rmd}$, yields
\[
    \lno{ \bx - \bxs }_2 - \rms(\bz)
\le
    \lrb{
    \frac{\sqrt{\rmd}}
    { \Delta  }
    } \, \rmr
\le
    \lrb{
    2
    \frac{\sqrt{\rmd}}
    { \Delta  }
    } \, \rmr \;.
\]
Thus we proved that
\begin{equation} \label{eq:norm:rz:minus:norm}
    \babs{ \rms(\bz) - \lno{ \bx - \bxs }_2 }
\le
    \lrb{
    2
    \frac{\sqrt{\rmd}}
    { \Delta  }
    } \, \rmr \;.
\end{equation}
Furthermore, there exists $\rmi \in \{ 1 , \ldots , \rmn\}$ such the geodesic distance of $\bz_\rmi$ and $\bz$ is smaller than or equal to $\beta \rmr$. 
Therefore we have, from \eqref{eq:norm:rz:minus:norm}, and the $\ell$-\lip{}ness of the function $\rmr$ with respect to the geodesic distance,
\[
    \babs{ \rms(\bz_\rmi) - \lno{ \bx - \bxs }_2 }
\le
    \babs{ \rms(\bz_\rmi) - \rms(\bz) } 
    + 
    \babs{ \rms(\bz) - \lno{ \bx - \bxs }_2 }
\le
    (\ell \beta) \, \rmr + \lrb{
    2
    \frac{\sqrt{\rmd}}
    { \Delta  }
    } \, \rmr \;.
\]
Hence, with $\gamma > 0$ to be chosen later, there exists 
\[
    \bx_\rmi' 
:= 
    \bxs + \rms(\bz_\rmi) \bz_\rmi + \rmk \gamma \rmr \bz_\rmi \;,   
\]
with $\rmk \in \bbZ$ such that
\[
    \labs{ \rmk } 
\le
    \frac{\ell \beta  +  
    \frac{2 \sqrt{\rmd}}
    { \Delta  }
    }{
    \gamma
    }
\]
and with 
\begin{equation} \label{eq:x:prime:i}
    \babs{ \lno{ \bx_\rmi'  - \bxs }_2 - \lno{ \bx - \bxs }_2 }
\le
    \gamma \rmr \;.
\end{equation}
This is obtained by covering the segment $[ -\ell \beta \rmr - 2 \rmr
\rmd^{1/2}/
 \Delta , \; \ell \beta \rmr + 2 \rmr 
\rmd^{1/2}/ \Delta     ]$ with points with equidistance $\gamma \rmr$. 
Then, we obtain
\begin{align*}
&
    \lno{ \bx  - \bx_\rmi' }_\iop
\le 
    \lno{ \bx  - \bx_\rmi' }_2
\\
&
    \hspace{8.63004pt}
= 
    \Bno{ \brb{ \bxs + \lno{ \bx - \bxs }_2 \, \bz }  - \brb{ \bxs + \lno{ \bx_\rmi' - \bxs }_2 \, \bz_\rmi}  }_2
\\
& 
    \hspace{8.63004pt}
=
    \bno{ \lno{ \bx - \bxs }_2 \, \bz - \lno{ \bx_\rmi' - \bxs }_2 \, \bz_\rmi }_2
\\
&
    \hspace{8.63004pt}
=
    \Bigl\lVert \lno{ \bx - \bxs }_2 ( \bz - \bz_\rmi) 
\\[-5pt]
&
    \hspace{80.58815pt} - \brb{ \lno{ \bx_\rmi' - \bxs }_2 - \lno{ \bx - \bxs }_2 } \bz_\rmi\Bigr\rVert_2
\\
& 
    \hspace{8.63004pt}
\le 
    \babs{ \lno{ \bx - \bxs }_2 -  \lno{ \bx_\rmi' - \bxs }_2 }
    +
    \sqrt{\rmd} \lno{ \bz - \bz_\rmi }_2
\\
&
    \hspace{8.63004pt}
\le
    \brb{ \gamma + \sqrt{\rmd} \beta } \rmr \;,
\end{align*}
from \eqref{eq:x:prime:i}. 
Hence, with 
\[
    \rmn' 
\le 
    2 \rmd \lrb{ \frac{3}{ 2 }\pi }^{\rmd-1} \lrb{\frac{1}{\beta\rmr}}^{\rmd-1}
    \lrb{
    1 +
    2
    \frac{\ell \beta  +  
    \frac{2 \sqrt{\rmd}}
    { \Delta }
    }{
    \gamma
    }
    }
\]
points, we have obtained a covering of $\bcb{ \labs{ \rmf - \rma } \le \rmr }$ with radius $\brb{ \gamma + \sqrt{\rmd} \beta } \rmr$ with respect to the $\sup$-norm $\lno{\cdot}_\iop$. 
Choosing $\beta := 1/\brb{ 4 \sqrt{\rmd} }$ and $\gamma := 1/4$ so that that $\brb{ \gamma + \sqrt{\rmd} \beta } \le 1/2$, we therefore determined a covering of $\bcb{ \labs{ \rmf - \rma } \le \rmr }$ with radius $\rmr/2$ with respect to the $\sup$-norm consisting of $\rmn'$ elements. 
Thus, $\rmn'$ is greater than or equal to the smallest cardinality $\cM \brb{ \bcb{ \labs{ \rmf - \rma } \le \rmr }, \; \rmr/2 }$ of a covering of $\bcb{ \labs{ \rmf - \rma } \le \rmr }$ with radius $\rmr/2$ with respect to the $\sup$-norm.
For a known result relating pickings and coverings (we recall it in~\eqref{eq:wainwright}, Section~\ref{sec:lemmaPacking}), we have
\[
    \cM \Brb{ \bcb{ \labs{ \rmf - \rma } \le \rmr }, \; \rmr/2 }
\ge
    \cN \Brb{ \bcb{ \labs{ \rmf - \rma } \le \rmr }, \; \rmr } \;,
\]
which concludes the proof.
\end{proof}

\begin{theorem}
\label{t:worst-case-upp-bound-gradho}
Consider the \biga{} algorithm (Algorithm~\ref{alg:bi-g}) run with input $\rma,\rmc_1,\gamma_1$. 
Let $\rmf \colon \dcube \to \bbR$ be an arbitrary convex $(\rmc_1,\gamma_1)$-\gradho{} function with proper level set $\fa$.
Fix any accuracy $\e > 0$.
Then, for all 
\[
    \rmn 
>
    \kappa \, \frac{1}{\e^{(\rmd-1)/(1+\gamma_1)}}
\]
the output $\rmS_\rmn$ returned after the $\rmn$-th query is an $\e$-approximation of $\fa$, where $\kappa>0$ is a constant independent of $\e$ that depends exponentially on the dimension $\rmd$.
\end{theorem}

\begin{proof}
Being $\rmf$ the restriction of a differentiable function defined on on an open set containing $\dcube$, it is \lip{} on the compact $\dcube$.
Thus we can apply Proposition~\ref{prop:covering:level:set:convex} to get a \nls{} dimension $\ds=\rmd-1$ for $\rmf$. 
The result then follows directly from Corollary~\ref{c:for-biga-main}. 
\end{proof}

\subsection{\Nearo{} Sample Complexity for Convex Gradient-\lip{} Functions}
\label{s:convex-lower-bound}

Theorem~\ref{t:worst-case-upp-bound-gradho} applied to the special case of \gradlip{} functions, states that the \biga{} algorithm (Algorithm~\ref{alg:bi-g}) needs order of $1/\e^{(\rmd-1)/2}$ queries to reliably output an $\e$-approximation of a \gradlip{} function.
The following theorem shows that this rate cannot be improved, i.e., that \biga{} is \nearo{} (Definition~\ref{d:near-opt}) for determining proper level sets of \gradlip{} functions.

\begin{theorem}
\label{t:lower-bound-convex}
Fix any level $\rma \in \bbR$ and an arbitrary accuracy $\e > 0$.
No deterministic algorithm $\rmA$ can guarantee to output an $\e$-approximation of any $\Delta$-proper level set $\fa$ of an arbitrary convex $\rmc_1$-\gradlip{} functions $\rmf$ with $\rmc_1\ge 3$ and $\Delta \in (0, \nicefrac{1}{4}]$, querying less than $\kappa / \e^{(\rmd-1)/2}$ of their values, where $\kappa > 0$ is a constant independent of $\e$.
This implies in particular that (recall Definition~\ref{n:smallest-n-queries}),
\[
    \inf_\rmA \sup_\rmf \fn (\rmf,\rmA,\e,\rma)
\ge
    \kappa \, \frac{1}{\e ^{(\rmd-1)/2}}\;,
\]
where the $\inf$ is over all deterministic algorithms $\rmA$ and the $\sup$ is over all $\rmc_1$-\gradlip{} functions $\rmf$ with $\Delta$-proper level set $\fa$, with $\rmc_1 \ge 3$ and $\Delta \in (0,\nicefrac{1}{4}]$.
\end{theorem}

\begin{proof}
We will prove the equivalent statement that no algorithm can output a $(\kappa'\e)$-approximation of any $\Delta$-proper level set $\fa$ of an arbitrary $\rmc_1$-\gradlip{} functions $\rmf$ with $\rmc_1\ge 3$ and $\Delta \in (0, \nicefrac{1}{4}]$, querying less than $1 / \e^{(\rmd-1)/2}$ of their values, where $\kappa' > 0$ is a constant independent of $\e$.

Let $\bo := ( \nicefrac{1}{2},\ldots,\nicefrac{1}{2} ) \in \dcube$,
$
    \oo
:= 
    ( 
    \nicefrac{1}{2} + \nicefrac{1}{(4\rmd)^{1/2}}, \ldots,
    \nicefrac{1}{2} + \nicefrac{1}{(4\rmd)^{1/2}}
    ) \in (0,1)^\rmd
$,
and
\begin{align*}
    \rmf_0 \colon \dcube & \to \bbR \\
    \bx & \mapsto \rmf_0(\bx) := \rma - \frac{1}{4} + \lno{\bx - \bo}_2^2\;.
\end{align*}
Then $\rmf_0$ is the restriction to $\dcube$ of the differentiable function $\bx\mapsto \rma -\frac{1}{4} + \lno{\bx - \bo}_2^2$ defined on $\bbR^\rmd$, and it satisfies, for all $\bx,\by \in \dcube$
\begin{align}
    \nonumber
    \bno{ \nabla\rmf_0(\bx) - \nabla\rmf_0(\by) }_\iop
&
=  
    \bno{ 2 (\bx - \bo) - 2 (\by - \bo) }_\iop
\\
&
=
    2 \, \lno{\bx-\by}_\iop
    \label{e:two-lip}
\;,
\end{align}
i.e., it is $2$-\gradlip{}.
Moreover, $\rmf_0$ has minimum equal to $\rma -  \nicefrac{1}{4}$ at $\bo$ and satisfies $\rmf_0(\oo) = \rma$.
Also, the minimum of $\rmf_0$ over $\partial\dcube$ is equal to $\rma + (1/2)^2 - 1/4 = \rma$. 
Hence $\{ \rmf_0 = \rma \}$ is a $\Delta$-proper level set, with $\Delta = \nicefrac{1}{4}$.

Consider an arbitrary deterministic algorithm $\rmA$ applied to the level set $\{ \rmf_0 = \rma \}$ of $\rmf_0$ and assume that only $\rmn < 1 /\e^{(\rmd-1)/2}$ values are queried before outputting a set $\rmS_\rmn$.

Let $\cS$ be the Euclidean sphere with center $\bo$ and radius $\lno{\bo-\oo}_2 =  \nicefrac{1}{2}$ (Figure~\ref{fig:geodesic-balls}).
\begin{figure}
    \centering
    \begin{tikzpicture}
    \draw[->] (-0.5,0) -- (4.5,0);
    \draw[->] (0,-0.5) -- (0,4.5);
    \draw (0,0) rectangle (4,4);
    \draw[blue] (2,2) circle (2);
    \draw [magenta, thick, domain = {2/sqrt(5)}:{4/sqrt(5)}] plot ({ 2+\x }, { 2+ sqrt(4 - \x*\x) });
    \fill ({sqrt(2) + 2}, {sqrt(2) + 2}) circle (0.5pt);
    \draw[red] (3.25, 4.5) -- (2,2) -- (4.5, 3.25);
    \draw ({sqrt(2) + 2}, {sqrt(2) + 2}) node[below left] {$\bx_0$};
    \draw (2,2) node[below] {$\bo$};
    \draw (0,0) node[below left] {$0$};
    \draw (4,0) node[below] {$1$};
    \draw (0,4) node[left] {$1$};
    \end{tikzpicture}
    \caption{The blue circle is the level set $\cS$; the geodesic ball on $\cS$ with center $\bx_0$ and radius $\kappa_1 \e^{1/2}$ is the arc in magenta and the corresponding cone is in red.}
    \label{fig:geodesic-balls}
\end{figure}
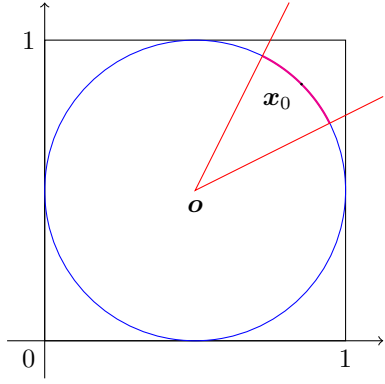
Note that $\{\rmf_0 = \rma\} = \cS$.
For any constant $\kappa_1 > 0$ and each point $\bx_0$ in $\cS$, consider the convex cone having origin $\bo$, and with intersection with $\cS$ equal to the geodesic ball on $\cS$ with center $\bx_0$ and radius $\kappa_1 \e^{1/2}$.
Then we can choose $\kappa_1$ (small enough) and $\bx_0$ such that this cone does not contain any points of $\rmf_0$ queried by the algorithm. 
Fix such a $\kappa_1$.
If $\rmS_\rmn$ does not contain $\bx_0$ then, since $\rmf_0(\bx_0) = \rma$, we have shown that $\{\rmf_0 = \rma\} \not \s \rmS_\rmn$, and the result follows.

Assume now that  $\bx_0 \in \rmS_\rmn$.
We will define a function $\rmf_1 \colon \dcube  \to \bbR$ such that the sum $\rmf_0 + \rmf_1$ is convex and $3$-\gradlip{}, the level set $\{\rmf_0 + \rmf_1 = \rma \}$ is $\Delta$-proper, and the algorithm applied to the level set $\{\rmf_0 + \rmf_1 = \rma \}$ of $\rmf_0 + \rmf_1$ does not return a $(\kappa'\e)$-approximation of $\{\rmf_0 + \rmf_1 = \rma \}$.
The idea is to carefully design a function $\rmf_1$ that is non-zero only on the cone that has not been explored by the algorithm.
This way, we can make $\rmf_0 + \rmf_1$ a perturbation of $\rmf_0$ that is not far enough from $\rmf_0$ so that the algorithm can distinguish the two, but it is different enough so that no $(\kappa'\e)$-approximation of $\{\rmf_0 = \rma\}$ can be a $(\kappa'\e)$-approximation of $\{\rmf_0 + \rmf_1 = \rma\}$.
The subtle part is that by construction, such an $\rmf_1$ is not convex, but the sum $\rmf_0 + \rmf_1$ has to retain the convexity of $\rmf_0$.

We begin by defining three non-negative auxiliary functions $\phi_1,\phi_2,\phi_3 \colon [0,1] \to \bbR$, for all $\rmt \in \dcube$, by
\[
    \phi_1(t)
:=
    \begin{cases}
        \rmt & \text{if } \rmt \in [0,\nicefrac{1}{4}] \\
        \nicefrac{1}{4} - (\rmt - \nicefrac{1}{4})  & \text{if } \rmt \in [\nicefrac{1}{4}, \nicefrac{3}{4}] \\
        -\nicefrac{1}{4} + (\rmt - \nicefrac{3}{4}) & \text{if } \rmt \in [\nicefrac{3}{4},1]
    \end{cases}\;,
\]
$
    \phi_2(\rmt) 
:= 
    \int_0^\rmt \mathrm{d}\rmx \int_{0}^\rmx \phi_1(\rmu) \dif \rmu
$,
and 
$
    \phi_3(\rmt)
:=
    \phi_2(1-\rmt)
$.
We remark that $\phi_2$ is twice differentiable with second derivative $\phi_1$. 
We see that $\phi_2(0) = 0$, $\phi_2'(0)=0$ and $\phi_2''(0) = 0$. 
We see that $\phi_2'$ is strictly positive on $[0,1]$. 
Hence, $\kappa_2 := \phi_2(1)>0$. 
We also see that $\phi_2'(1) = 0$ and $\phi_2''(1) = 0$.
Then $\phi_3$ is twice differentiable and non-negative on $[0,1]$ and satisfies $\phi_3''(0) = 0$, $\phi_3''(1) = 0 $, $\phi_3'(0) = 0$, $\phi_3'(1) = 0$, $\phi_3(0) =  \kappa_2 >0$ and $\phi_3(1) = 0$. 

We write $\rmB(\bx,\rmr)$ for the closed Euclidean ball with center $\bx$ and radius $\rmr$ intersected with $\dcube$.
We define the function
$\rmf_1 \colon \dcube \to \bbR$, for all $\bx \in \dcube$, by
\[
    f_1(\bx) := 
        \begin{cases}
            \displaystyle{
            \beta \e \, \phi_3 \lrb{
            \frac{ \lno{ \bx - \bx_0 }_2 }{ \kappa_3 \e^{1/2} }
            }
            }
            & \hspace{-5.40663pt} \text{if }\bx \in \rmB( \bx_0 , \kappa_3 \e^{1/2} )
            \\
            0 & \hspace{-5.40663pt} \text{otherwise} \;,
        \end{cases}
\]
with $\kappa_3, \beta > 0$ to be selected later. 
We can find $\kappa_3 > 0$ small enough such that $\rmB( \bx_0 , \kappa_3 \e^{1/2} )$ is included in the cone discussed above (recall Figure~\ref{fig:geodesic-balls}). 
Fix such a $\kappa_3$. 
Then, $\rmf_0$ and $\rmf_0 + \rmf_1$ differ only on this cone which is not explored by the algorithm. As a consequence, the algorithm applied to $\rmf_0 + \rmf_1$ returns the same set $\rmS_\rmn$, which contains $\bx_0$.  
Since $\rmf_0(\bx_0) + \rmf_1(\bx_0) = \rma + \beta \e \kappa_2$, the proof will be completed (letting $\kappa' := \beta \kappa_2 /2$) once we show that we can select $\beta > 0$, independently of $\e$, such that $\rmf_0 + \rmf_1$ is a convex $3$-\gradlip{} function with $\Delta$-proper level set $\{ \rmf_0 + \rmf_1 = \rma \}$, where $\Delta = \nicefrac{1}{4}$.

Because of the above discussed inclusion of the ball in the cone, we have $\rmf_1(\bo)  = 0$. 
Hence
\[
    \min_{\bx \in \dcube}
    \brb{ \rmf_0(\bx) + \rmf_1(\bx) }
\le
    \rmf_0(\bo) + \rmf_1(\bo)
= 
    \rma - \frac{1}{4}
\le
    \rma
=
    \min_{\bx \in \partial \dcube}
     \rmf_0(\bx) 
\le
    \min_{\bx \in \partial \dcube}
    \brb{ \rmf_0(\bx) + \rmf_1(\bx) }\;,
\]
which proves that the level set $\{ \rmf_0 + \rmf_1 = \rma \}$ is $\Delta$-proper, with $\Delta = \nicefrac{1}{4}$.

By definition of $\rmf_1$, its gradient is, for $\bx \in \dcube$,
\[
    \nabla f_1 (\bx) 
= 
    \beta \e \phi_3' \lrb{ \frac{ \lno{ \bx - \bx_0 }_2 }{ \kappa_3 \e^{1/2} } }
    \frac{1}{\kappa_3 \e^{1/2}} \frac{\bx - \bx_0}{ \lno{ \bx - \bx_0 }_2 }
\] 
if $\bx \in B( \bx_0 , \kappa_3 \e^{1/2} )$, $0$ otherwise.
We remark that in the above formula, by convention, $\nabla \rmf_1 (\bx_0) = 0$, which follows from the properties of $\phi_3$. 
%
Next, we observe that $\nabla \rmf_1$ satisfies
\[
    \sup_{\substack{\bu,\bv \in \dcube\\\bu \neq \bv }}
    \frac{ \bno{ \nabla \rmf_1 (\bu) - \nabla \rmf_1 (\bv) }_2 }
    { \lno{ \bu - \bv }_2 }
\le 
    \sup_{\substack{\bu,\bv \in \rmB( \bx_0 , \kappa_3 \e^{1/2} )\\\bu\neq\bv}}
    \frac{ \bno{ \nabla \rmf_1 (\bu) - \nabla f_1 (\bv) }_2 }
    { \lno{ \bu - \bv }_2 } \;,
\]
Indeed, for $\bu,\bv \notin \rmB( \bx_0 , \kappa_3 \e^{1/2} )$ the gradient difference is zero while for $\bu \in \rmB( \bx_0 , \kappa_3 \e^{1/2} )$ and $\bv \notin \rmB( \bx_0 , \kappa_3 \e^{1/2} )$ the gradient difference is equal to the difference between the gradient at $\bu$ and the gradient at the intersection of the segment $[u,v]$ and the boundary $\partial \rmB( \bx_0 , \kappa_3 \e^{1/2} )$. 
Hence,
\begin{align*}
&
    \sup_{\substack{\bu,\bv \in \dcube\\\bu \neq \bv }}
    \frac{ \bno{ \nabla \rmf_1 (\bu) - \nabla \rmf_1 (\bv) }_2 }
    { \lno{ \bu - \bv }_2 }
\\
& 
\hspace{22.93959pt}
\le
    \sup_{\substack{\bu,\bv \in \rmB( \bx_0 , \kappa_3 \e^{1/2} )\\\bu\neq\bv}}
    \frac{ 
    \beta \e
    }{
        \lno{ \bu - \bv }_2
    }
    \biggl\lVert
        \phi_3' \lrb{  \frac{ \lno{ \bu - \bx_0 }_2 }{ \kappa_2 \e^{1/2} } }
        \frac{1}{\kappa_3 \e^{1/2}} \frac{\bu - \bx_0}{ \lno{ \bu - \bx_0 }_2 }
    -
        \phi_3' \lrb{  \frac{ \lno{ \bv - \bx_0 }_2 }{ \kappa_2 \e^{1/2} } }
        \frac{1}{\kappa_3 \e^{1/2}} \frac{\bv - \bx_0}{ \lno{ \bv - \bx_0 }_2 }
        \biggr\rVert_2
\\
& 
\hspace{22.93959pt}
=
    \sup_{\substack{\bu,\bv \in \rmB( \bzero , 1 )\\\bu \neq \bv}}
    \frac{ 
    \frac{\beta}{\kappa_3^2}
    \lno{
    \phi_3' \brb{ \lno{\bu}_2 }
    \frac{ \bu }{ \lno{\bu}_2 }
    -
    \phi_3' \brb{ \lno{\bv}_2 }
    \frac{ \bv }{ \lno{\bv}_2 }
    }_2
    }{
    \lno{\bu - \bv}_2
}\;.
\end{align*}

Letting $\tilde{\rmf}_1 \colon \rmB(\bzero,1) \to \bbR$ be defined for all $\rmt \in \rmB(\bzero,1)$, by $\tilde{f}_1(\bx) = \phi_3 \brb{ \lno{\bx}_2 }$, we obtain
\[
    \sup_{\substack{\bu,\bv \in \dcube\\\bu \neq \bv }}
    \frac{ \bno{ \nabla \rmf_1 (\bu) - \nabla \rmf_1 (\bv) }_2 }
    { \lno{ \bu - \bv }_2 }
\le 
    \sup_{\substack{\bu,\bv \in \rmB(\bzero,1)\\\bu \neq \bv }}
    \frac{ 
    \frac{\beta}{\kappa_3^2}
    \bno{
    \nabla \tilde{\rmf}_1 (\bu) 
    -
    \nabla \tilde{\rmf}_1 (\bv) 
    }_2
    }{
    \lno{\bu - \bv}_2
    } \;.
\]
Since $\tilde{\rmf}_1$ is a fixed twice differentiable function which does not depend on $\e$, we can choose $\beta >0$ small enough, independently of $\e$, such that 
\[
    \sup_{\substack{\bu,\bv \in \dcube\\\bu \neq \bv }}
    \frac{ \bno{ \nabla \rmf_1 (\bu) - \nabla \rmf_1 (\bv) }_2 }
    { \lno{ \bu - \bv }_2 }
\le
    \frac{1}{\sqrt{\rmd}}\;.
\]
This implies that, for all $\bu, \bv \in \dcube$,
\begin{equation}
    \bno{ \nabla \rmf_1 (\bu) - \nabla \rmf_1 (\bv) }_\iop
\le
    \bno{ \nabla \rmf_1 (\bu) - \nabla \rmf_1 (\bv) }_2
\le
    \frac{1}{\sqrt{\rmd}} \lno{ \bu - \bv }_2
\le
    \lno{ \bu - \bv }_\iop \;.
    \label{e:one-lip}
\end{equation}
Thus, the two bounds~\eqref{e:two-lip}~and~\eqref{e:one-lip} yield
\[
    \sup_{\substack{\bu,\bv \in \dcube\\\bu \neq \bv }}
    \frac{  \bno{ \nabla(\rmf_0 + \rmf_1) (\bu) - \nabla(\rmf_0 + \rmf_1)(\bv) }_\iop }
    { \lno{ \bu - \bv }_\iop }
\le
    3 \;.
\]
Therefore, $\rmf_0 + \rmf_1$ is $3$-\gradlip{}. 
Finally, we have
\begin{align*}
&
\inf_{\substack{\bu,\bv \in \dcube\\\bu \neq \bv }}
    \frac{ \ban{ \nabla(\rmf_0 + \rmf_1) (\bu) - \nabla(\rmf_0 + \rmf_1)(\bv) , \frac{\bu - \bv}{ \lno{ \bu - \bv }_2 } } }
    { \lno{ \bu - \bv }_2 }
\\
& 
\hspace{49.31444pt}
\ge
    \inf_{\substack{\bu,\bv \in \dcube\\\bu \neq \bv }}
    \frac{ \ban{ \nabla\rmf_0 (\bu) - \nabla \rmf_0 (\bv) , \frac{\bu - \bv}{ \lno{ \bu - \bv }_2 } } }
    { \lno{ \bu - \bv }_2 }
    -
    \sup_{\substack{\bu,\bv \in \dcube\\\bu \neq \bv }}
    \frac{ \ban{ \nabla\rmf_1 (\bu) - \nabla \rmf_1 (\bv) , \frac{\bu - \bv}{ \lno{ \bu - \bv }_2 } } }
    { \lno{ \bu - \bv }_2 }
\ge 
    2 - \frac{1}{\sqrt{\rmd}}
\ge
    1\;.
\end{align*}
Hence $\rmf_0 + \rmf_1$ is $1$-strongly convex and thus it is convex. 
In conclusion, we have eventually selected a constant $\beta > 0$, independent of $\e$, such that $\rmf_0 + \rmf_1$ is a convex $3$-\gradlip{} function with $\Delta$-proper level set $\{ \rmf_0 + \rmf_1 = \rma \}$, but $\rmS_\rmn$ is not a $(\kappa'\e)$-approximation of $\{ \rmf_0 + \rmf_1 = \rma \}$.
This concludes the proof.
\end{proof}

We conclude this section by remarking the analogy between the problem of approximating the level set of a convex function and that of determining an approximation of a convex body in Hausdorff distance. 
The latter problem has been studied extensively in convex geometry.
Notably, while the scope of the and the techniques used in this field differ from ours, the sample complexity results for the two problems are similar.
For an overview of these results, we refer the reader to the two surveys \citep{kamenev2019optimal,gruber1993aspects}.

\end{document}